\documentclass[11pt]{article}

\pdfoutput=1

\usepackage[font=small,format=plain,labelfont=bf,up]{caption}

\usepackage[%
a4paper,nohead,ignorefoot,%
twoside,scale=0.85,bindingoffset=1.25cm
]{geometry}%
\usepackage{amssymb,amsfonts,amsmath,amsthm}
\usepackage{graphicx}
\usepackage{color}
\newcommand{\figwidth}{}%
\newcommand{\url}[1]{#1} 
\usepackage{hyperref}%
\definecolor{gray}{rgb}{0.2,0.2,.2}
\hypersetup{%
  unicode=true,          
  colorlinks=true,       
  linkcolor=gray,          
  citecolor=gray      
}
\definecolor{colorGreen}{rgb}{0.,0.67,0}
\definecolor{colorRed}{rgb}{0.67,0,0}
\definecolor{colorBlue}{rgb}{0.,0.,0.67}

\newcommand{\EMHC}{\color{black}}

\DeclareMathOperator{\laplace}{\Delta}

\newcommand{\fspace}[1]{{\mathsf{#1}}}
\newcommand{\fspaceL}{\fspace{L}}
\newcommand{\fspaceH}{\fspace{H}}
\newcommand{\fspaceC}{\fspace{C}}
\newcommand{\fspaceW}{\fspace{W}}

\newcommand{\Rset}{{\mathbb{R}}}
\newcommand{\Zset}{{\mathbb{Z}}}

\newcommand{\DO}[1]{{O\at{#1}}}

\newcommand{\nDO}[1]{{O\nat{#1}}}

\newcommand{\loc}{\mathrm{loc}}
\newcommand{\fin}{{\rm fin}}

\newcommand{\ini}{{\rm ini}}

\newcommand{\supp}{{\rm supp}}

\newlength{\mhpicDwidth}
\newlength{\mhpicDvsep}
\newlength{\mhpicDhsep}

\newlength{\mhpicPwidth}
\newlength{\mhpicPvsep}
\newlength{\mhpicPhsep}

\newlength{\mhpicWhsep}

\newcommand{\pair}[2]{{\left({#1},\,{#2}\right)}}

\newcommand{\bpair}[2]{{\big({#1},\,{#2}\big)}}

\newcommand{\at}[1]{{\left({#1}\right)}}
\newcommand{\nat}[1]{(#1)}
\newcommand{\bat}[1]{{\big(#1\big)}}
\newcommand{\Bat}[1]{{\Big(#1\Big)}}

\newcommand{\ato}[1]{{\left[{#1}\right]}}

\newcommand{\triple}[3]{{\left({#1},\,{#2},\,{#3}\right)}}

\newcommand{\D}{\displaystyle}

\newcommand{\bigpar}{\par\quad\newline\noindent}

\newcommand{\jump}[1]{{|\![#1]\!|}}

\newcommand{\norm}[1]{\|{#1}\|}
\newcommand{\abs}[1]{\left|{#1}\right|}
\newcommand{\babs}[1]{\big|{#1}\big|}
\newcommand{\nabs}[1]{|{#1}|}

\newcommand{\dint}[1]{\,\mathrm{d}#1}

\newcommand{\Om}{{\Omega}}
\newcommand{\al}{{\alpha}}

\newcommand{\ga}{{\gamma}}

\newcommand{\eps}{{\varepsilon}}

\newcommand{\si}{{\sigma}}

\newcommand{\calD}{\mathcal{D}}
\newcommand{\calE}{\mathcal{E}}

\newcommand{\calM}{\mathcal{M}}

\newcommand{\calR}{\mathcal{R}}

\newcommand{\bbR}{\mathbb{R}}

\newcommand{\bbZ}{\mathbb{Z}}

\DeclareMathOperator{\sgn}{sgn}

\newcommand{\set}[2][\empty]{\ensuremath{%
    \left\{%
      \ifx\empty#1%
      \relax%
      \else%
      #1:%
      \fi%
      #2%
    \right\}%
  }%
}

\newcommand{\pdiff}[2][\empty]{\ensuremath{%
        \ifx\empty#1%
        \frac{\partial}{\partial{#2}}%
        \else%
        \frac{\partial{#1}}{\partial{#2}}%
        \fi%
}}

\newcommand{\diff}[2][\empty]{\ensuremath{%
        \ifx\empty#1%
        \frac{\mathrm{d}}{\mathrm{d}{#2}}%
        \else%
        \frac{\mathrm{d}{#1}}{\mathrm{d}{#2}}%
        \fi%
}}

\newcommand{\tdiff}[2][\empty]{{\textstyle\diff[#1]{#2}}}

\makeatletter
\newcommand*\if@single[3]{%
  \setbox0\hbox{${\mathaccent"0362{#1}}^H$}%
  \setbox2\hbox{${\mathaccent"0362{\kern0pt#1}}^H$}%
  \ifdim\ht0=\ht2 #3\else #2\fi
  }
\newcommand*\rel@kern[1]{\kern#1\dimexpr\macc@kerna}
\newcommand*\widebar[1]{\@ifnextchar^{{\wide@bar{#1}{0}}}{\wide@bar{#1}{1}}}
\newcommand*\wide@bar[2]{\if@single{#1}{\wide@bar@{#1}{#2}{1}}{\wide@bar@{#1}{#2}{2}}}
\newcommand*\wide@bar@[3]{%
  \begingroup
  \def\mathaccent##1##2{%
    \if#32 \let\macc@nucleus\first@char \fi
    \setbox\z@\hbox{$\macc@style{\macc@nucleus}_{}$}%
    \setbox\tw@\hbox{$\macc@style{\macc@nucleus}{}_{}$}%
    \dimen@\wd\tw@
    \advance\dimen@-\wd\z@
    \divide\dimen@ 3
    \@tempdima\wd\tw@
    \advance\@tempdima-\scriptspace
    \divide\@tempdima 10
    \advance\dimen@-\@tempdima
    \ifdim\dimen@>\z@ \dimen@0pt\fi
    \rel@kern{0.6}\kern-\dimen@
    \if#31
      \overline{\rel@kern{-0.6}\kern\dimen@\macc@nucleus\rel@kern{0.4}\kern\dimen@}%
      \advance\dimen@0.4\dimexpr\macc@kerna
      \let\final@kern#2%
      \ifdim\dimen@<\z@ \let\final@kern1\fi
      \if\final@kern1 \kern-\dimen@\fi
    \else
      \overline{\rel@kern{-0.6}\kern\dimen@#1}%
    \fi
  }%
  \macc@depth\@ne
  \let\math@bgroup\@empty \let\math@egroup\macc@set@skewchar
  \mathsurround\z@ \frozen@everymath{\mathgroup\macc@group\relax}%
  \macc@set@skewchar\relax
  \let\mathaccentV\macc@nested@a
  \if#31
    \macc@nested@a\relax111{#1}%
  \else
    \def\gobble@till@marker##1\endmarker{}%
    \futurelet\first@char\gobble@till@marker#1\endmarker
    \ifcat\noexpand\first@char A\else
      \def\first@char{}%
    \fi
    \macc@nested@a\relax111{\first@char}%
  \fi
  \endgroup
}
\makeatother

\theoremstyle{plain}
\newtheorem{theorem}             {Theorem}[section]
\newtheorem{corollary}  [theorem]{Corollary}

\newtheorem{lemma}      [theorem]{Lemma}

\newtheorem*{result}{Main result}
\theoremstyle{definition}
\newtheorem{definition} [theorem]{Definition}
\newtheorem{assumption} [theorem]{Assumption}
\theoremstyle{remark}
\newtheorem{remark}     [theorem]{Remark}

\numberwithin{figure}{section}
\numberwithin{table}{section}
\numberwithin{equation}{section}
\usepackage{float}
\begin{document}


\title{Interface dynamics in discrete forward-backward diffusion equations}

\date{\today}

\author{%
  Michael Helmers\footnote{
    {\tt{helmers@iam.uni-bonn.de}},
    Institut f\"ur Angewandte Mathematik, Universit\"at Bonn.}
  \and
  Michael Herrmann\footnote{
    {\tt{michael.herrmann@math.uni-sb.de}},
    Fachrichtung Mathematik, Universit\"at des Saarlandes.}
}

\maketitle


\begin{abstract}
  We study the motion of phase interfaces in a diffusive lattice
  equation with bistable nonlinearity and derive a free boundary
  problem with hysteresis to describe the macroscopic evolution in the
  parabolic scaling limit.
  
  The first part of the paper deals with general bistable
  nonlinearities and is restricted to numerical experiments and
  heuristic arguments. We discuss the formation of macroscopic data
  and present numerical evidence for pinning, depinning, and
  annihilation of interfaces. Afterwards we identify a generalized
  Stefan condition along with a hysteretic flow rule that characterize
  the dynamics of both standing and moving interfaces.

  In the second part, we rigorously justify the limit dynamics for
  single-interface data and a special piecewise affine
  nonlinearity. We prove persistence of such data, derive upper bounds
  for the macroscopic interface speed, and show that the macroscopic
  limit can indeed be described by the free boundary problem. The
  fundamental ingredient to our proofs is a representation formula
  that links the solutions of the nonlinear lattice to the discrete
  heat kernel and enables us to derive macroscopic compactness results
  in the space of continuous functions.
\end{abstract}


\small

~\newline\noindent%
\begin{minipage}[t]{0.15\textwidth}%
  Keywords: 
\end{minipage}%
\begin{minipage}[t]{0.8\textwidth}%
  \emph{forward-backward diffusion in lattices,
    coarse graining for gradient flows,\\
    hysteretic models for phase transitions,
    pinning and depinning of interfaces,\\
    regularization of ill-posed parabolic PDEs}
\end{minipage}%
\medskip
\newline\noindent
\begin{minipage}[t]{0.15\textwidth}%
  MSC (2010): %
\end{minipage}%
\begin{minipage}[t]{0.8\textwidth}%
  34A33, 
  35R25, 
  37L60, 
  74N20, 
  74N30  
\end{minipage}%

\normalsize


\setcounter{tocdepth}{3}
\setcounter{secnumdepth}{3}{\small\tableofcontents}


\section{Introduction}

Discrete forward-backward diffusion equations appear in many different
applications such as edge-detection in digital images \cite{PeMa90},
models for population dynamics based upon random walks on lattices
\cite{HoPaOt04}, or phase transition problems with supercooling and
superheating \cite{Elliott85}. In all applications it is a common
major problem to understand how the backward-parabolic regions affect
the dynamics on large scales.
In this paper, we study the diffusion lattice
\begin{equation}
  \label{eq:lattice}
  \dot{u}_j\at{t}
  =
  \laplace \Phi^\prime\at{u_j\at{t}}
\end{equation}
for $j \in \bbZ$, $t \geq 0$, or equivalently
\begin{equation}
  \label{eq:lattice2}
  \dot{w}_j\at{t}
  =
  \nabla_- \Phi^\prime\at{\nabla_+w_j\at{t}}.
\end{equation}
Here $\nabla_-$, $\nabla_+$ are the left- and right-sided discrete
difference operators, $\laplace$ denotes the standard discrete
Laplacian $\laplace p_j= p_{j+1}-2p_j+p_{j-1}$, and $u_j$, $w_j$ are
connected via $u_j=\nabla_+ w_j$.
As illustrated in
Figure~\ref{fig:potential}, we always suppose that $\Phi^\prime$ the
derivative of a double-well potential $\Phi$, so it consists of two
stable branches that enclose an unstable one.
\begin{figure}[ht!]%
  \centering%
  \includegraphics[width=0.9\textwidth]{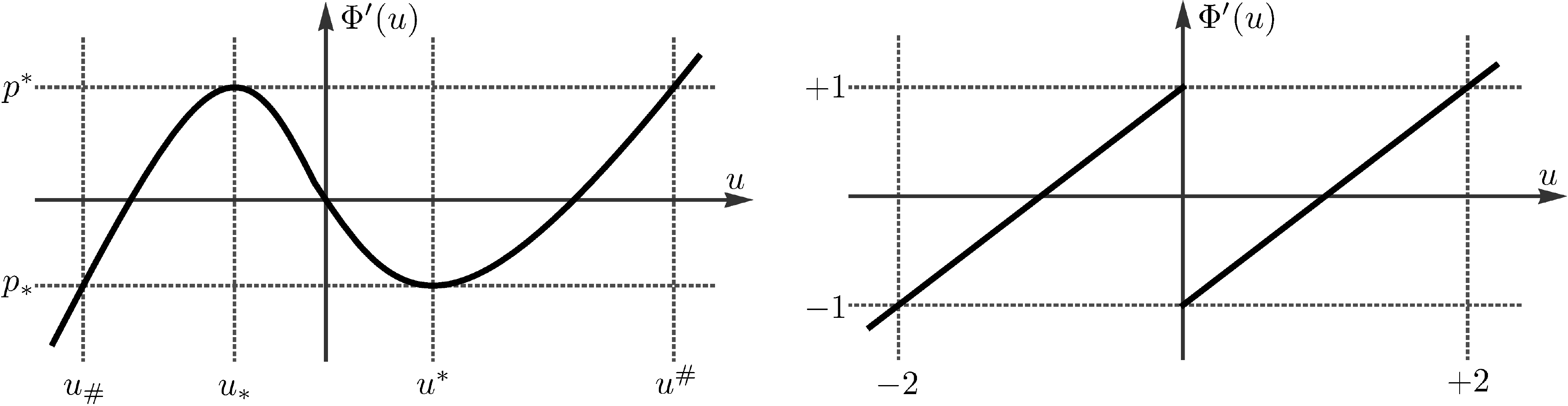}%
  \caption{\emph{Left.} Bistable derivative of a general double-well
    potential $\Phi$. The increasing and decreasing branches of
    $\Phi'$ are called \emph{stable} and \emph{unstable},
    respectively, while \emph{spinodal region} refers to the
    interval $[u_*,u^*]$ on which $\Phi^\prime$ is
    decreasing.  \emph{Right.} Piecewise affine derivative of the
    degenerate double-well potential that is studied in
    \S\ref{sect:ToyModel} and corresponds to $u_\#=-2$, $u_*=u^*=0$,
    $u^\#=+2$.}
  \label{fig:potential}%
\end{figure}%
\bigpar
Our goal is to characterize the effective dynamics of
\eqref{eq:lattice} in the parabolic scaling limit. Interpreting $t
\geq 0$ and $j\in\Zset$ as the microscopic variables, we introduce the
macroscopic time $\tau \geq 0$ and space $\xi\in\Rset$ by
\begin{equation}
  \label{eq:scaling}
  \tau=\eps^2 t,\qquad \xi = \eps{j},
\end{equation}
where $\eps>0$ is a small scaling parameter. We formally
identify
\begin{align}
  \label{eq:identification}
  u_j\at{t} = U\pair{\eps^2t}{\eps j}
\end{align}
and aim to describe the evolution of $U$ in the limit $\eps \to 0$.
The scaling \eqref{eq:scaling} and \eqref{eq:identification}
transforms the microscopic dynamics \eqref{eq:lattice} into
\begin{align}
  \label{eq:Scaled.Lattice}
  \partial_\tau U=\laplace_\eps \Phi^\prime(U),
\end{align} 
where $\laplace_\eps$ is the standard finite difference approximation
of $\partial_\xi^2$ on $\eps\Zset$, so the \emph{na\"{i}ve continuum
  limit} as $\eps \to 0$ reads
\begin{equation}
  \label{eq:general-eq}
  \partial_\tau U
  =
  \partial_\xi^2 \Phi'(U).
\end{equation}
This PDE, however, is ill-posed due to the unstable branches of
$\Phi^\prime$ and can therefore not determine the macroscopic limit of
\eqref{eq:lattice} completely.
Actually, the lattice can be viewed as a \emph{regularization} of
\eqref{eq:general-eq} that accounts for small scale effects and
provides in the limit $\eps\to0$ additional dynamical information such
as laws for the motion of phase interfaces or even measure-valued
solutions.
Other notable regularizations of \eqref{eq:general-eq} are the
Cahn-Hilliard equation
\begin{align}
  \label{eq:CahnHilliard}
  \partial_\tau U
  =
  \partial_\xi^2 \Phi'(U)-\eps^2\partial_\xi^4 U,
\end{align}
which has enjoyed a lot of attention over the last decades, and the
viscous approximation
\begin{equation}
  \label{eq:PDERegularization}
  \left( 1 - \eps^2 \partial^2_\xi \right) \partial_\tau U
  =
  \partial_\xi^2 \Phi'(U),
\end{equation}
studied in \cite{NoCoPe91,Plotnikov94,EvPo04}. As discussed below, the
available results indicate that the macroscopic limits of the scaled
lattice equation \eqref{eq:Scaled.Lattice} and the viscous
approximation \eqref{eq:PDERegularization} are identical but different
from the limit of the Cahn-Hilliard model.
In a formal way this can be understood by expanding the spatial
operators in powers of $\eps$: Equations \eqref{eq:lattice} and
\eqref{eq:PDERegularization} share (up to a redefinition of $\eps$)
the same leading order terms according to
\begin{align*}
  \laplace_\eps P
  =
  \Bat{\partial_\xi^2 +
    \frac{\eps^2}{12} \partial_\xi^4 +\DO{\partial_\xi^6 }} P
  \,,\qquad \at{1-\eps^2\partial_\xi}^{-1}\partial_\xi^2 P
  =\at{ \partial_\xi^2+\eps^2\partial_\xi^4 P+\DO{\partial_\xi^6}}P,
\end{align*}
with $P$ being shorthand for $\Phi^\prime(U)$, while
\eqref{eq:CahnHilliard} replaces the right hand side in
\eqref{eq:general-eq} by $\partial_\xi^2P-\eps^2\partial_\xi^4 U$.
For $\eps>0$, however, the rescaled lattice and the viscous
approximation are different and it remains open whether there exists a
unified theory that is capable of describing the limit $\eps\to0$ for
both models.

\bigpar
A key feature of any regularization of \eqref{eq:general-eq} with
double-well potential $\Phi$ are \emph{phase interfaces}, which evolve
according to certain jump conditions and separate regions where $U$
attains values in different stable regions (\emph{phases}).
Other types of nonconvex potentials give rise to different phenomena
such as coarsening of localized spikes, see for instance
\cite{EsGr09}.

Numerical simulations of \eqref{eq:lattice} as performed in
\S\ref{sec:Numerics} with a generic double-well potential provide
evidence for the existence of two different types of phase
interfaces. \emph{Type-I interfaces} correspond to piecewise smooth
functions $U$ and separate regions where $U$ takes values in either
one of the phases $U<u_*$ and $U>u^*$, where $[u_*,u^*]$
is the spinodal interval. \emph{Type-II interfaces}, however, are
related to measure-valued solutions of \eqref{eq:general-eq} and model
a phase mixture on at least one side of the interface. For both types,
a phase interface can have a fixed position or move, depending on the
behavior of $P=\Phi^\prime\at{U}$ near the interface.
While $P$ is smooth across a standing interface and takes values in
$[p_*,p^*] = [\Phi^\prime\at{u^*}, \Phi^\prime\at{u_*}]$, a moving
interface is driven by a jump in $\partial_\xi P$ but requires either
$P=p_*$ or $P=p^*$ subject to the propagation direction. In the
macroscopic limit we therefore find hysteretic behavior in the sense
that fronts moving into different phases comply with different
constraints, see Figure~\ref{fig:hysteresis}.
Further intriguing properties of the macroscopic lattice dynamics are
sketched in Figure~\ref{fig:interfaces}. Driven by the bulk diffusion,
a standing interface can suddenly start to move (\emph{depinning}) and
a moving interface can eventually come to rest
(\emph{pinning}). Moreover, two interfaces can disappear after a
collision (\emph{annihilation}).

\begin{figure}
  \centering
  \includegraphics[width=0.99\textwidth]{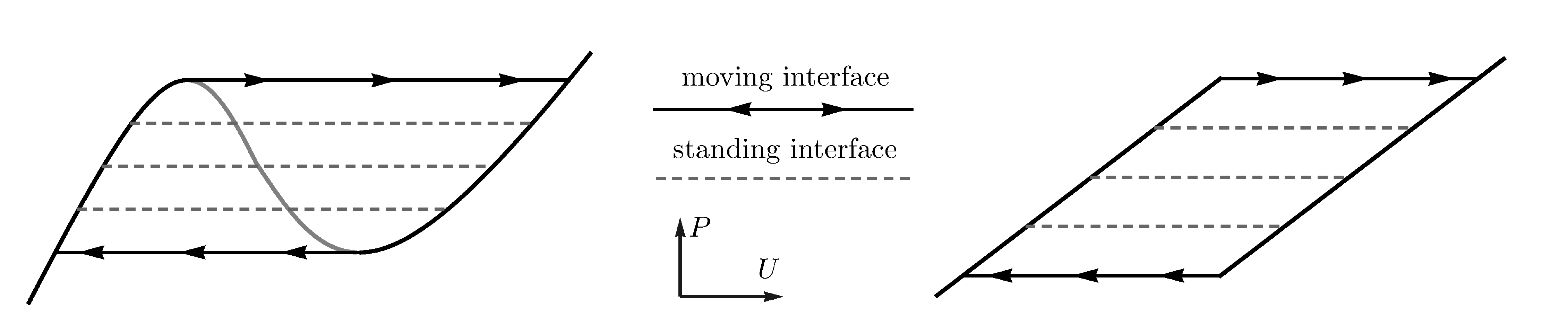}
  \caption{\emph{Left.} Cartoon of the macroscopic hysteresis for the
    potentials from Figure \ref{fig:potential}. The arrows indicate
    the temporal jump of $U$ when it undergoes a phase transition at a
    fixed position $\xi$. In particular, $P=p^*$ holds at any
    interface that moves into the phase $U<u_*$, whereas propagation
    into $U>u^*$ requires $P=p_*$. The dashed lines represent standing
    interfaces, at which $P$ takes values in $[p_*,p^*]$.
  }
  \label{fig:hysteresis}
  \bigskip
  \centering
  \includegraphics[width=0.7\textwidth]{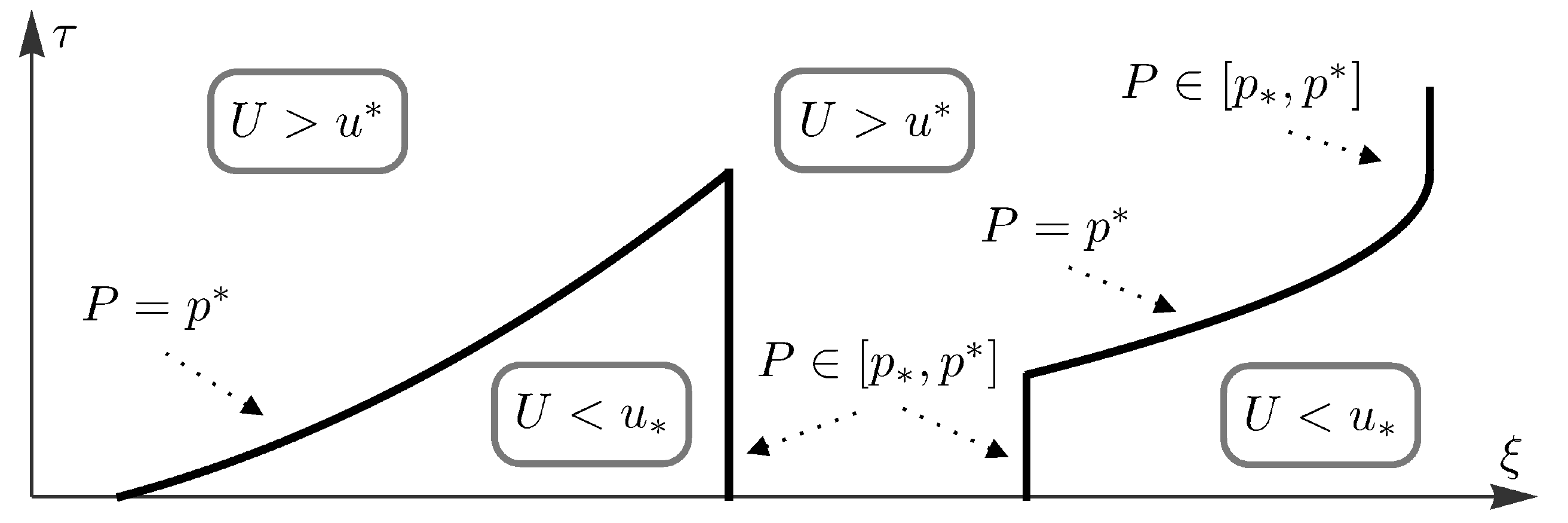}
  \caption{Cartoon of three macroscopic type-I interfaces.  The first
    interface (moving) and the second one (standing) eventually
    collide with each other and disappear (\emph{annihilation}).  The
    third one is initially at rest, starts to move at a later
    time(\emph{depinning}) and stops eventually again
    (\emph{pinning}).}
  \label{fig:interfaces}
\end{figure}

Assuming that the lattice data $p_j=\Phi^\prime\at{u_j}$ converge as
$\eps\to0$ to a sufficiently regular function $P$ and that any phase
interface is of type I, the dynamics in the parabolic scaling limit
can described by combining bulk diffusion via \eqref{eq:general-eq}
with the generalized Stefan condition
\begin{align}
  \label{eq:Intro.Stefan}
  \jump{P}=0,
  \qquad
  \frac{\dint \xi^*}{\dint\tau}\jump{U} + \jump{\partial_\xi{P}}
  =
  0
\end{align}
and the hysteretic flow rule
\begin{align}
  \label{eq:Intro.Flow}
  P = p_*\quad\text{for}\quad\frac{\dint
    \xi^*}{\dint\tau}\jump{U}>0,
  \qquad\qquad
  P = p^*\quad\text{for}\quad\frac{\dint \xi^*}{\dint\tau}\jump{U}<0,
\end{align}
where $\xi^*\at\tau$ is the position of an interface and
$\jump{\cdot}$ denotes the jump across this interface. These
conditions have also been proposed in \cite{EvPo04} to model the propagation
of type-I interfaces in the limit of the viscous approximation and are
naturally related to the notion of \emph{entropy solutions}, see also
\cite{MaTeTe09} and the discussion below.
For the Cahn-Hilliard equation \eqref{eq:CahnHilliard}, the parabolic
scaling limit does not imply any hysteresis. Here each interface
corresponds to $P=0$ and evolves therefore according to the classical
Stefan condition, see for instance \cite{BeBeMaNo12} for a rigorous
proof. We also note that there exists at least one other macroscopic
limit for \eqref{eq:CahnHilliard}, which is, however, not related to
the parabolic scaling \eqref{eq:scaling}: In the regime of almost
vanishing bulk diffusion, interfaces move and merge on a much slower
time scale which is exponentially small in $\eps$
\cite{AlBaFu91,BrHi92}.

In the case of hysteretic interface motion, there seems to be no
rigorous result -- neither for the lattice nor the viscous
approximation -- that derives \eqref{eq:Intro.Stefan} and
\eqref{eq:Intro.Flow} rigorously from the dynamics for $\eps>0$. 
Previous results for the lattices \eqref{eq:lattice} or
\eqref{eq:lattice2} are either restricted to standing interfaces, see
\cite{GeNo11} and \cite{BeGeNo13} for type-I and type-II interfaces,
respectively, or do not capture the dynamics of moving interfaces
completely, e.\,g. \cite{BeNoPa06}.

As a first step towards a mathematical justification of the
macroscopic evolution laws for type-I interfaces, we study in
\S\ref{sect:ToyModel} the special case of
\begin{align}
  \label{eq:Intro.ToyPotential}
  \Phi\at{u}=\tfrac12\min\big\{\at{u-1}^2,\,\at{u+1}^2\big\}.
\end{align}
At the cost of being discontinuous at $u=0$, the derivative
$\Phi^\prime$ of \eqref{eq:Intro.ToyPotential} has two advantages over
a generic bistable function. First, the nonlinearity in
\eqref{eq:lattice} is piecewise affine and second, the spinodal region
has shrunk to the point $u=0$, see the right panel of
Figure~\ref{fig:potential}. These properties simplify the dynamical
system \eqref{eq:lattice} significantly and enable us to represent
solutions to the nonlinear lattice by a summation formula that
involves delayed and shifted versions of the discrete heat kernel.

Due to the degenerate nature of \eqref{eq:Intro.ToyPotential}, it is
not our intention to identify the most general class of admissible
initial data for which the macroscopic limit can be described by a
free boundary problem. On the contrary, in order to keep the
presentation as simple as possible, we restrict our considerations in
\S\ref{sect:ToyModel} to initial data that produce a single type-I
interface which cannot change its direction of propagation. Our main
findings are formulated in Theorems
\ref{thm:existence-single-interface}, \ref{thm:existence}, and
\ref{thm:uniqueness}, and can informally be summarized as follows.

\begin{result}
  The lattice \eqref{eq:lattice} with \eqref{eq:Intro.ToyPotential}
  has the following properties.
  \begin{enumerate}
  \item \emph{Microscopic single-interface solutions}: Type-I
    interfaces are naturally related to a class of lattice states that
    is invariant under the dynamics. 
    Specifically, intervals
    of linear diffusion are interrupted by an increasing sequence of
    \emph{phase transition times} $\at{t^*_k}_{k\geq k_1}$ such that
    $u_k$ switches from negative to positive sign at $t=t^*_k$.
\item \emph{Macroscopic evolution:} When starting with macroscopic
    single-interface initial data, the limit $\eps\to0$ can be characterized as follows.
        
\begin{enumerate}
  \item \emph{Convergence}: The lattice data $u_j$ converge in
    a strong sense to a function $U$ that is smooth
    outside of an interface curve $\tau\mapsto\xi^*\at\tau$, where
    $\xi^*$ is Lipschitz continuous and nondecreasing. Moreover, the
    function $P=U-\sgn{U}$ is continuous across the interface.

  \item \emph{Limit dynamics}: For almost all times $\tau\geq0$ we
    have either
    \begin{align*}
      P\pair{\tau}{\xi^*\at\tau}\in[-1,+1]\quad
      \text{and}\quad \frac{\dint}{\dint\tau}\xi^*\at\tau=0
    \end{align*}
    or
    \begin{align*}
      P\pair{\tau}{\xi^*\at\tau}=+1\quad \text{and}\quad
      \frac{\dint}{\dint\tau}\xi^*\at\tau>0,
    \end{align*}
    and $P$ solves the linear heat equation outside of the interface.

  \item \emph{Uniqueness}: $U$ and $\xi^*$ are uniquely determined by
    the macroscopic initial data and the limit model.
  \end{enumerate}

  \end{enumerate}
\end{result}

The rest of the paper is organized as follows. In
\S\ref{sec:Gradient.Structure} we employ the gradient flow structure
of \eqref{eq:lattice} to describe the formation of macroscopic data on
a heuristic level.  Afterwards we report on our numerical
investigations for general double-well potentials.  We present several
examples for the macroscopic motion of type-I and type-II interfaces
in \S\ref{sec:Macroscopic.Interfaces}, and proceed in
\S\ref{sec:Microscopic.Interfaces} with discussing the key features of
the microscopic dynamics near moving interfaces.  These are:
sequentiality of phase transitions, small-scale fluctuations, and
existence of multiple time scales. Finally, in
\S\ref{sec:Limit.Models} we give a more detailed description of the
macroscopic limit model for type-I interfaces and interpret the
hysteretic flow rule in terms of entropy inequalities.

\S\ref{sect:ToyModel} contains our analytical results for the special
case \eqref{eq:Intro.ToyPotential}. We first employ ODE arguments in \S\ref{sec:exist-single-interf}
in order to
prove the persistence of single-interface data.  In particular, in Corollary
\ref{cor:decay-regular-velocities} we establish the aforementioned
representation formula. In \S\ref{sec:interface-speed} we then
introduce the concept of macroscopic single-interface initial data and
derive upper bounds for the macroscopic interface speed from the
properties of the discrete heat kernel, see also Appendix
\ref{sec:discrete-heat-kernel}. The main technical work is done in
\S\ref{sec:macr-cont-comp}, were we establish macroscopic compactness
results for $\xi^*$ and $P$ in the spaces of Lipschitz and H\"older
continuous functions, respectively. In \S\ref{sec:PassageToLimit} we
finally pass to the limit $\eps\to0$.  To this end, we first justify
the limit model along subsequences, and obtain afterwards both
uniqueness and convergence by adapting some arguments from the
theory of free boundary problems with hysteresis operators.


\section{Heuristic arguments and numerical simulations}
\label{sec:Numerics}

In this section we employ heuristic arguments as well as numerical
simulations in order to gain a qualitative understanding of the key
dynamical features of the nonlinear lattice diffusion with double-well
potential. In particular, we discuss $\at{i}$ the underlying gradient
flow structure and the formation of macroscopic data during a fast
initial transient regime, $\at{ii}$ the microscopic and macroscopic
dynamics of phase interfaces, and $\at{iii}$ the macroscopic evolution
equations in the limit $\eps\to0$.
\par
To keep the presentation as simple as possible, we use a finite
dimensional lattice with $j=-N,\ldots,N$ and close the resulting
ODE system by imposing homogeneous Neumann boundary conditions 
\begin{align}
\label{eq:num.bound.cond}
u_{-N-1}\at{t}=u_{-N}\at{t}\,,\qquad u_{N+1}\at{t}=u_{N}\at{t}.
\end{align}
The natural scaling parameter on such a finite lattice is $\eps=1/N$, that
means the macroscopic space variable $\xi$ takes values in the
interval $[-1,1]$. The numerical simulations presented
below are computed by the explicit Euler scheme, where the time step
size is chosen sufficiently small so that the energy, see
\eqref{eq:energy1} below, is strictly decreasing. Moreover, all
simulations are performed with
\begin{align*}
\Phi\at{u}=\frac{2\at{1-u^2}^2}{1+u^2},\qquad
\Phi^\prime\at{u}=4u-\frac{16u}{\at{1+u^2}^2},
\end{align*}
which is convenient for numerical computations since the linear growth
of $\Phi^\prime$ for $\abs{u}\to\infty$ allows to use a relatively
large time-step size.
%
%
\subsection{Gradient flow structure and onset of macroscopic data}
\label{sec:Gradient.Structure}
%
%
%
The lattice diffusion \eqref{eq:lattice} can be regarded as a discrete analogue to the $\fspaceH^{-1}$-gradient flow
of a nonlinear bulk energy. More precisely, defining the energy
\begin{align}
\label{eq:energy1}
\calE\at{u}:=\eps\sum_{j=-N}^{N} \Phi\at{u_j}
\end{align}
and the metric potential 
\begin{align*}
\calR\at{\dot{u}}=\frac{\eps}{2} \sum_{j=-N}^{N} \bat{\nabla_+ v_j}^2\quad\text{with}\quad -\laplace v_j = \dot{u}_j
\quad\text{for}\quad j=-N,\ldots,N
\quad \text{and}\quad v_{\pm \at{N+1}}=v_{\pm N}
\end{align*}
we readily verify -- using discrete integration by parts along with the boundary conditions \eqref{eq:num.bound.cond} -- that \eqref{eq:lattice} is equivalent to
\begin{align*}
\partial_{\dot{u}}\calR\at{\dot{u}}+\partial_u \calE\at{u}=0,
\end{align*}
where the metric tensor $\partial_{\dot{u}}\calR$ is formally given by $\at{-\laplace}^{-1}$.
In particular,  we obtain
the energy balance 
\begin{align*}
\frac{\dint\calE}{\dint{t}}=\eps^2\frac{\dint\calE}{\dint{\tau}}=-\eps^2\calD,\qquad \calD\at{u}:=\eps^{-1}\sum_{j=-N}^N \bat{\nabla_+\Phi^\prime\at{u_j}}^2,
\end{align*}
where the dissipation $\calD$ gives the squared and rescaled length of the energy gradient with respect to the metric induced by $\calR$. Notice that $\calE$, $\calR$, and $\calD$ are scaled macroscopically, that means 
the identification \eqref{eq:identification} implies
\begin{align}
\label{eq:LD.Formal.Limit1}
\calE\at{U}=\int_{-1}^{+1} \Phi\at{U}\dint\xi,\qquad
\calD\at{U}=\int_{-1}^{+1}\at{\partial_\xi \Phi^\prime\at{U}}^2\dint\xi
\end{align}
as well as
\begin{align}
\label{eq:LD.Formal.Limit2}
\calR\at{\partial_\tau U}=\tfrac12\int_{-1}^{+1}\at{\partial_\xi V}^2\dint\xi\quad\text{with}\quad
-\partial_{\xi}^2V=\partial_\tau{U}\quad\text{and}\quad \partial_\xi{V}|_{\xi=\pm1}=0
\end{align}
provided that $U$ is sufficiently smooth with respect to $\tau$ and
$\xi$. The formal gradient flow corresponding to
\eqref{eq:LD.Formal.Limit1} and \eqref{eq:LD.Formal.Limit2}, however,
is the ill-posed PDE \eqref{eq:general-eq} and can hence not govern
the limit dynamics.
\par
A further important observation is that the nonlinear lattice
\eqref{eq:lattice} admits a comparison principle on the increasing
branches of $\Phi^\prime$. Specifically, using standard arguments for
ODEs we easily show
\begin{align*}
\sup\limits_{t\geq0} \sup\limits_{\abs{j}\leq{N}}{u_j\at{t}}\leq \max\big\{u^\#,\sup\limits_{\abs{j}\leq{N}}u_j\at{0}\big\},\qquad
\inf\limits_{t\geq0} \inf\limits_{\abs{j}\leq{N}}{u_j\at{t}}\geq \min\big\{u_\#,\inf\limits_{\abs{j}\leq{N}}u_j\at{0}\big\}.
\end{align*}    
This implies $\calE\at{t}=\DO{\eps{N}}=\DO{1}$ for all $t\geq0$ and
hence $\int_0^\infty\calD\at{u\at{t}}\dint{t}=\nDO{1}$ provided that
the initial data $u_j\at{0}$ are bounded independently of $\eps$.

\renewcommand{\figwidth}{0.4\textwidth}
\begin{figure}[t!]%
\centering{%
\begin{tabular}{ccc}%
\includegraphics[width=\figwidth]{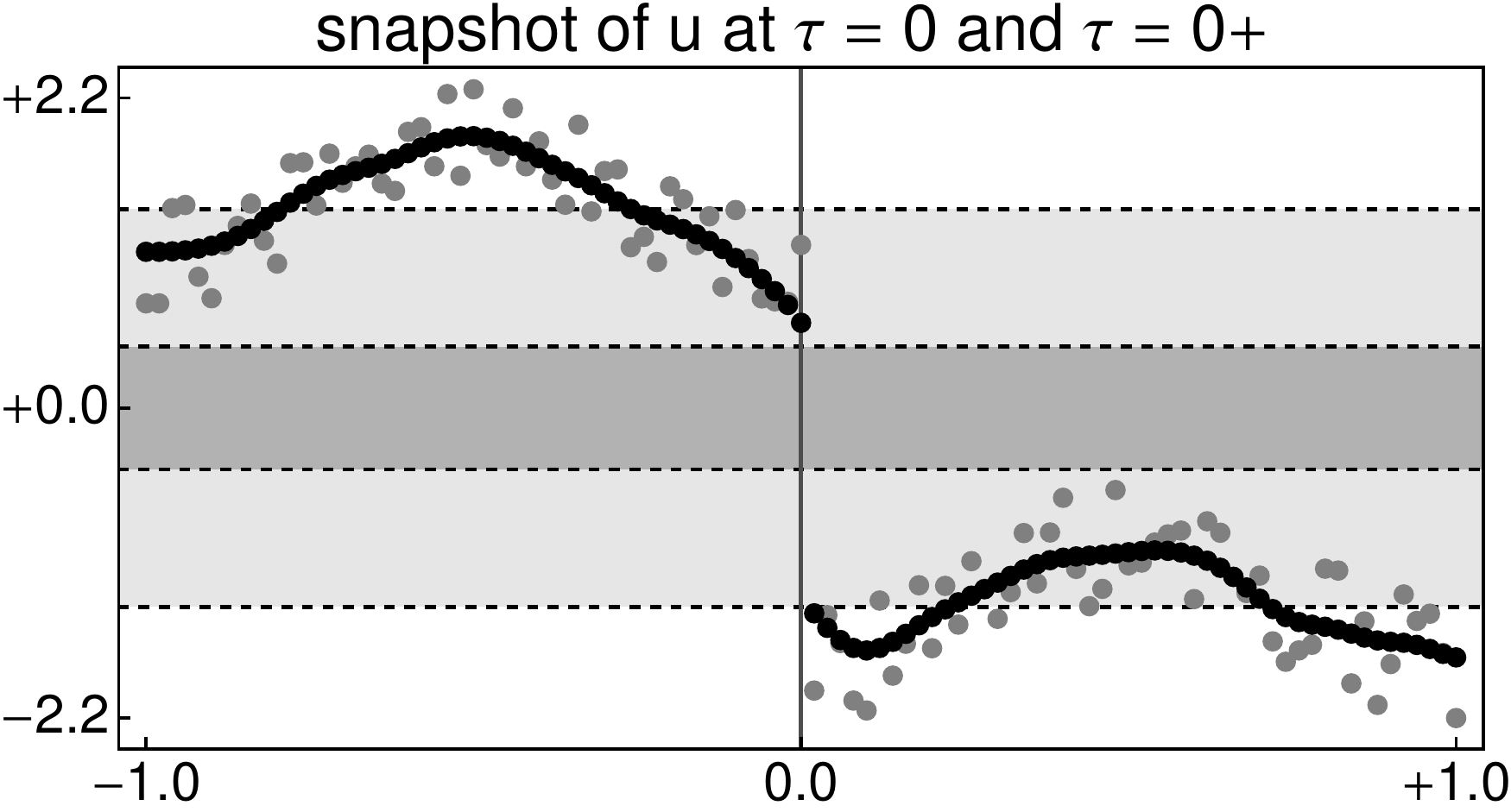}%
&\quad&%
\includegraphics[width=\figwidth]{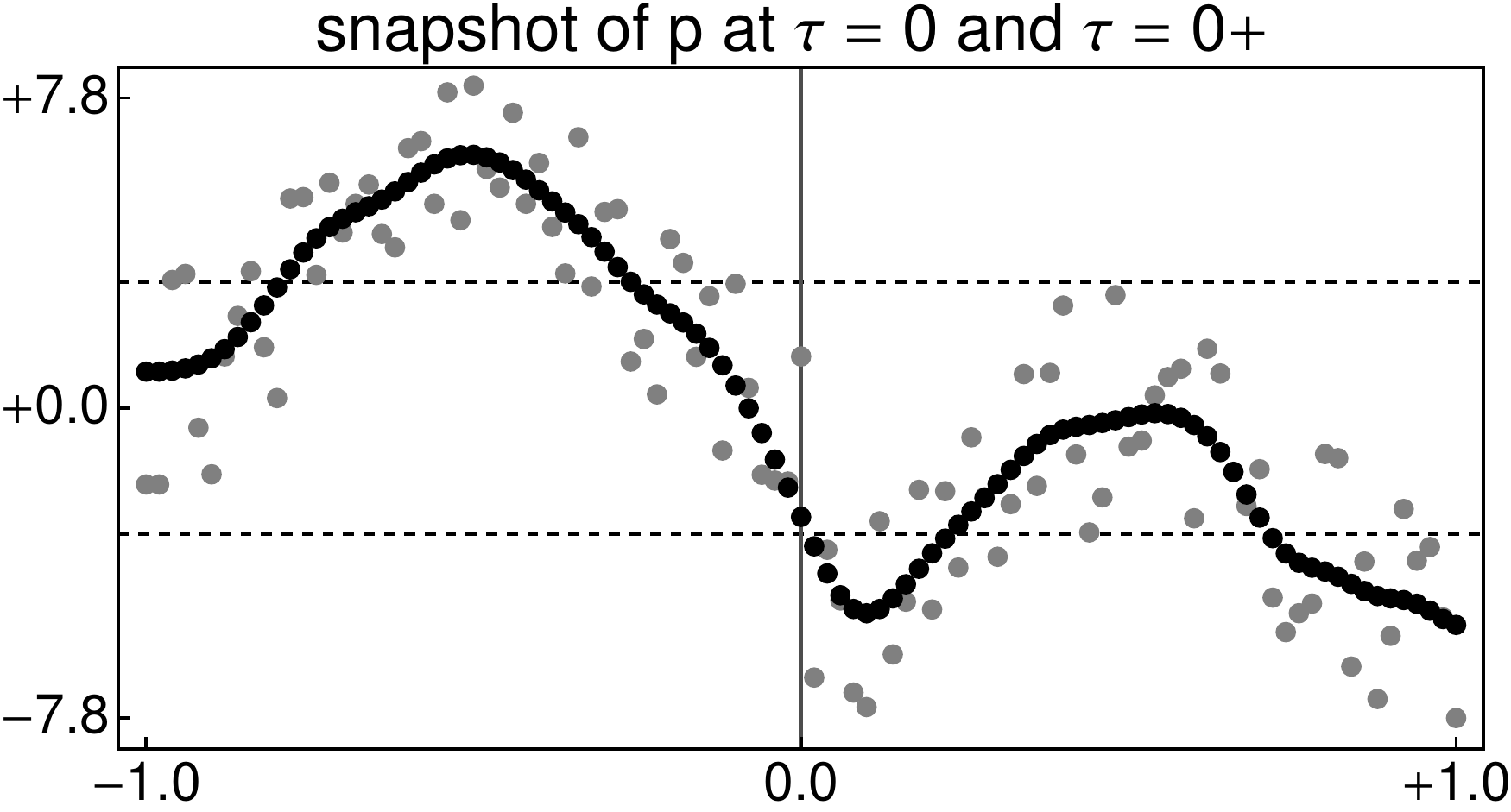}%
\\%
\includegraphics[width=\figwidth]{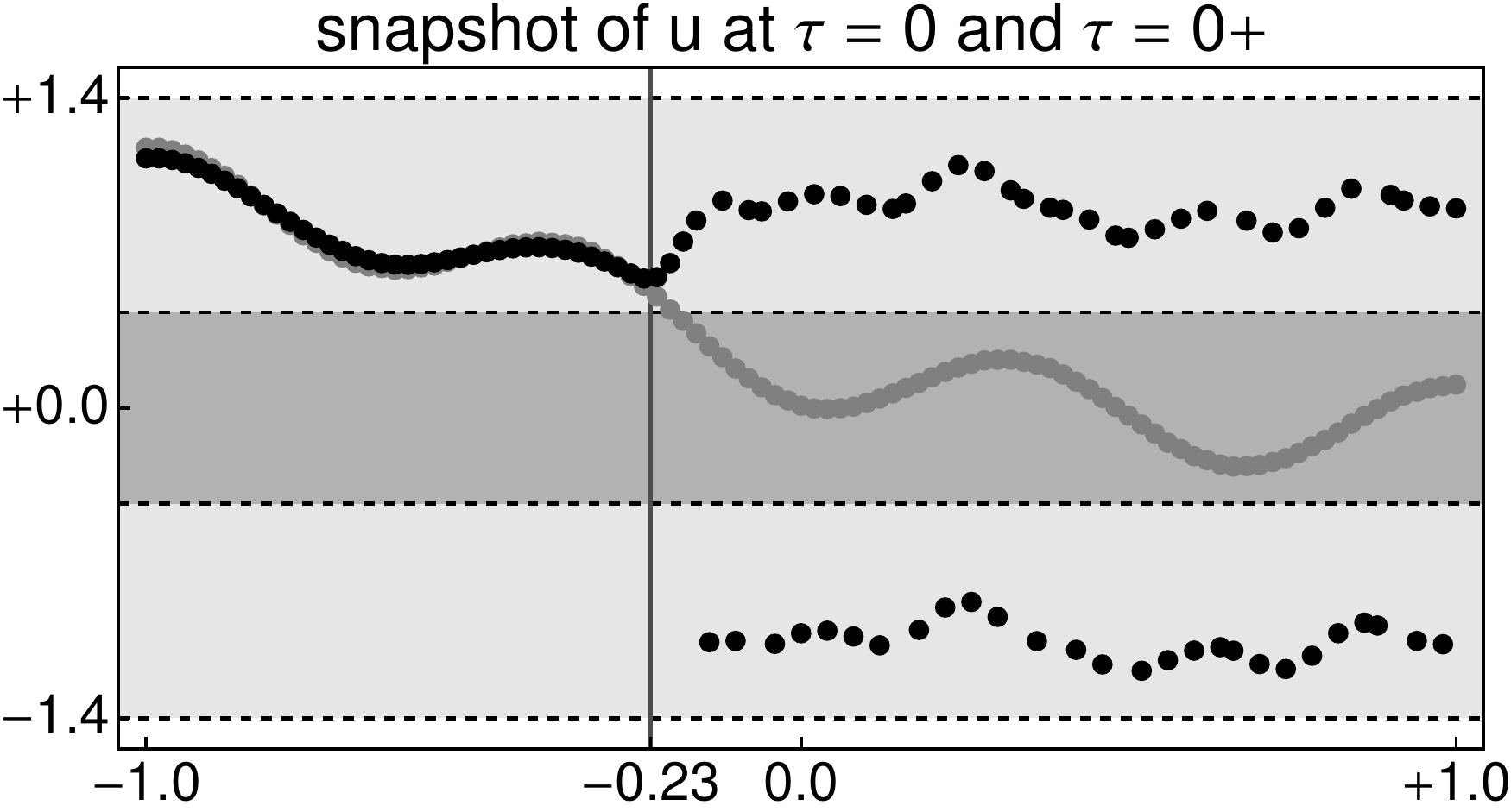}%
&\quad&%
\includegraphics[width=\figwidth]{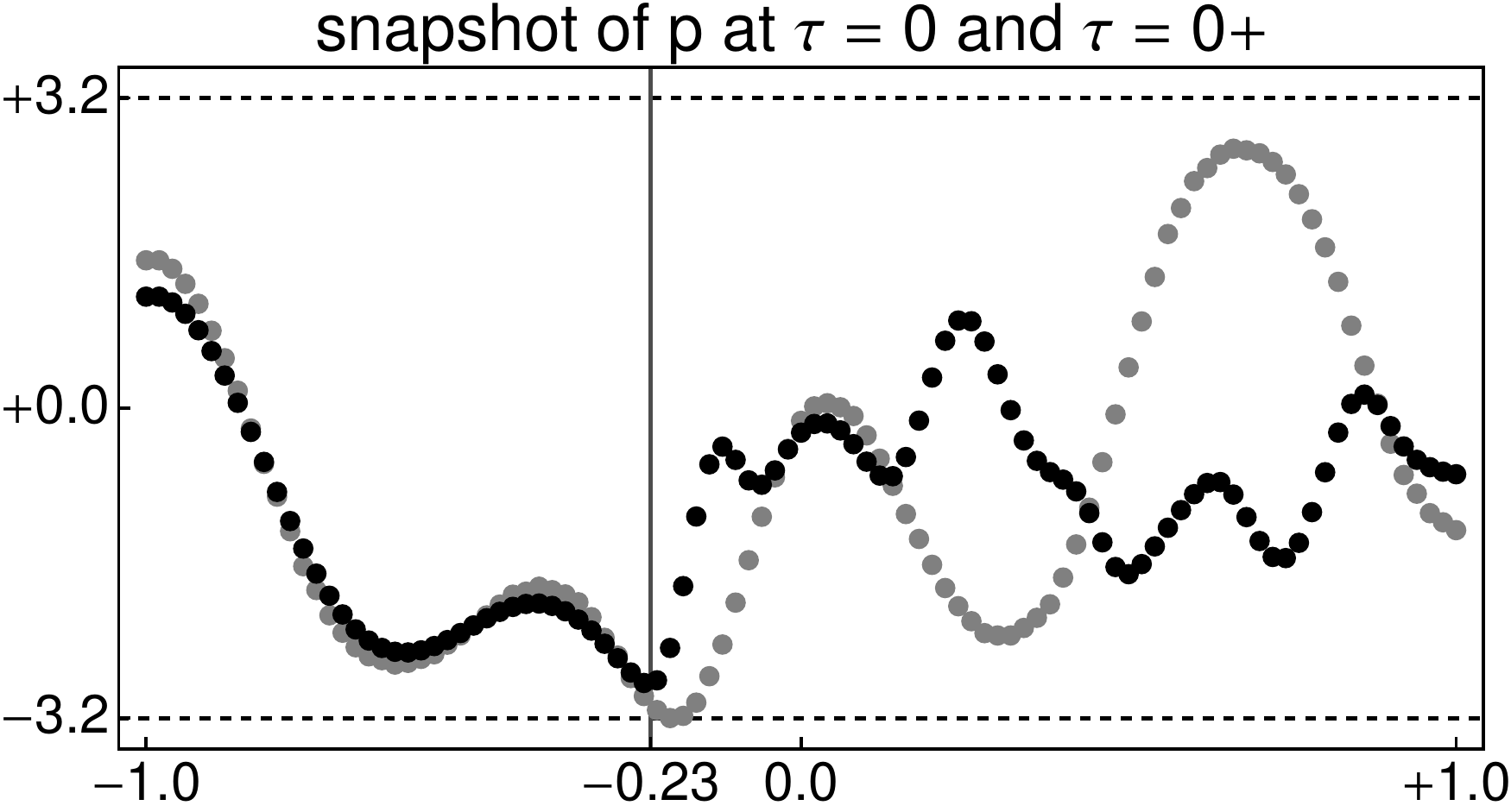}%
\end{tabular}%
}%
\caption{%
  Illustration of the initial transient regime for $N=50$, using
  snapshots of $u_j$ and $p_j$ (\emph{left} and \emph{right} column,
  respectively) against $\xi=\eps{j}$ at time $\tau=0$ (gray) and a
  short time $\tau=0+$ afterwards (black). The shadowed regions in the
  left column indicate the intervals $[u_*,u^*]$ and $[u_\#,u^\#]$;
  the dotted horizontal lines in the right column represent the
  critical values $\{p_*,p^*\}$. The vertical lines describe the
  macroscopic phase interfaces.  \emph{Top row.} In this example, the
  initial data $u_j\at{0}$ at $\tau=0$ do not penetrate the spinodal
  interval.  The lattice data $u_j$ and $p_j$ at $\tau=0+$ therefore
  resemble macroscopic functions $U$ and $P$, which are discontinuous
  and continuous, respectively, at the phase interface.  \emph{Bottom
    row.} Part of the initial data $u_j\at{0}$ are now taken from the
  spinodal interval.  At time $\tau=0+$, the lattice data $u_j$ can no
  longer be described by macroscopic functions but only by Young
  measures. The data $p_j$, however, still approximate a continuous
  macroscopic function $P$.}%
\label{Fig:transient.regime}%
\end{figure}%

\bigpar
We next discuss the small-time dynamics. For $0<\eps\ll1$ the initial
evolution of \eqref{eq:lattice} is related to a very fast transient
regime -- where `fast' refers to the macroscopic time $\tau$ -- during
which the system quickly approaches a state with macroscopic behavior
of $p_j$ and dissipation of order $\DO{1}$. This is illustrated in
Figure \ref{Fig:transient.regime}, which depicts two prototypical
examples of initial data at $\tau=0$ along with the state of the
system at a small macroscopic time $\tau=0+$ afterwards.
\par
In the top row of Figure \ref{Fig:transient.regime}, we start with
microscopic (i.\,e., oscillatory) initial data that are confined to
the two stable regions $u\in(-\infty,u_*)$ and
$u\in(u^*,+\infty)$. Due to the oscillations, the initial dissipation
is of order $\nDO{\eps^{-1} N}=\nDO{\eps^{-2}}$.  The initial energy
gradient is therefore also very large and drives the system
rapidly. At time $\tau=0+\approx\nDO{\eps^2}$, the dissipation and the
energy gradient eventually become of order $\nDO{1}$.  This implies
that the discrete data $p_j$ resemble a macroscopic function $P$ that
admits a weak spatial derivative and is hence continuous with respect
to $\xi$. We also observe that the discrete data $u_j$ at $\tau=0+$
approximate a piecewise continuous function $U$ which satisfies
$P=\Phi^\prime\at{U}$ and jumps across the interface located at
$\xi=0$. In particular, the \emph{phase fraction} $\mu$ defined by
\begin{align*}
\mu:= \chi_{[u^*,+\infty)}(U) - \chi_{(-\infty,u_*)}(U)
\end{align*}
takes the values $-1$ and $+1$ outside of the interface, where
$\chi_J$ denotes the indicator function of the interval $J$.

\par
The second example, see the bottom row in Figure
\ref{Fig:transient.regime}, is different since now some of the initial
data $u_j\at{0}$ belong to the spinodal interval
$[u_*,u^*]$, in which the discrete diffusion coefficient
is negative. In the numerical simulation we therefore observe that
each of those $u_j$ quickly leaves the spinodal interval
(\emph{spinodal decomposition}) and that the data for adjacent
$j$ can be attracted by different stable regions. In particular,
the data at time $\tau=0+\approx\nDO{\eps^2}$ exhibit a phase
interface near $\xi\approx-0.23$, in the sense that 
$u_j$ is non-oscillatory on the left but highly oscillatory on the
right of the interface. The fine structure of these oscillations
depends on the microscopic details and each reasonable macroscopic
theory must describe them in terms of a Young measure
$\nu=\nu\triple{\tau}{\xi}{\mathrm{d}u}$, which provides a probability
distribution with respect to $u$ for any macroscopic point
$\pair{\tau}{\xi}$. The discrete data $p_j$, however, still resemble a
macroscopic function $P$ because otherwise the dissipation could not
be of order $\DO{1}$. Since $p_j$ and $u_j$ are coupled by
$\Phi^\prime$, we then conclude that the $u$-support of the Young
measure $\nu$ consists of only two points. This reads
\begin{align}
  \label{eq:young.measure}
  \nu\triple{\tau}{\xi}{\mathrm{d}u}
  =
  \frac{1-\mu\pair{\tau}{\nu}}{2}
  \delta_{\beta_-\at{P\pair{t}{\xi}}}\at{\mathrm{d}u}
  +
  \frac{1+\mu\pair{\tau}{\nu}}{2}
  \delta_{\beta_+\at{P\pair{t}{\xi}}}\at{\mathrm{d}u}.
\end{align}
Here, $\delta_\beta\at{\mathrm{d}u}$ is the Dirac distribution at
$\beta$, the functions $\beta_-$, $\beta_+$ denote the two stable
branches of the inverse of $\Phi^\prime$, and the phase fraction $\mu$
takes values in $[-1,+1]$.
\par
The simulations from Figure \ref{Fig:transient.regime} reveal that
there exist (at least) two different types of macroscopic phase
interfaces: Type-I interfaces separate regions where the microscopic
data $u_j$ are confined to either one of the phases $(-\infty,u_*)$
and $(u^*,+\infty)$, whereas type-II interfaces describe that the
lattice data oscillate between the two phases on at least one side of
the interface.  Below we argue that type-I interfaces can be described
by a free boundary value problems, which exhibits hysteresis and
involves only the macroscopic fields $P$ and
$\mu\in\{-1,+1\}$. Type-II interfaces, however, are more complicated
and their investigation is postponed to future research.

\subsection{Examples of macroscopic interface dynamics }
\label{sec:Macroscopic.Interfaces}
%
In this section we study the dynamics of type-I interfaces in
numerical simulations. In particular, we investigate the macroscopic
jump conditions across such interfaces and provide numerical evidence
for pinning, depinning, and annihilation. At the end we also present
an example of a type-II interface.
\par
For simplicity, and in view of the discussion in the previous section,
we always impose initial data $u_j\at{0}$ such that
$p_j\at{0}=\Phi^\prime\at{u_j\at{0}}$ resemble a macroscopic function
$\xi\mapsto P\pair{0}{\xi}$. In all simulations we observe -- for,
loosely speaking, most of the macroscopic times $\tau>0$ -- that the
discrete data $p_j\at{\tau/\eps^2}$ approximate a macroscopic function
$\xi\mapsto P\pair{\tau}{\xi}$. We therefore expect that the
macroscopic limit $\eps\to0$ can in fact be characterized by a PDE for
$P$ and the phase field $\mu$, or equivalently, in terms of a free
boundary problem for $P=\Phi^\prime\at{U}$ and the interface
curves. There exist, however, small macroscopic times intervals in
which the discrete data $p_j\at{t}$ exhibit strong temporal and
spatial fluctuations near a moving phase interface.  These
fluctuations are discussed in the next section.
\renewcommand{\figwidth}{0.4\textwidth}%
\begin{figure}[t!]%
\centering{%
\begin{tabular}{ccc}%
\includegraphics[width=\figwidth]{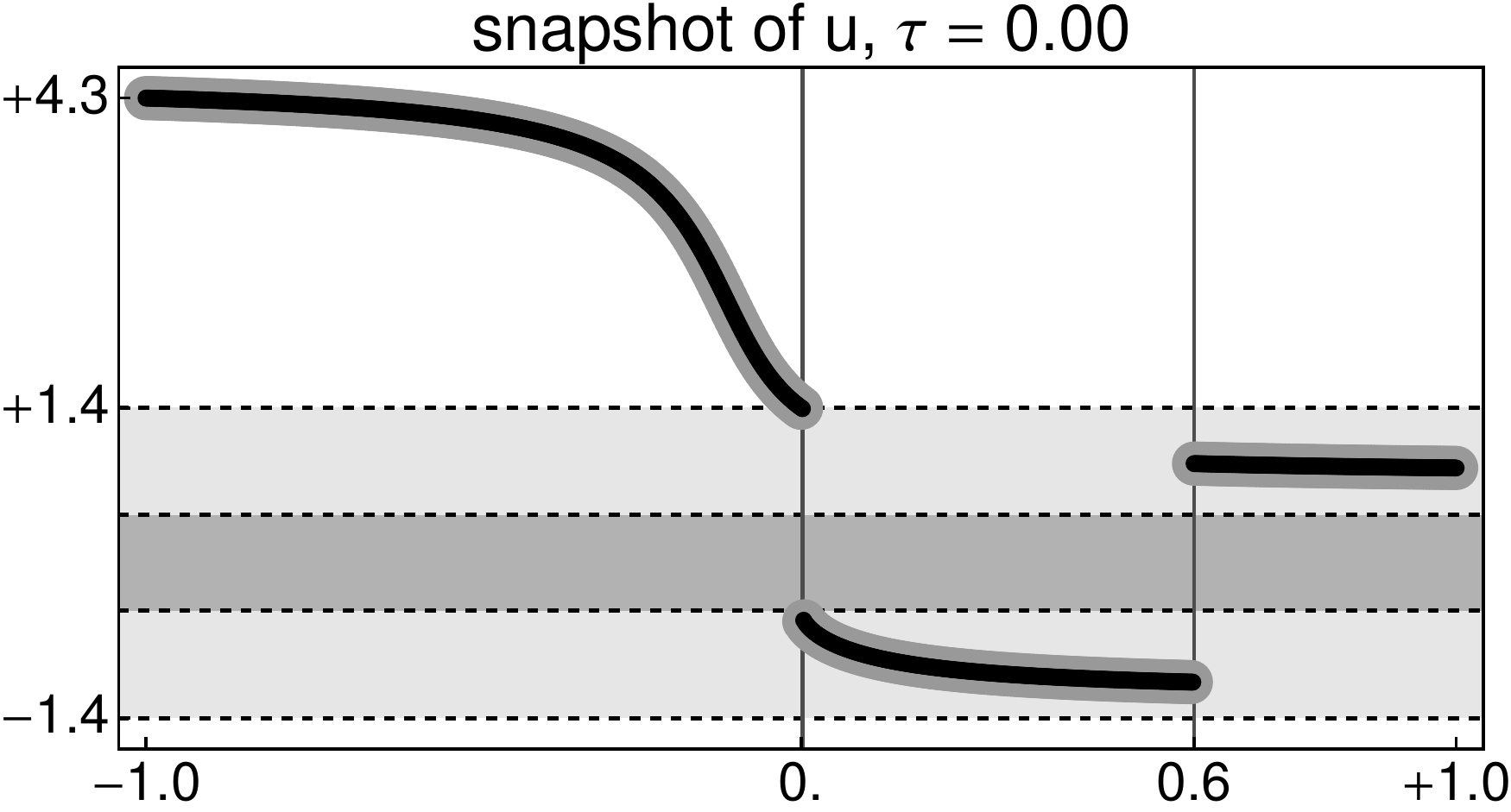}%
&\quad&%
\includegraphics[width=\figwidth]{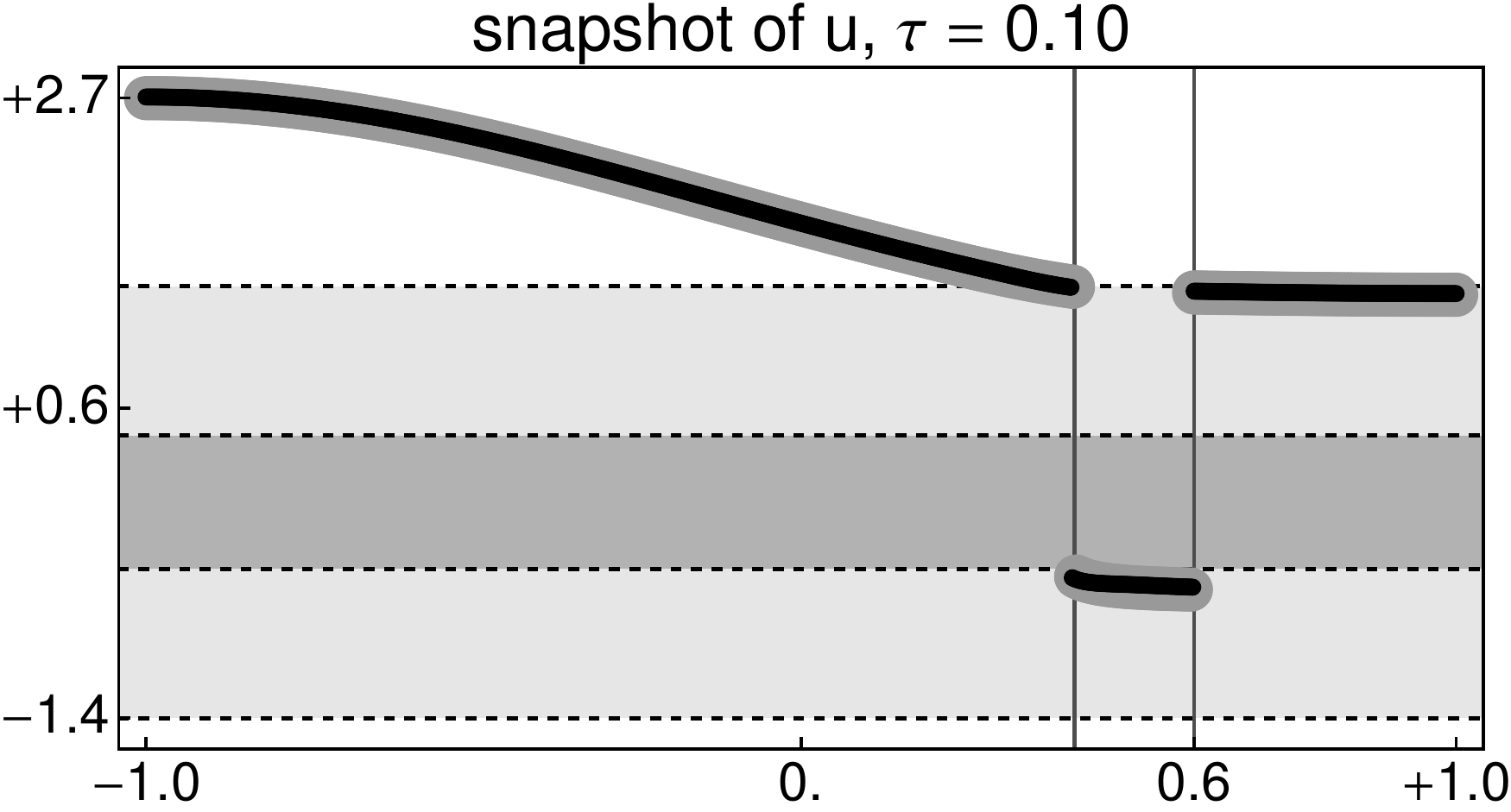}%
\\%
\includegraphics[width=\figwidth]{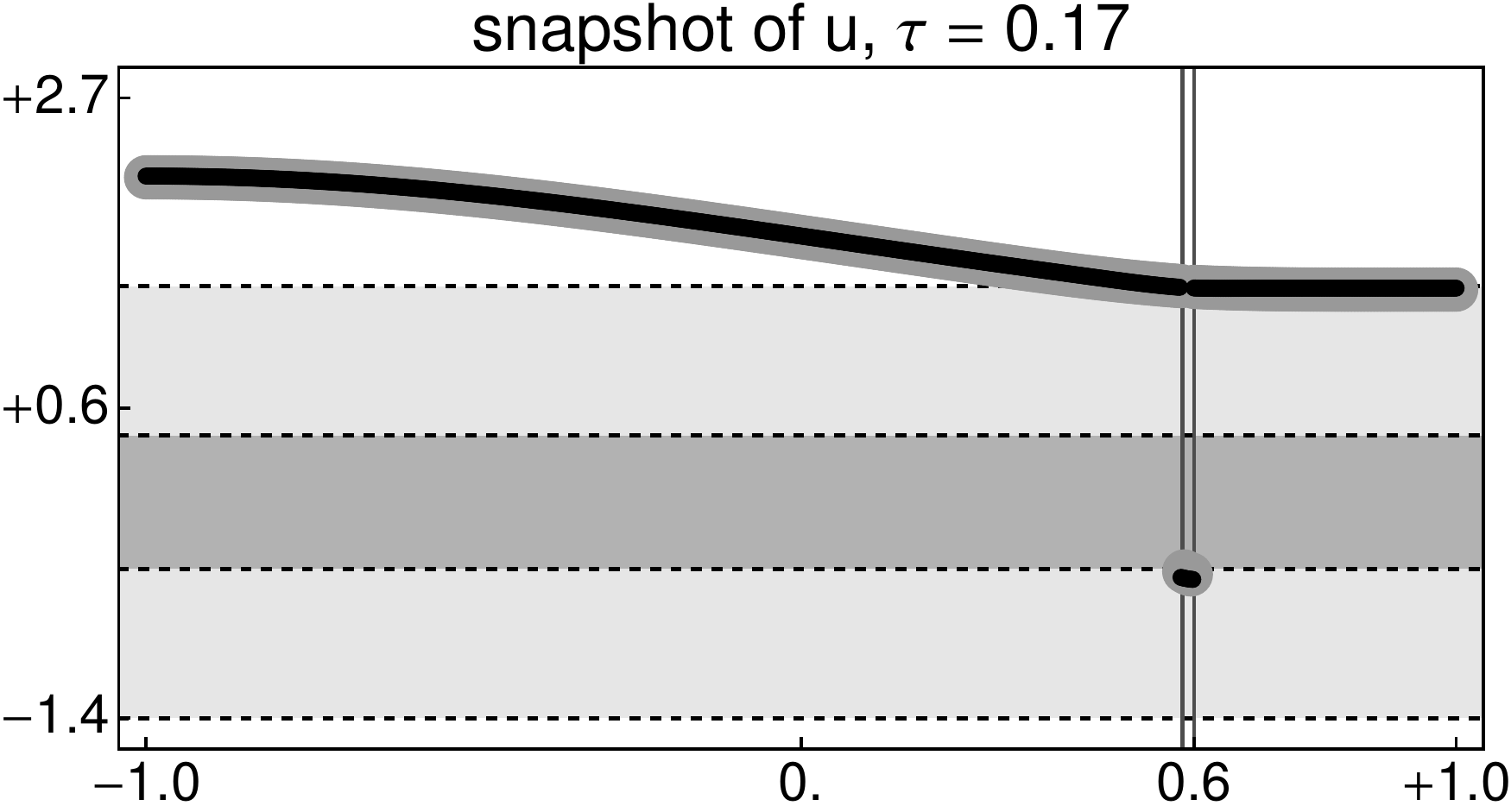}%
&\quad&%
\includegraphics[width=\figwidth]{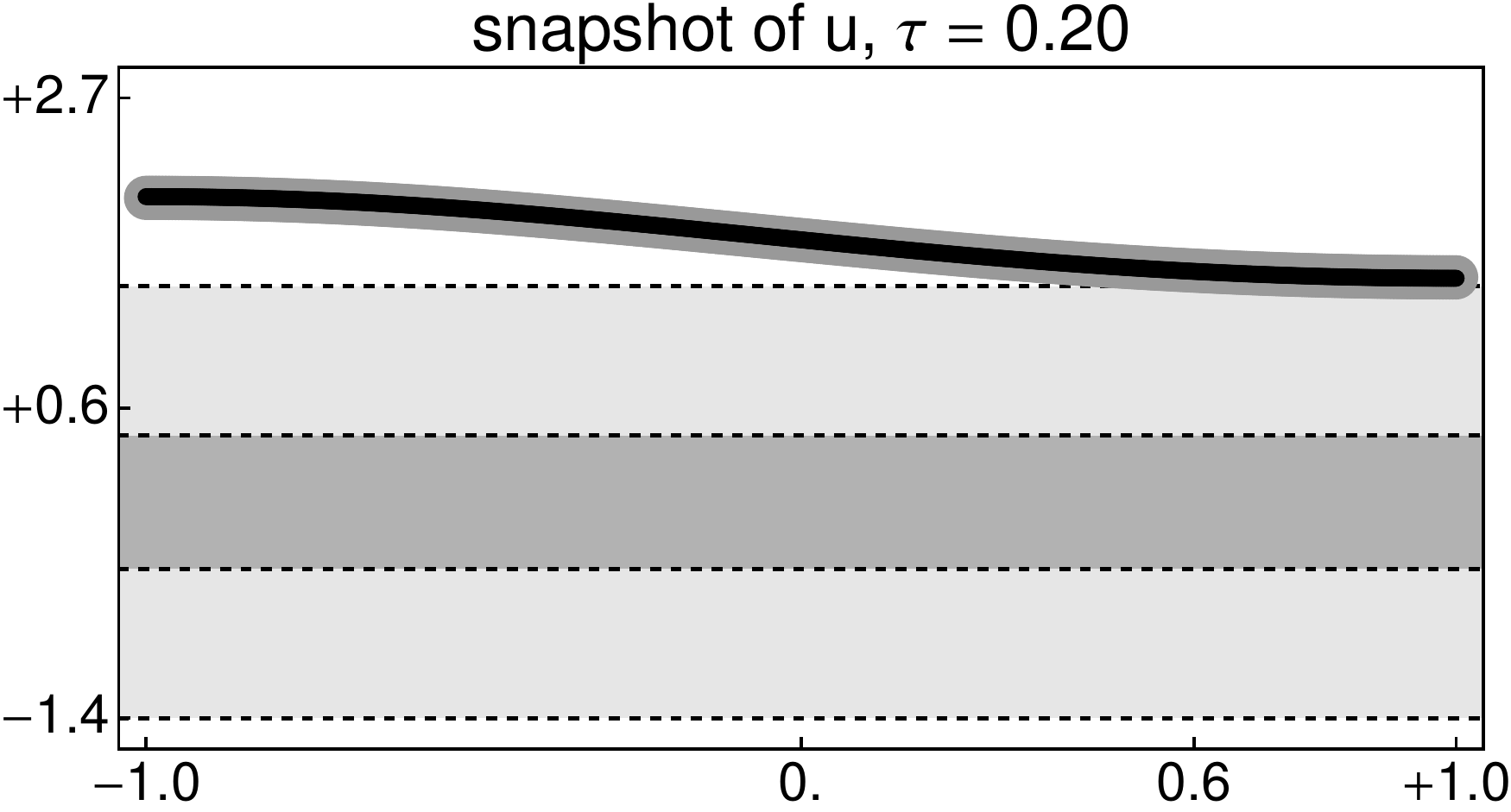}%
\end{tabular}%
}%
\caption{%
  Annihilation of a moving and a standing type-I interface for two
  different simulations with $N=200$ (thick gray curves) and $N=500$
  (thin black curves). The snapshots show $u_j$ against the
  macroscopic position $\xi=\eps j$ at fixed macroscopic times $\tau$;
  the vertical lines indicate the interface positions and the two
  shaded regions represent the intervals $[u_*,u^*]$ and
  $[u_\#,u^\#]$.}%
\label{Fig:front_1a}%
\bigskip
\centering{%
\begin{tabular}{ccc}%
\includegraphics[width=\figwidth]{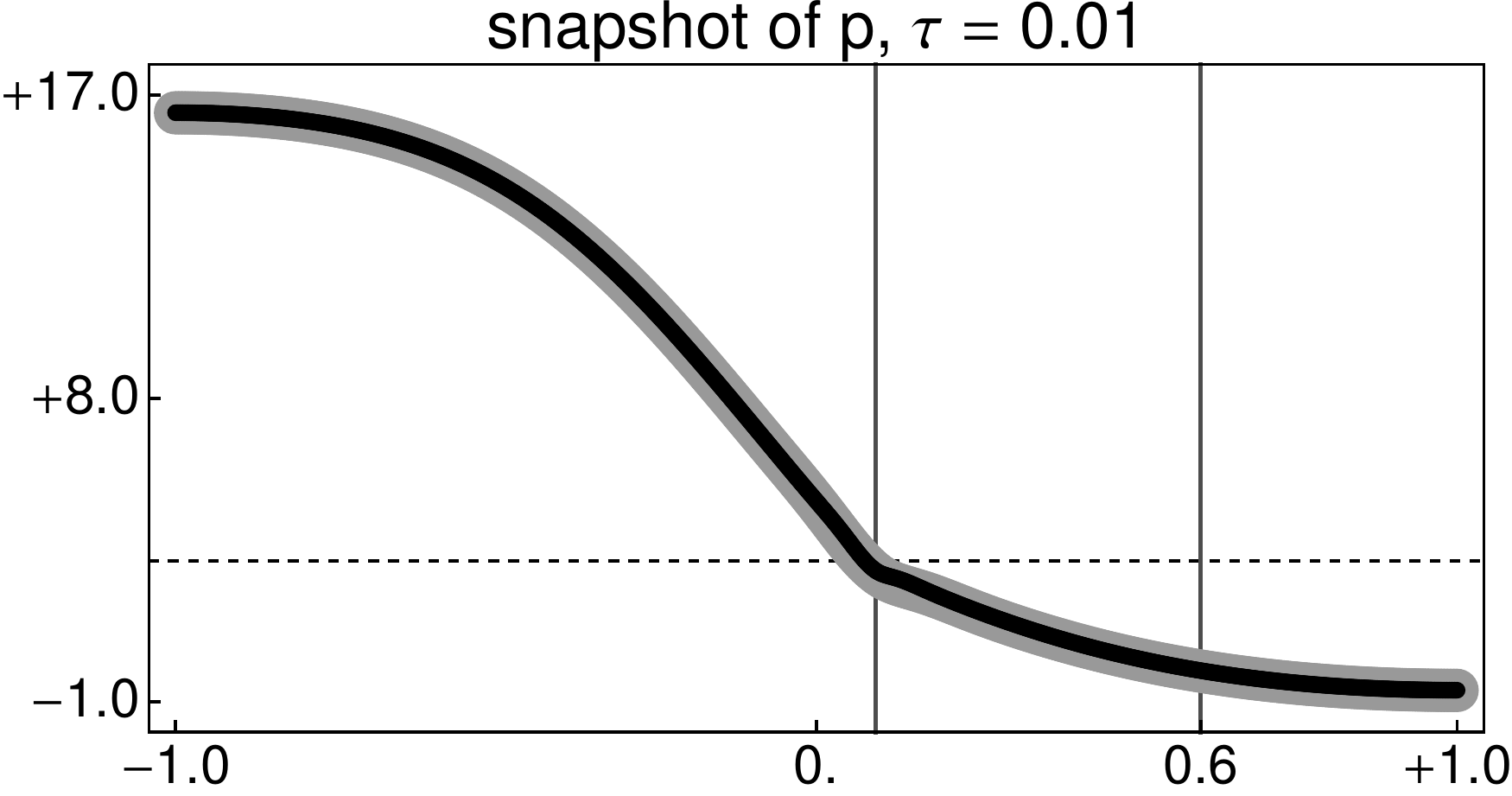}%
&\quad&%
\includegraphics[width=\figwidth]{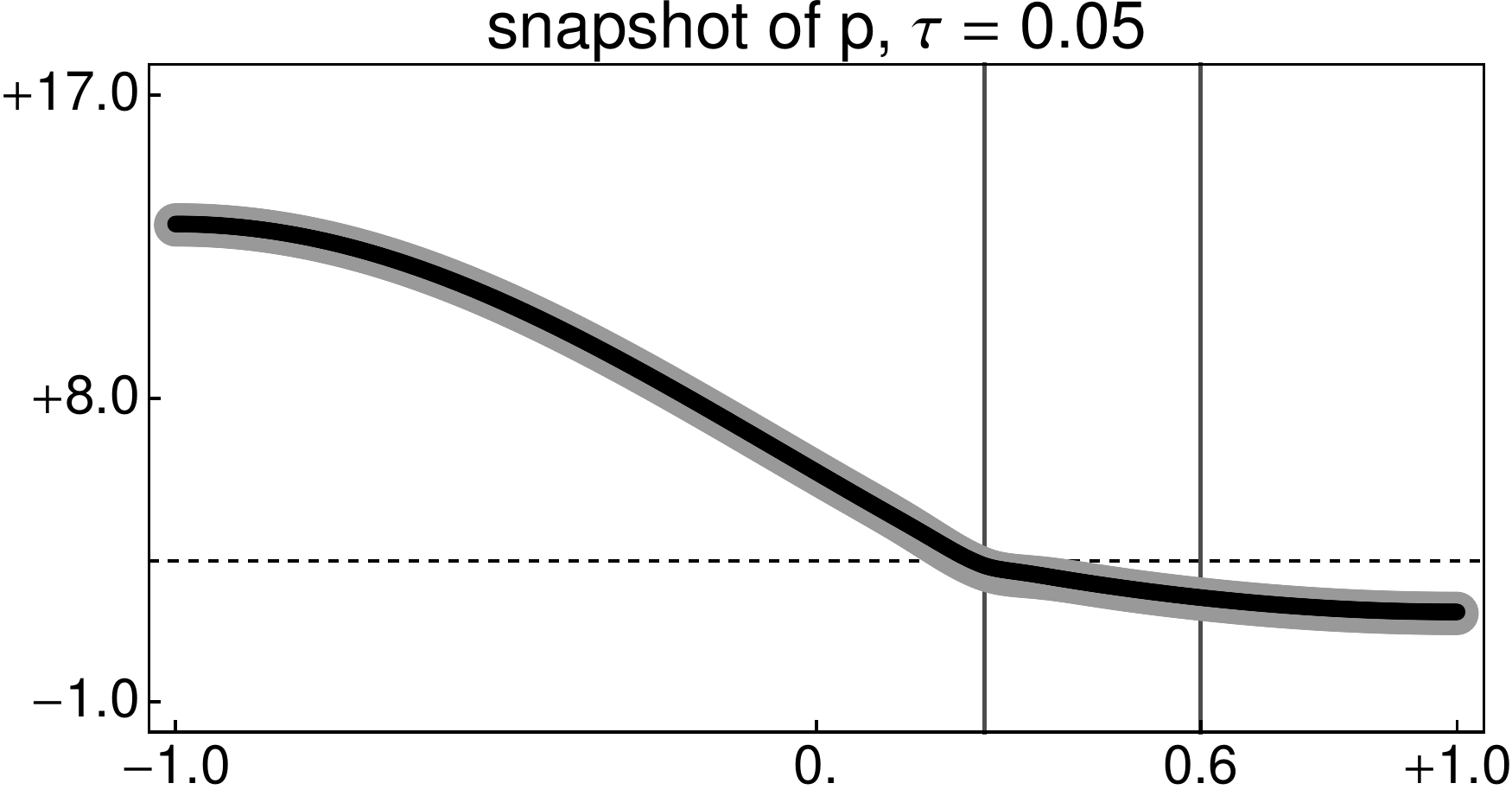}%
\end{tabular}%
}%
\caption{%
Snapshots of $p_j$ for the example from Figure \ref{Fig:front_1a}. The horizontal line represents $p=p^*$.
}%
\label{Fig:front_1b}%
\bigskip
\centering{%
\begin{tabular}{ccc}%
\includegraphics[width=\figwidth]{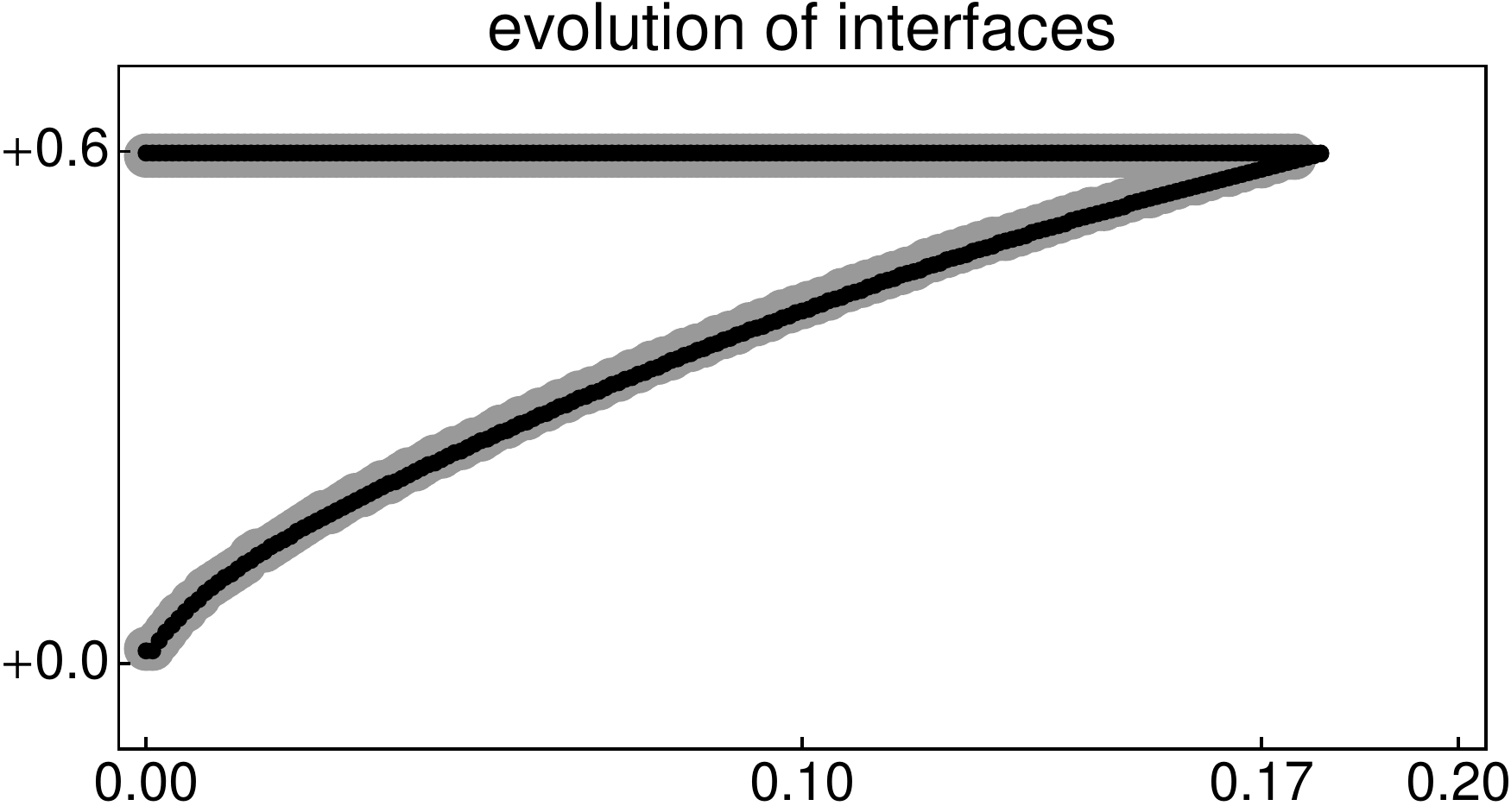}%
&\quad&%
\includegraphics[width=\figwidth]{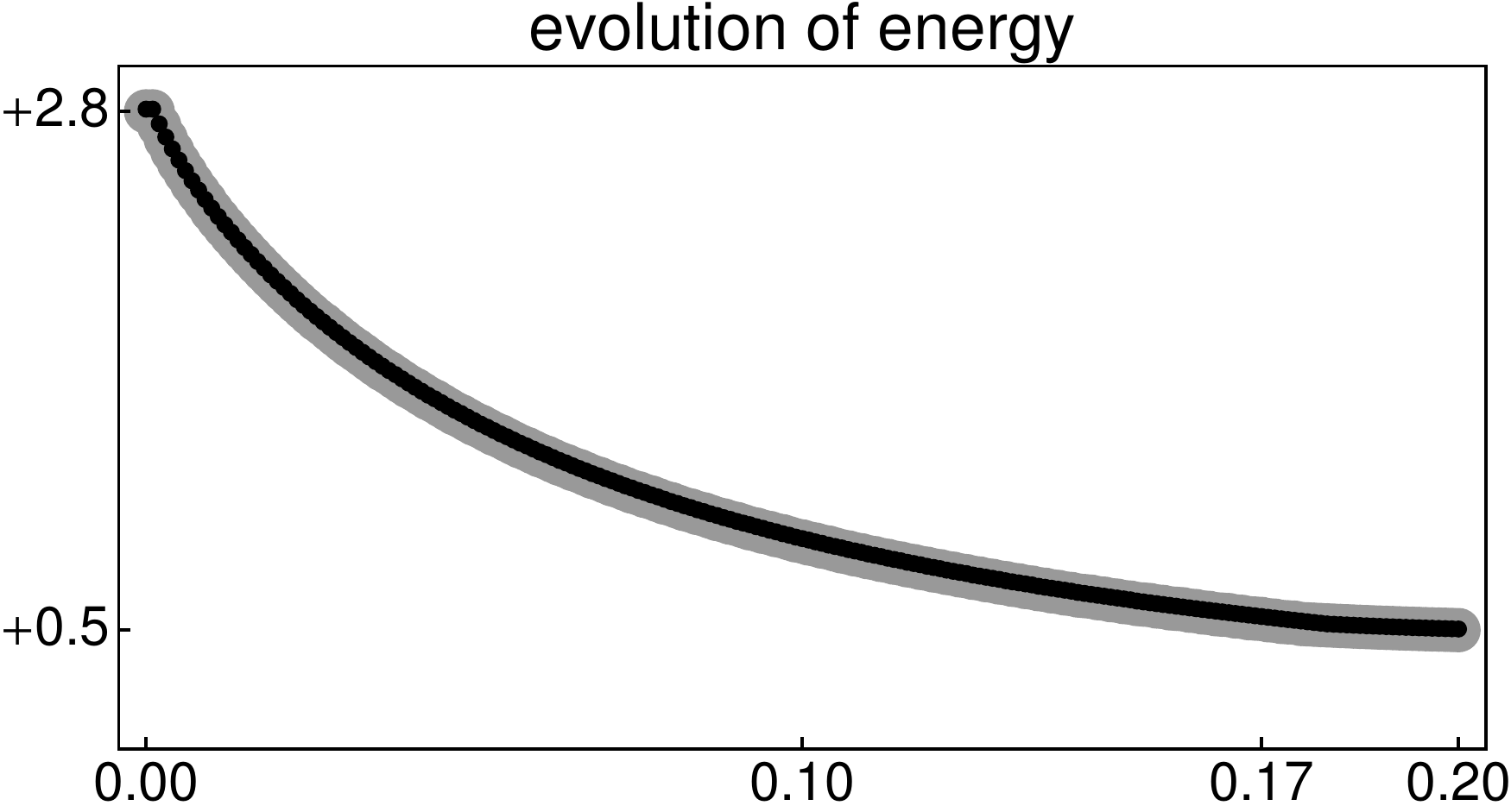}%
\end{tabular}%
}%
\caption{%
Macroscopic positions
of the interfaces and the energy as function of $\tau$
for the example from Figure \ref{Fig:front_1a}.
}%
\label{Fig:front_1c}%
\end{figure}%
\bigpar
Figure \ref{Fig:front_1a} depicts lattice simulations with two
different values of $N$, showing that the macroscopic plots of the
discrete data are basically independent of $N$. \EMHC In this example,
we initialize two macroscopic phase interfaces which are located at
$\xi=0$ and $\xi=0.6$ and separate regions with $U>u^\#$,
$U\in[u_\#,u_*]$, and $U\in[u^*,u^\#]$. The first interface moves to
the right while the second one clearly keeps its initial position. At
the later time $\tau\approx 0.18$ both interfaces annihilate each
other in a collision process, see also Figure \ref{Fig:front_1c}, and
the macroscopic evolution afterwards is governed by nonlinear
diffusion inside the phase $U\in[u^*,\infty)$.  Figure
\ref{Fig:front_1b} provides numerical evidence for the jump rules
across the interface: The moving interface is driven by a jump in
$\partial_\xi P$ whereas $P$ is smooth across the standing
interface. Moreover, while $P=p^*$ holds on the moving interface, $P$
evolves on the standing interface and takes values in $[p_*,p^*]$.
\renewcommand{\figwidth}{0.4\textwidth}%
\begin{figure}[t!]%
\centering{%
\begin{tabular}{ccc}%
\includegraphics[width=\figwidth]{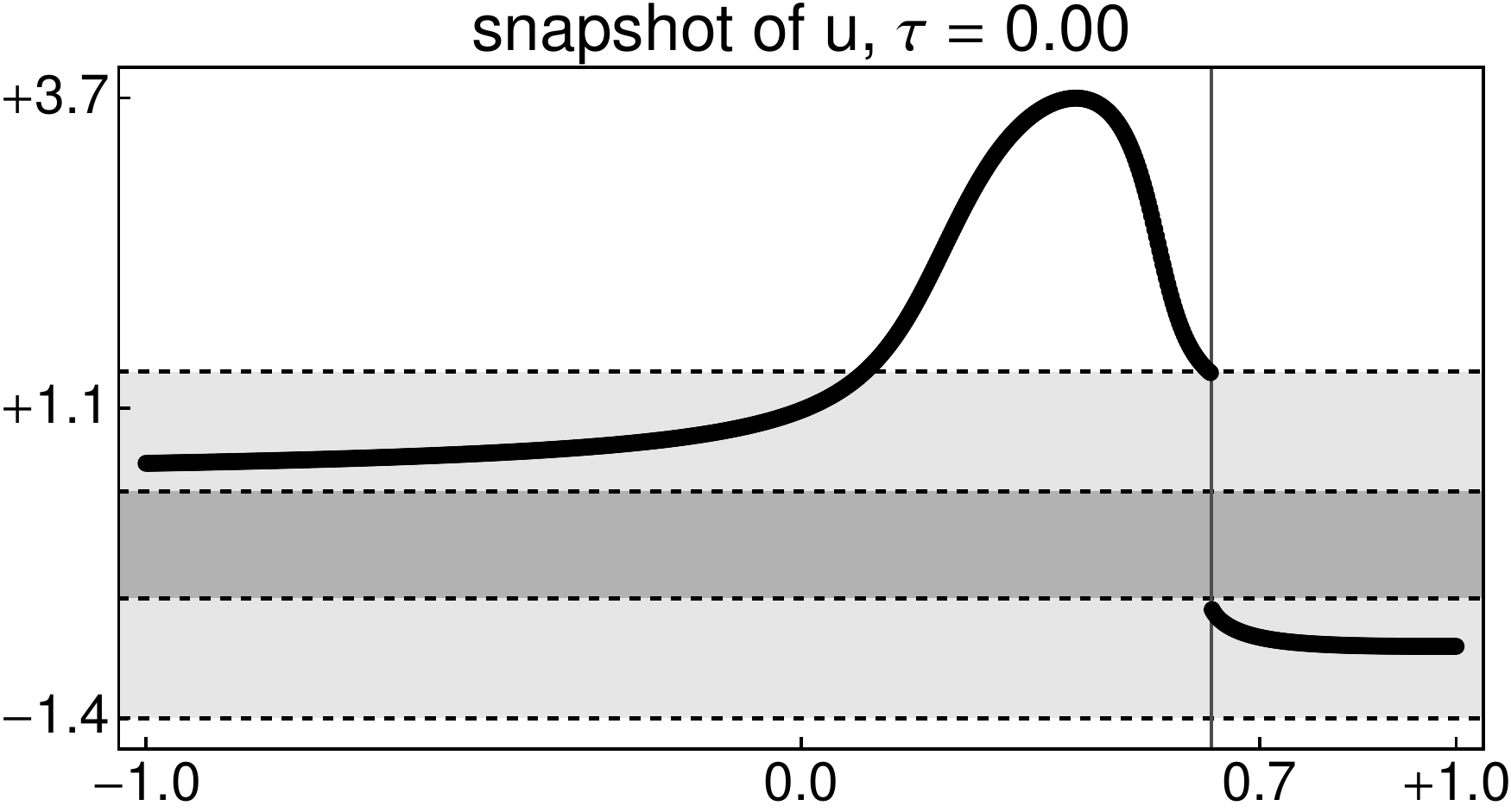}%
&\quad&%
\includegraphics[width=\figwidth]{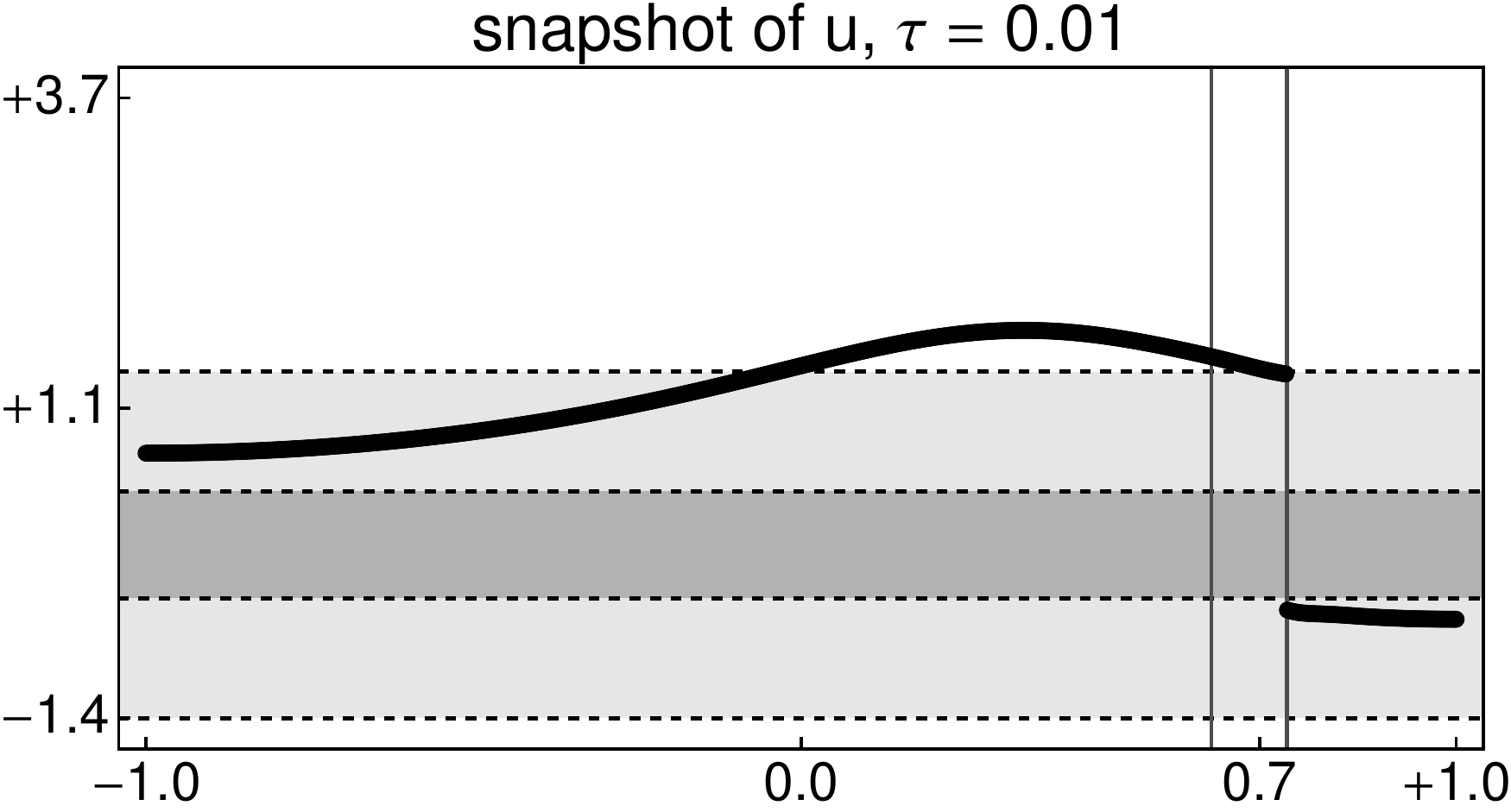}%
\\%
\includegraphics[width=\figwidth]{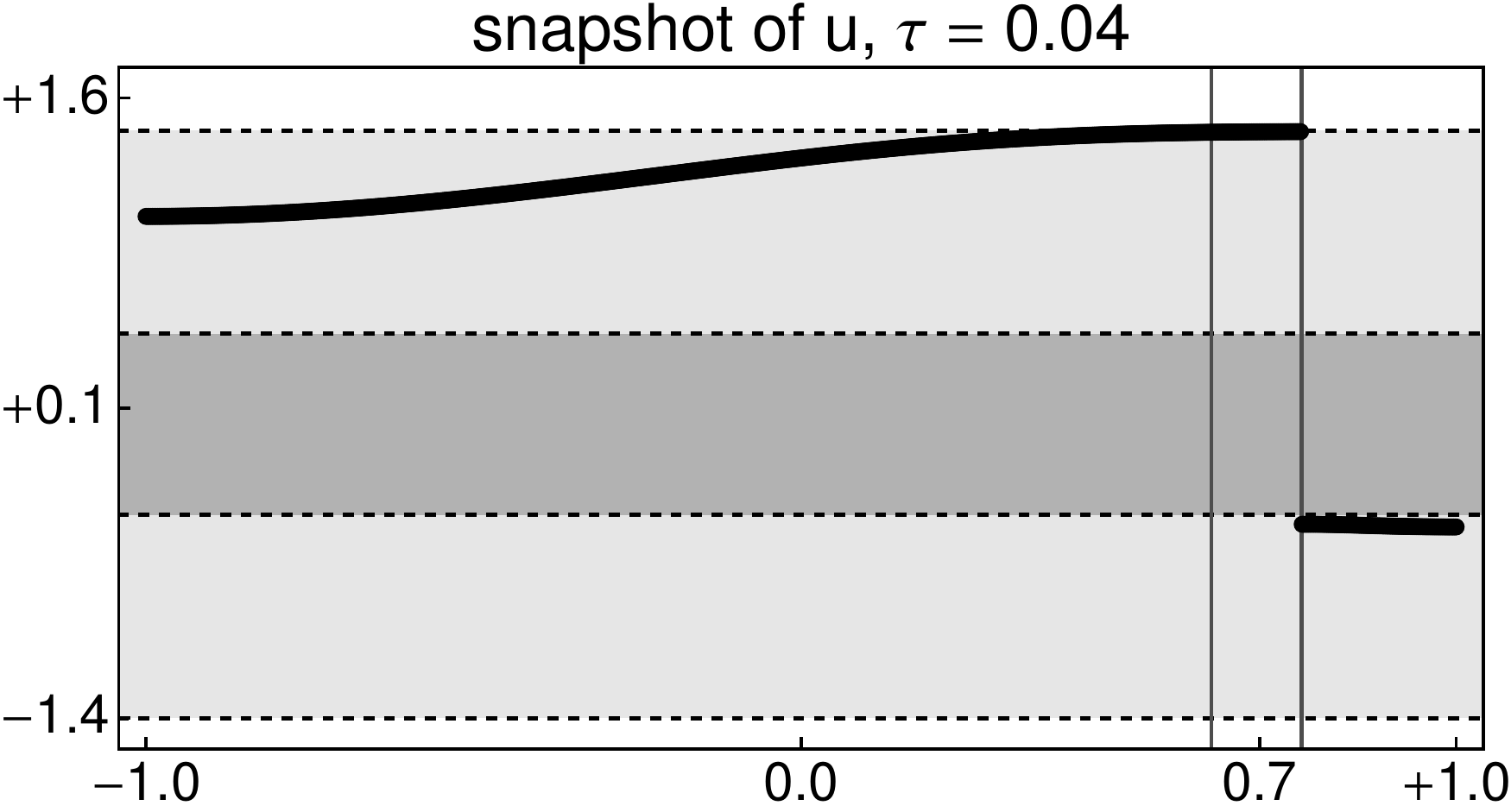}%
&\quad&%
\includegraphics[width=\figwidth]{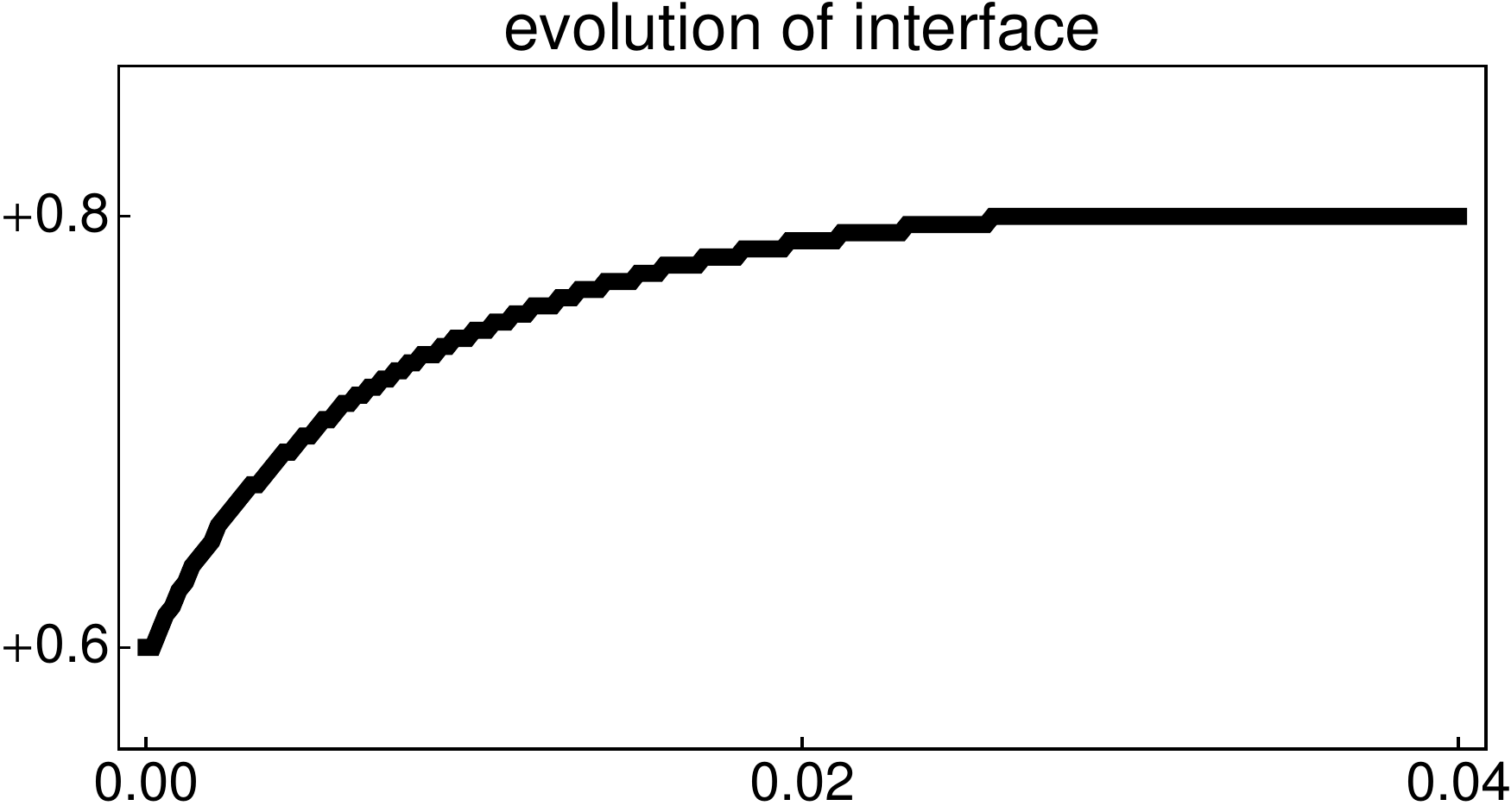}%
\end{tabular}%
}%
\caption{%
Pinning of a type-I interface with $N=400$. The vertical lines indicate the initial and the current position of the interface.
}%
\label{Fig:front_2}%
\end{figure}%
\renewcommand{\figwidth}{0.4\textwidth}%
\begin{figure}[t!]%
\centering{%
\begin{tabular}{ccc}%
\includegraphics[width=\figwidth]{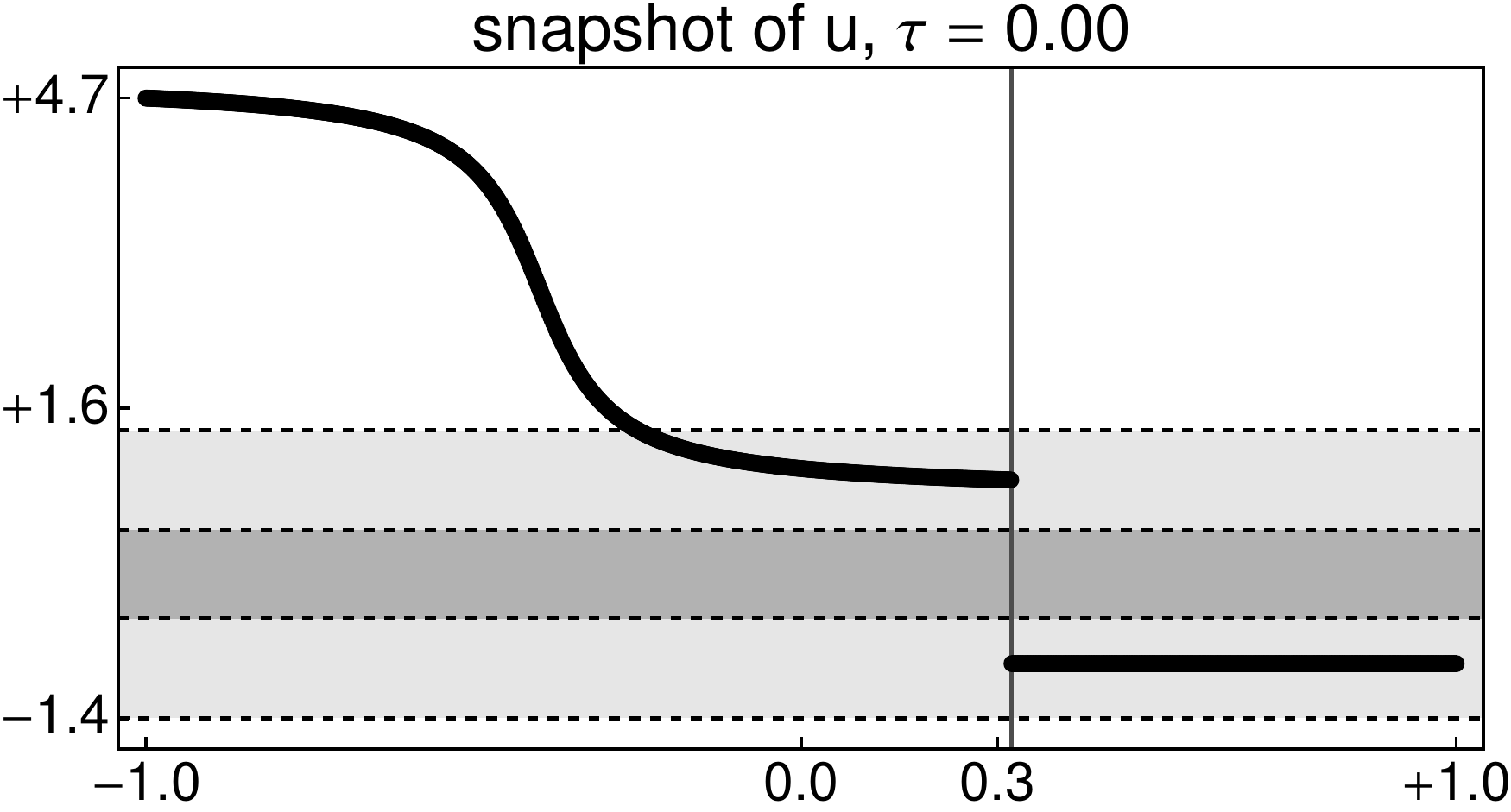}%
&\quad&%
\includegraphics[width=\figwidth]{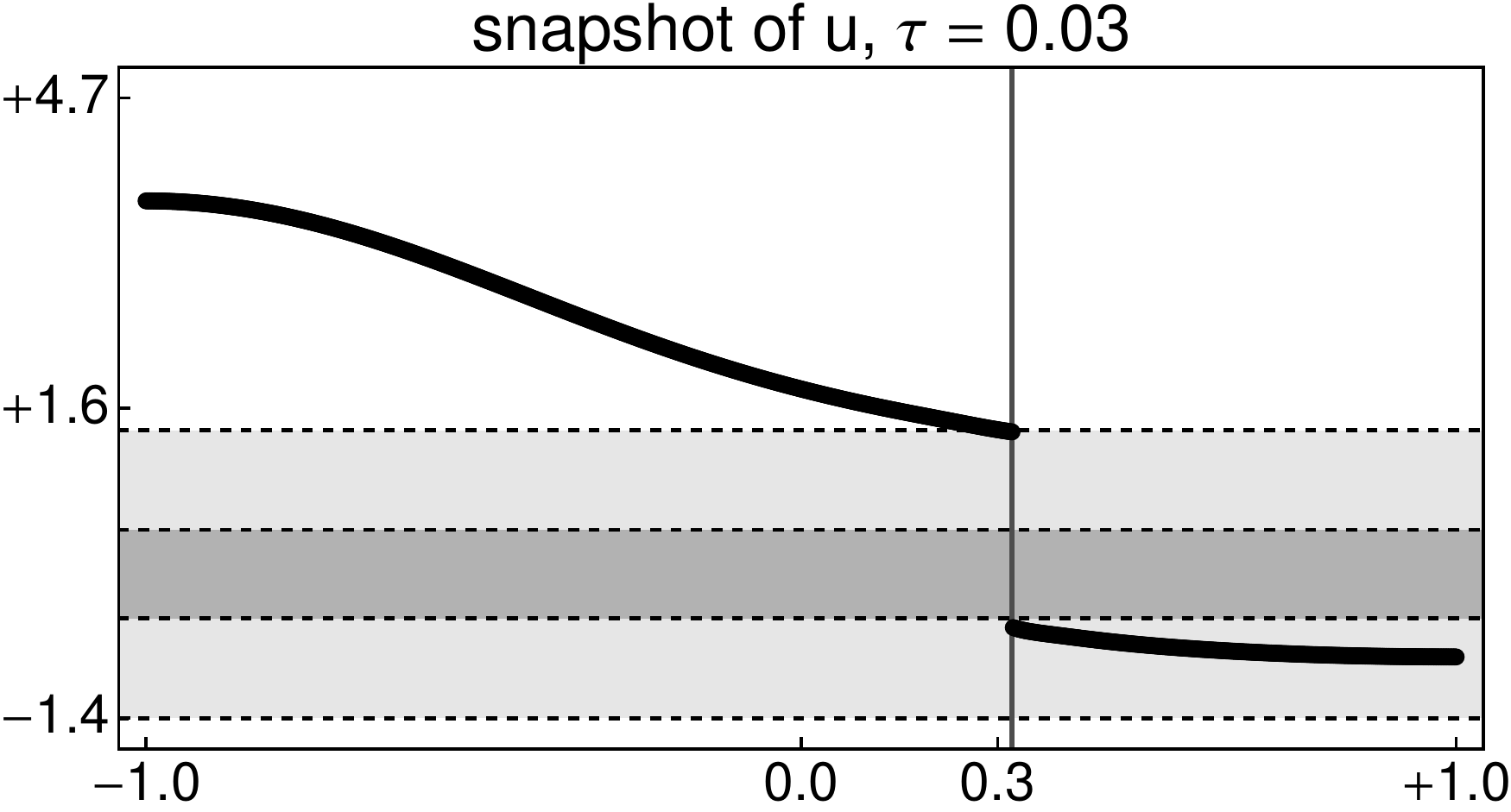}%
\\%
\includegraphics[width=\figwidth]{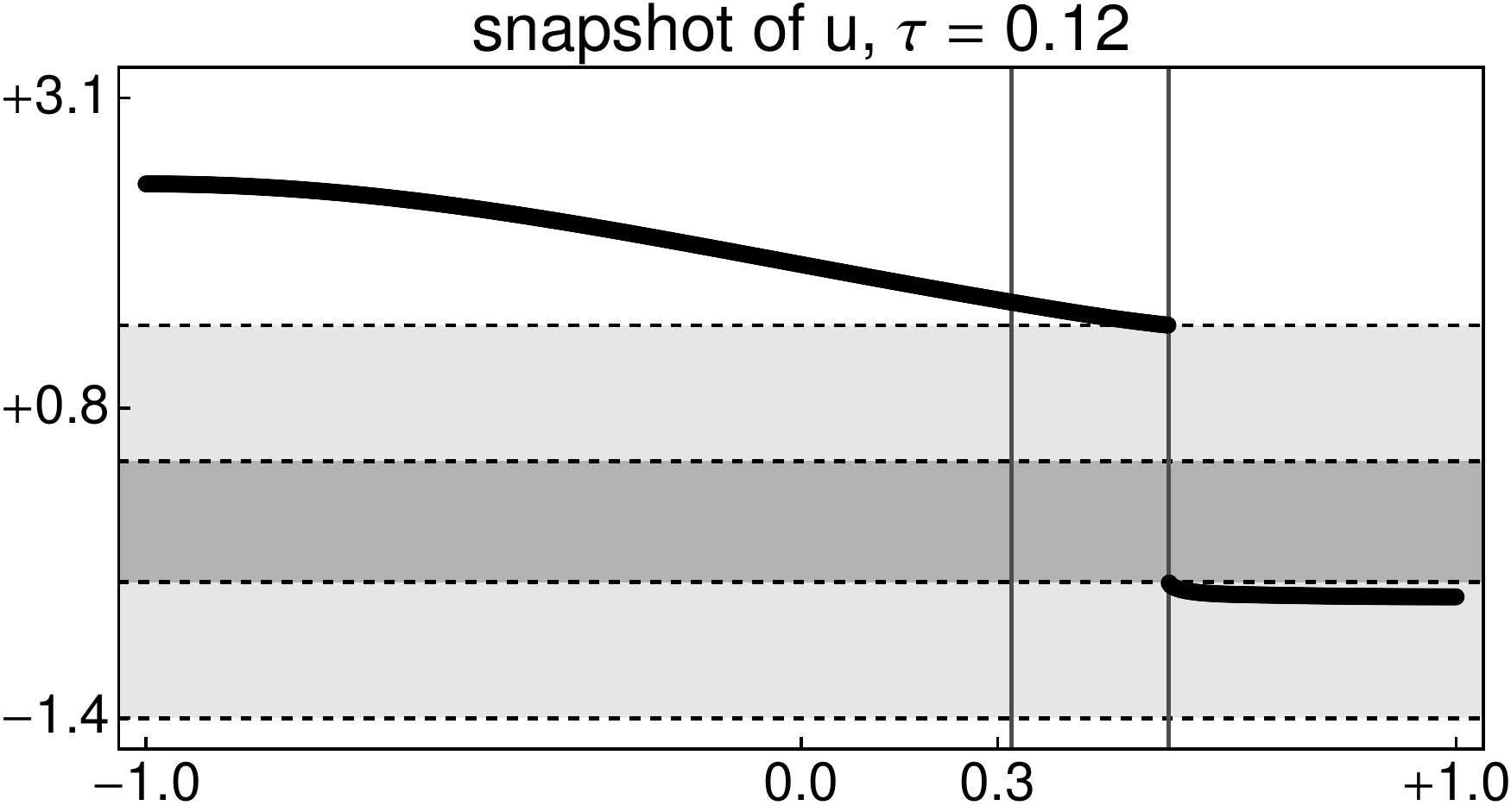}%
&\quad&%
\includegraphics[width=\figwidth]{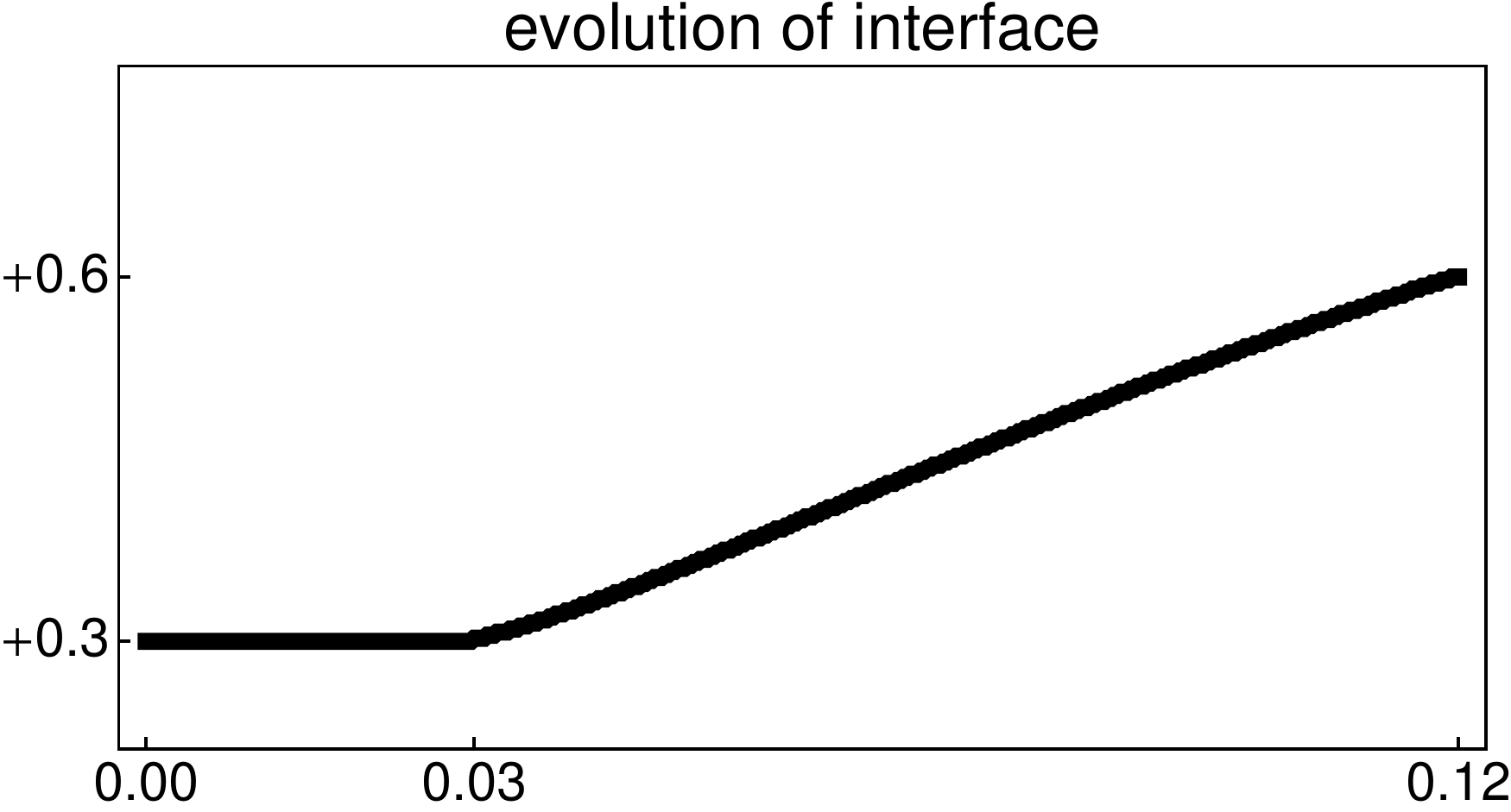}%
\end{tabular}%
}%
\caption{%
Depinning of a type-I interface with $N=500$.
}%
\label{Fig:front_3}%
\end{figure}%
\par
A further dynamical feature of the lattice \eqref{eq:lattice}, namely
the pinning of interfaces, is illustrated in Figure
\ref{Fig:front_2}. At time $\tau=0$, we initialize a single
macroscopic interface that separates regions with $U\geq u^*$ and
$U\in[u_\#,u_*]$, where the data are chosen such that
$P>p^*$ and $P<p^*$ holds locally on the left and on the right of the
interface, respectively.  This interface starts propagating to the
right but stops moving at $\tau\approx0.02$ because the bulk diffusion
behind the interfaces enforces $P\leq p^*$ for $\tau\gtrsim0.02$. The
inverse process, that is the depinning of interfaces, is shown in
Figure \ref{Fig:front_3}. There, a single macroscopic interface is
initially at rest with $P\in(p_*,p^*)$ but propagates with
$P=p^*$ for $\tau\gtrsim0.03$.
\renewcommand{\figwidth}{0.4\textwidth}%
\begin{figure}[t!]%
\centering{%
\begin{tabular}{ccc}%
\includegraphics[width=\figwidth]{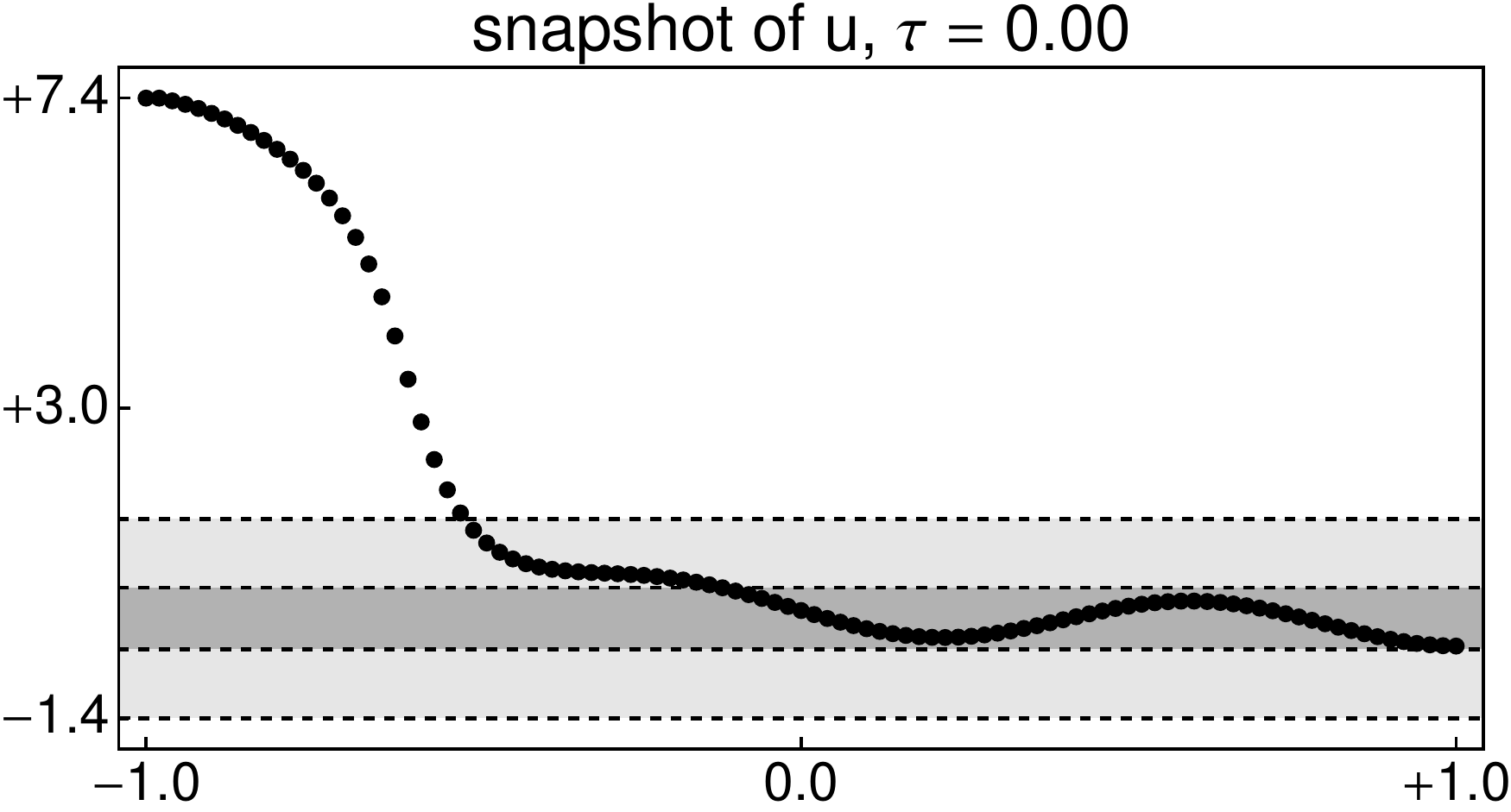}%
&\quad&%
\includegraphics[width=\figwidth]{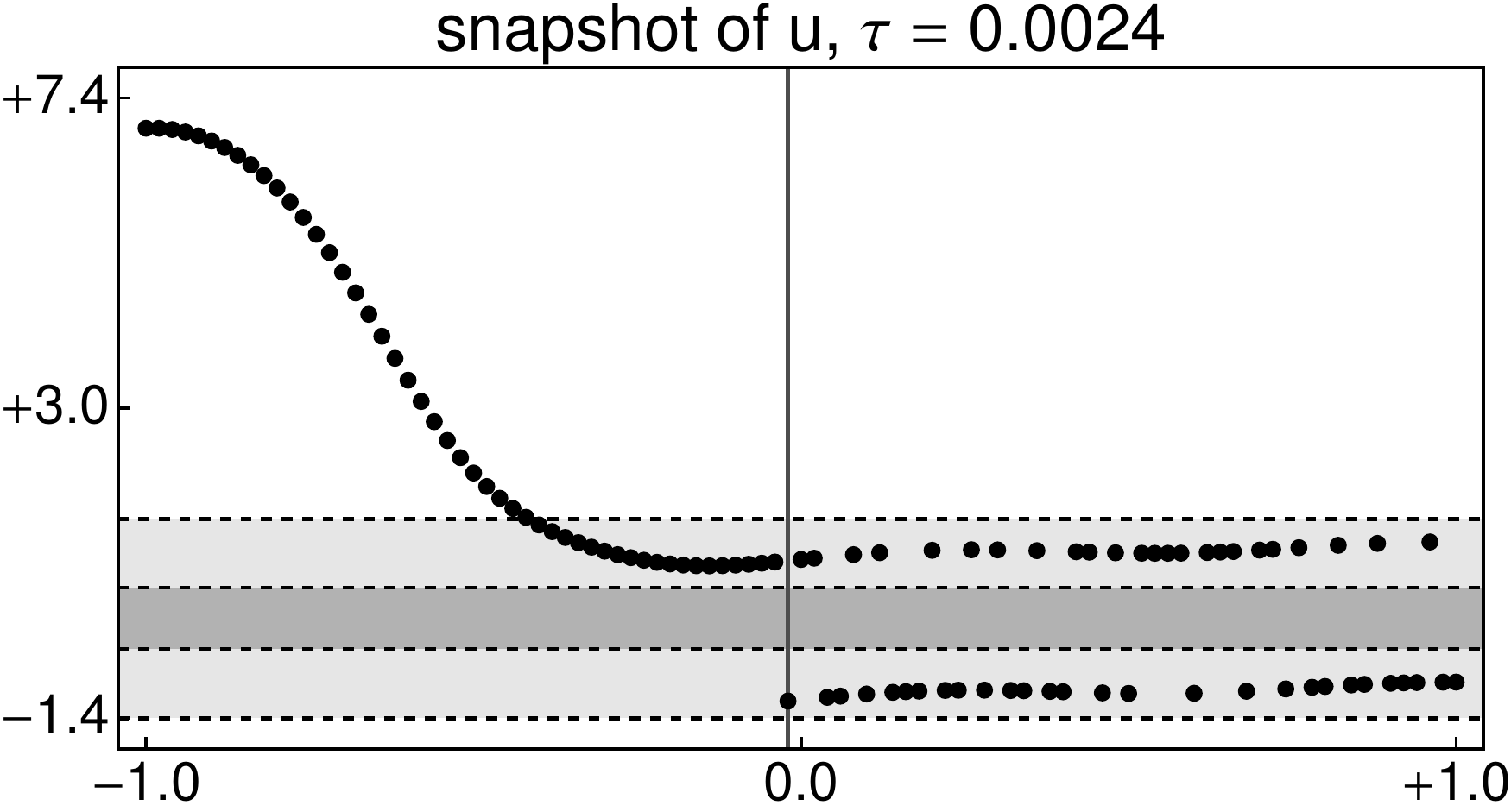}%
\\%
\includegraphics[width=\figwidth]{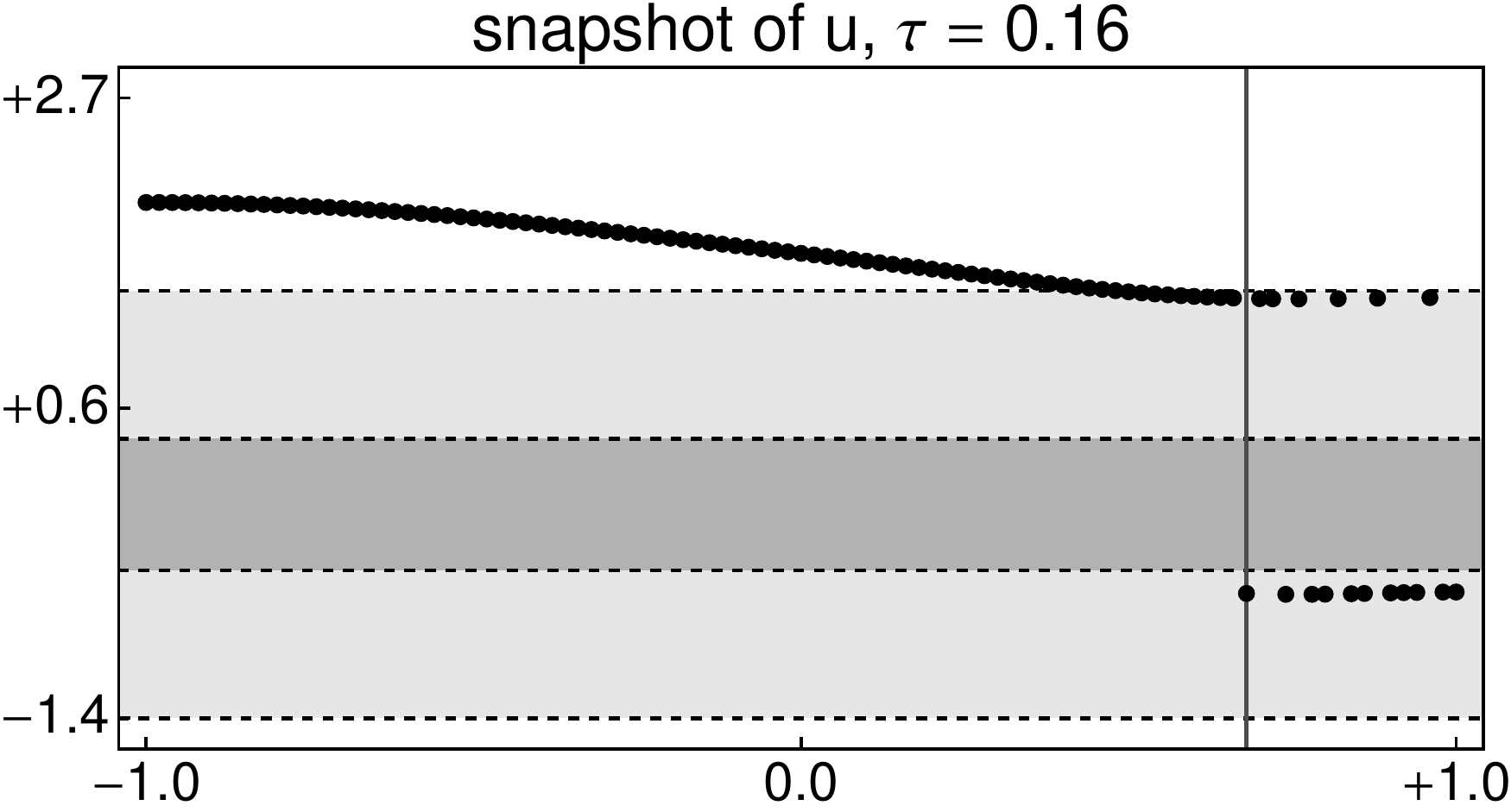}%
&\quad&%
\includegraphics[width=\figwidth]{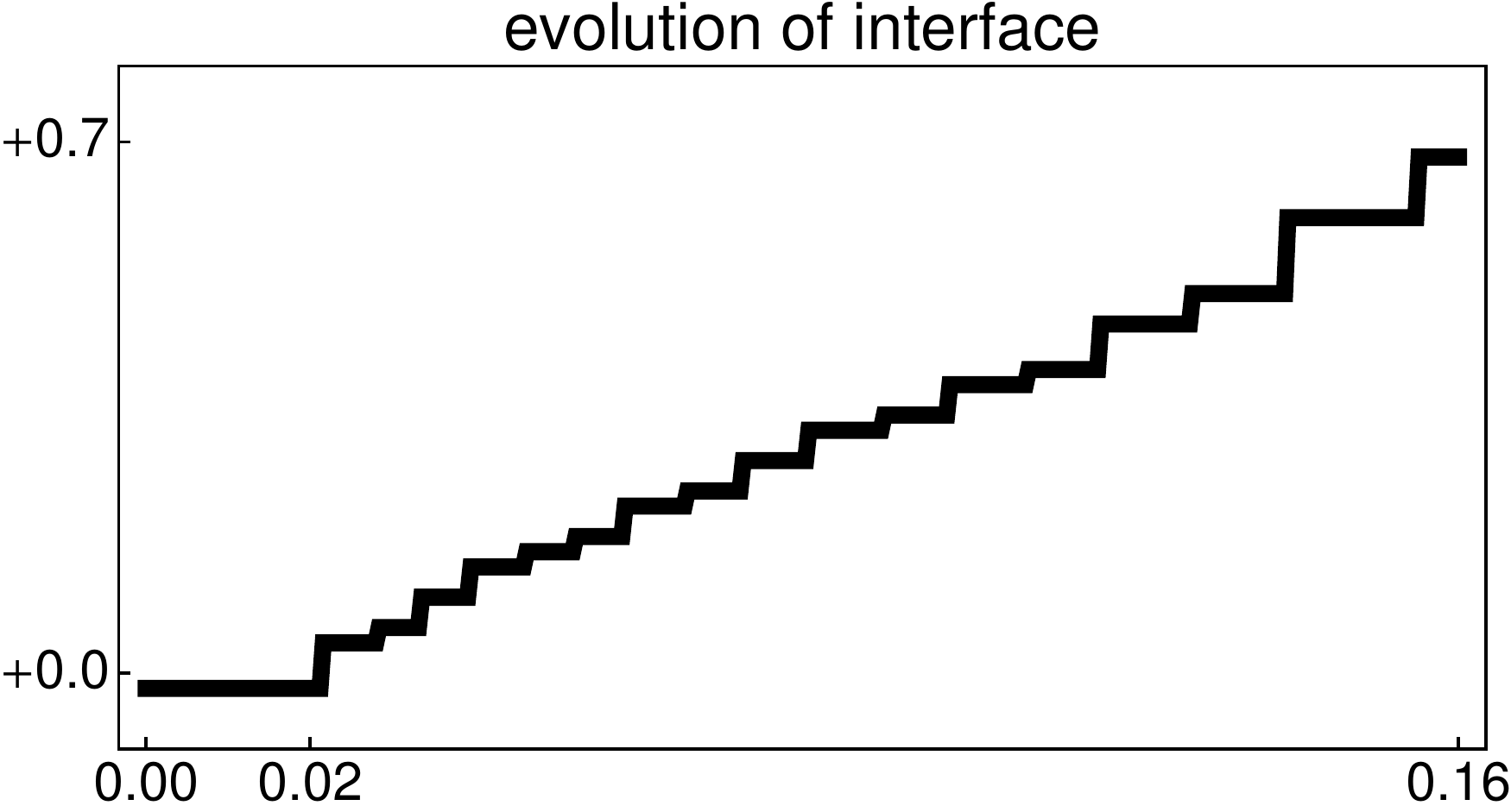}%
\\%
\includegraphics[width=\figwidth]{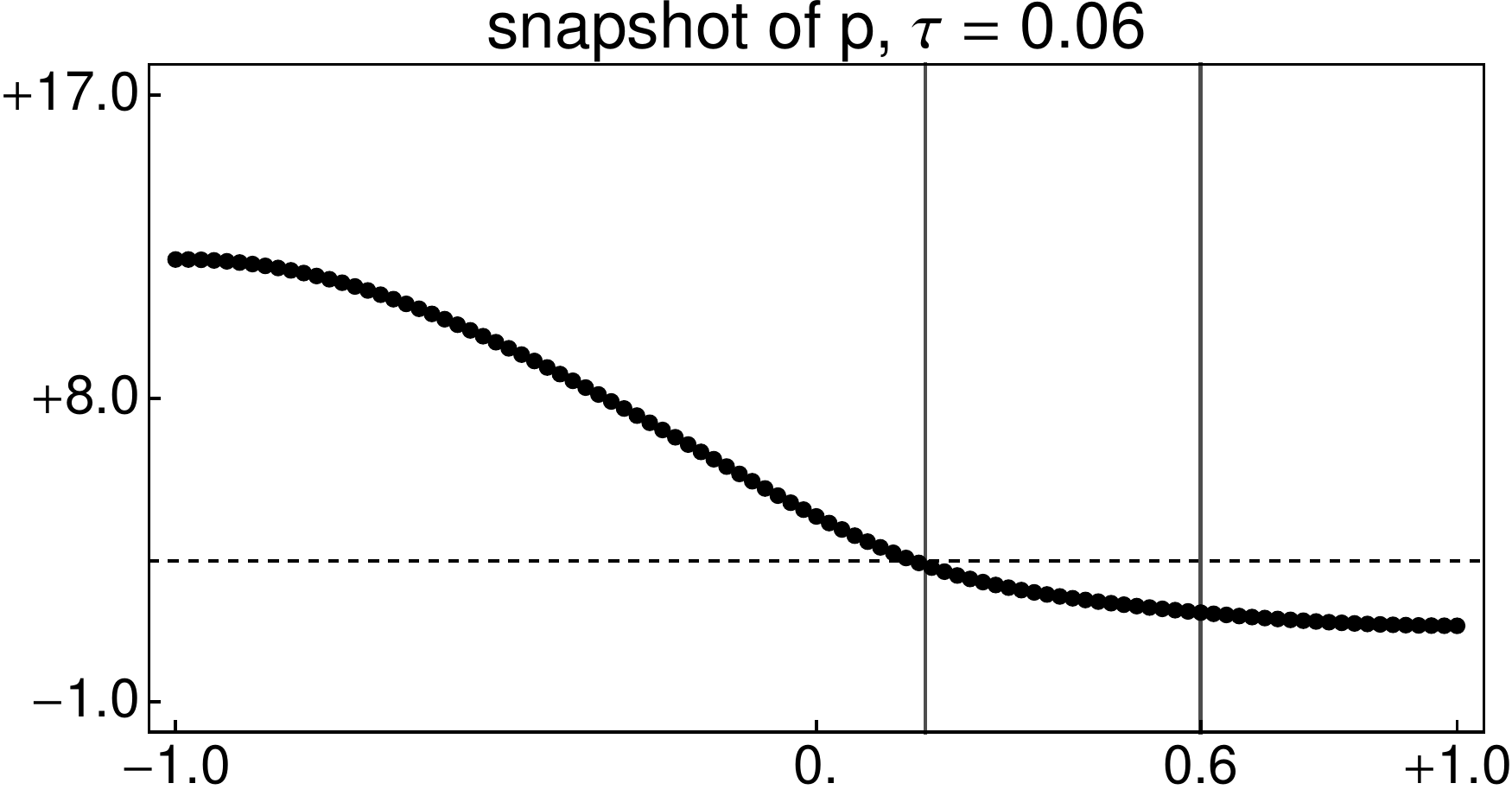}%
&\quad&%
\includegraphics[width=\figwidth]{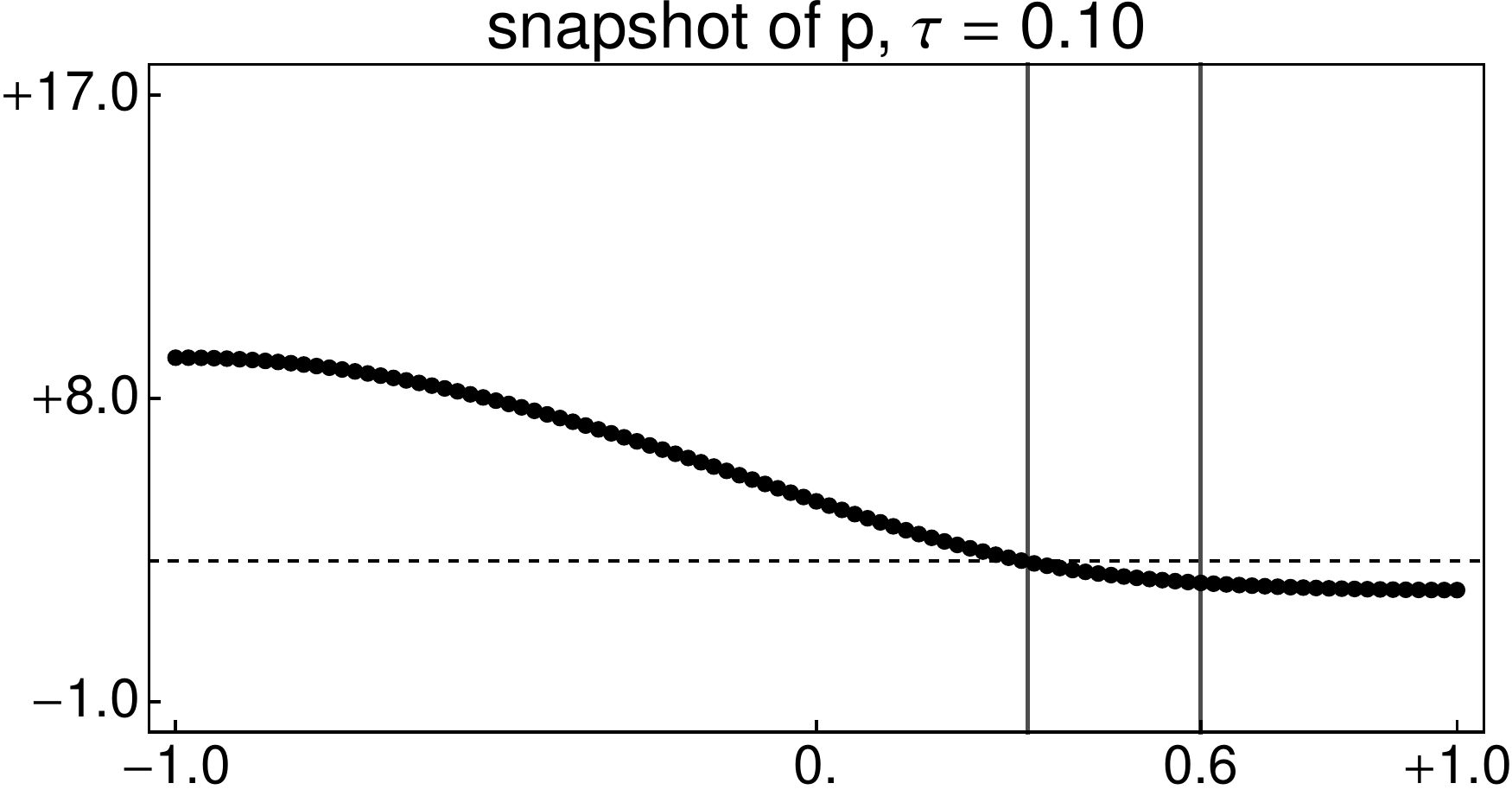}%
\end{tabular}%
}%
\caption{%
Type-II interface with $N=50$ that propagates into a region with microscopic oscillations. In  the macroscopic limit,
these oscillations cannot be described by a function $U$ but only in terms of a Young measure $\nu$, see \ref{eq:young.measure}.
}%
\label{Fig:front_4}%
\end{figure}%
\par
The last example is shown in Figure \ref{Fig:front_4} and concerns the
dynamics of type-II interfaces.  The discrete initial data $u_j\at{0}$
resemble a smooth function $U$ that penetrates the spinodal region,
and the initial transient regime therefore leads to spinodal
decomposition. This means the lattice dynamics forms a phase interface
that separates a region with $U>u_*$ from a region with strong
microscopic oscillations, where the latter can be regarded as an
approximation of a nontrivial Young measure $\nu$.  The interface is
initially at rest but depins at $\tau\approx0.02$ and propagates into
the oscillatory phase afterwards. The dynamics of type-II interfaces
are not yet well-understood. In particular, since there exist many
microscopic realizations of a given Young measure $\nu$, it is not
clear whether the macroscopic evolution can be completely
characterized in terms of the fields $P$ and $\mu\in[-1,+1]$, or
whether further macroscopic quantities are needed.  All subsequent
considerations are therefore restricted to type-I interfaces.
%
%
%
%
%
\subsection{Microscopic dynamics of phase interfaces }
\label{sec:Microscopic.Interfaces}%
%
%
To conclude our numerical investigations, we now discuss the
microscopic dynamics of moving type-I interfaces in greater detail.  In
this way we not only obtain a better understanding of the lattice
dynamics but also gain some inside into the analytical problems that
must be addressed when passing to the limit $\eps\to0$.
\par
For a moving type-I interface that is isolated -- i.\,e., sufficiently
far from any other interface and the boundary of the computational
domain -- the key numerical observations are illustrated in Figures
\ref{Fig:PhaseTransition} and \ref{Fig:FrontMotion} and can be
summarized as follows.

\begin{enumerate}
\item \emph{Sequentiality.} At each time there exists at most one
  index $k$ such that $u_k$ is inside the spinodal interval
  $[u_*,u^*]$. In other words, the interface moves because
  the $u_j$'s undergo the phase transition \emph{one after another},
  where `phase transition' means passage through the spinodal
  interval.
\item \emph{Fluctuations near the interface.} Each phase transition
  produces strong microscopic fluctuations.  These fluctuations are
  initially very localized, correspond to relatively large
  dissipation, and are spread out before the next phase transition
  occurs.
\item \emph{Temporal scale separation.} The time needed to spread out
  the fluctuations is as least as large as the time to undergo a phase
  transition. Moreover, both times are typically much smaller than the
  time between two subsequent phase transition.
\end{enumerate}
The fundamental sequentiality property can -- at least in an idealized
single-interface setting -- be derived from elementary comparison
principles for ODEs. The underlying idea is that as long as $u_k$ is
inside of $[u_*,u^*]$, the lattice equation ensures that any other
$u_j$ cannot enter the spinodal interval. For the piecewise quadratic
potential, we employ a similar argument in the proof of Theorem
\ref{thm:existence-single-interface} in order to show the persistence
of single-interface data.
\begin{figure}[ht!]%
  \centering
  \includegraphics[width=0.9\textwidth]{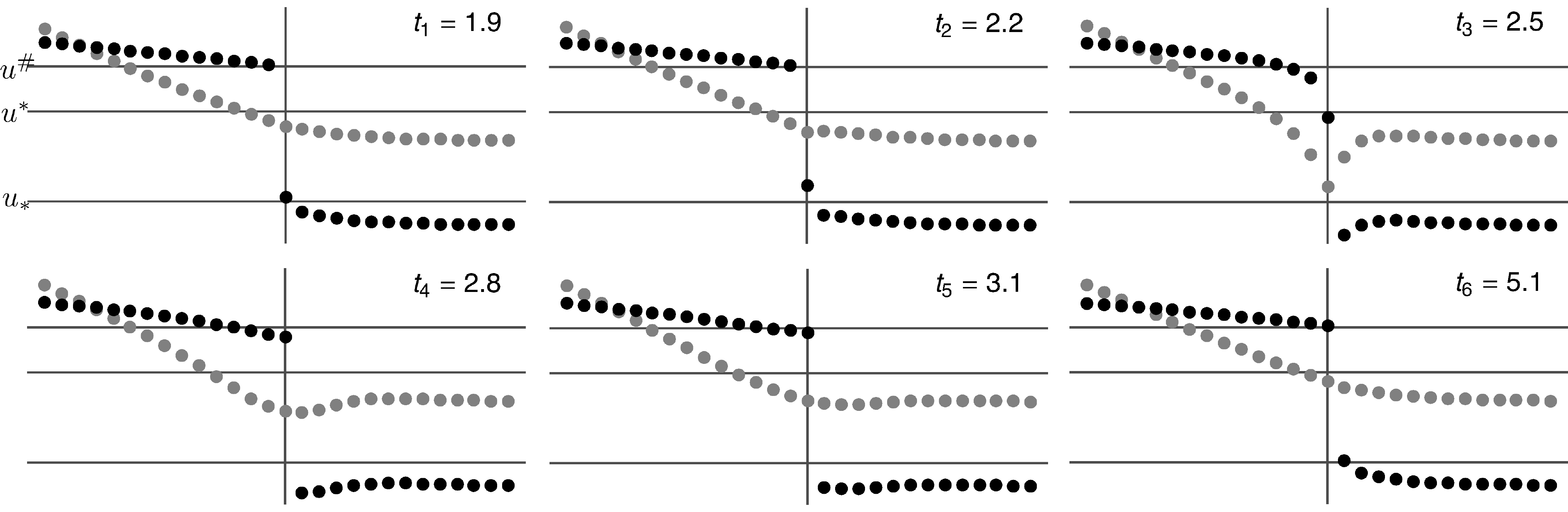}
  \caption{%
    Microscopic dynamics of a moving type-I interface: snapshots of
    $u_j\at{t}$ (black) and $p_j$ (gray, rescaled) against $j\in\Zset$
    at six non-equidistant times; the vertical line indicates the
    interface position at $j=k$.  At time $t_1$, the value $u_{k}$ has
    just crossed $u_*$ from below.  Afterwards, it passes rapidly
    through the spinodal interval $[u_*,u^*]$ and evokes
    microscopic fluctuations which are of order $1$ but localized near
    $j\approx k$. At time $t_3$, the value $u_{k}$ leaves the spinodal
    interval but the fluctuations are still present and not spread out
    before the time $t_5$.  Between $t_5$ and $t_6$, the data change
    only little but prepare $u_{k+1}$ for the next phase transition,
    that means $u_k\at{t_6}>u^*$ and $u_{k+1}\at{t_6}=u_*$.}
  \label{Fig:PhaseTransition}
\end{figure}%
\begin{figure}[ht!]%
\centering{%
\includegraphics[width=0.7\textwidth]{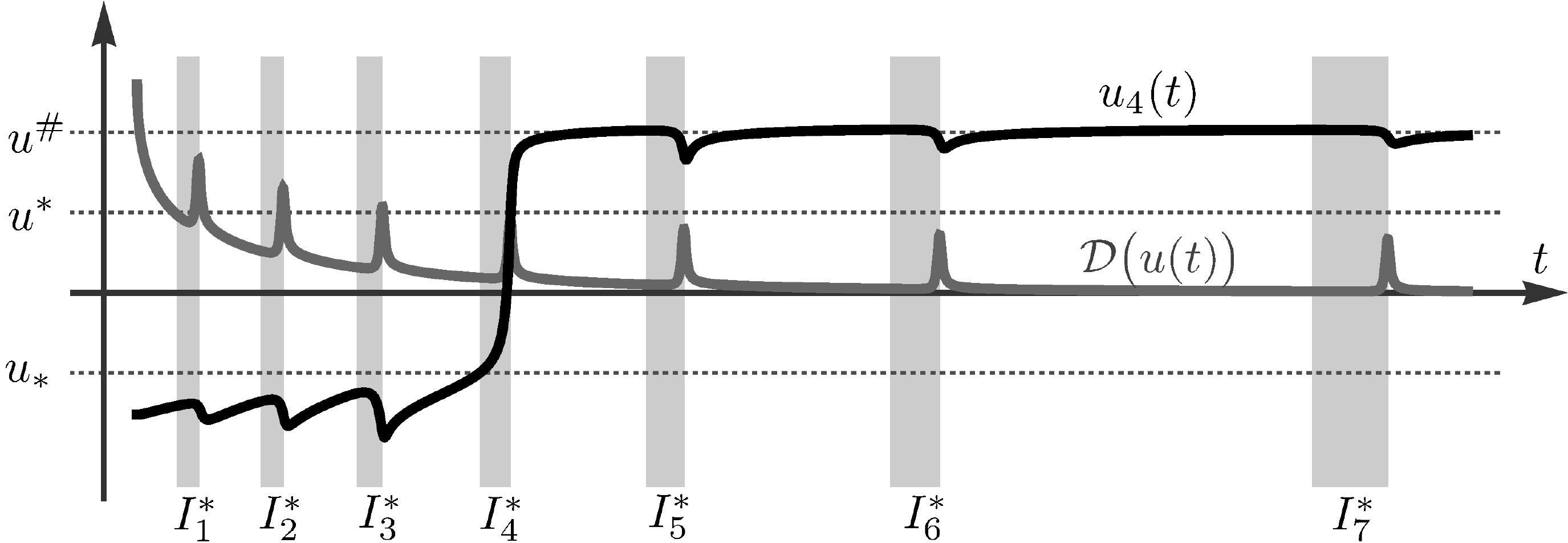}%
\\%
\includegraphics[width=0.7\textwidth]{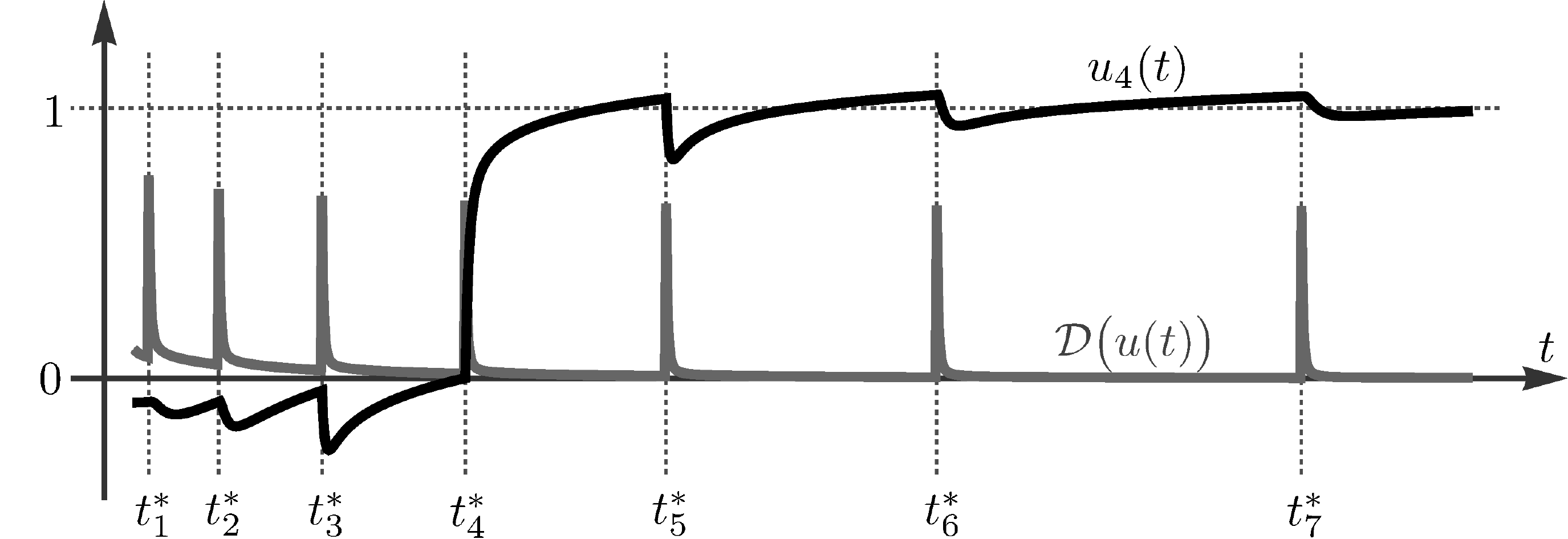}%
}%
\caption{%
Microscopic dynamics of a type-I interface that moves from $j=1$ to $j=8$.
The black curves represent the time trajectory of $u_4$, 
and the gray curves describe the (rescaled) dissipation of the system.
\emph{Top.} $\Phi$ is a smooth double well-potential and each interval $I^*_k$ contains all times
at which $u_k$ is inside of the spinodal interval. \emph{Bottom.}
For the piecewise quadratic potential, each interval $I_k^*$ 
is degenerate and consists of a single time $t_k^*$ with $u_k\at{t^*_k}=0$. 
}
\label{Fig:FrontMotion}%
\end{figure}%
\par
The microscopic fluctuations are much harder to describe rigorously. More precisely,
although it is relatively simple to understand the onset of fluctuations heuristically,
it is not obvious, at least to the authors, how to estimate
their spatial and temporal decay in the case of a generic double-well potential $\Phi$. For the piecewise
quadratic potential, however, the fluctuations can be controlled by
splitting $p=\Phi^\prime\at{u}$ into a regular part related to linear diffusion
and a sum over delayed and shifted variants of the discrete heat kernel, see 
the discussion in \S\ref{sect:ToyModel}.
\par
A further challenge for any rigorous treatment is to give a suitable
description of the different time scales in the problem. For instance,
in order to guarantee that each interface propagates with finite speed
on the macroscopic scale, one has to show that the microscopic time
between two adjacent phase transitions is of order
$\nDO{\eps^{-1}}$. Moreover, proving that any limit function $P$ is in
fact continuous at the interface requires to show that both the time
for each phase transition and the decay time of the fluctuations are
much smaller than $\nDO{\eps^{-1}}$. For piecewise quadratic $\Phi$,
the problem is again much simpler.  At first, phase transitions take
place at precise times due to the degenerate spinodal region, and
second, all other time scales can be related to the properties of the
discrete heat kernel.
%

%
\subsection{Effective evolution equations for the macroscopic dynamics}
\label{sec:Limit.Models}

We now derive the free boundary problem for the dynamics of type-I
interfaces on a heuristic level. Since our arguments are very similar
to those for the viscous approximation, we only sketch the main ideas
and refer to \cite{EvPo04} for more details.
\par
We first suppose that the lattice data $u_j$ and $p_j$ converge
strongly as $\eps\to0$ to macroscopic functions $U$ and $P$ which are
sufficiently regular.  For simplicity we also assume that there is
only a single interface located at $\xi^*\at{\tau}$ and that $U$
satisfies the \emph{phase condition}
\begin{align}
\label{eq:LD.phase.cond}
U>u^*\quad\text{for}\quad \xi<\xi_*\quad\text{an}\quad
U<u_*\quad\text{for}\quad \xi>\xi_*.
\end{align}
Under these assumptions, and using the weak formulation of \eqref{eq:Scaled.Lattice}, we readily verify that
that macroscopic evolution is governed by
\emph{bulk diffusion} via
\begin{align}
\label{eq:LD.bulk}
\partial_\tau U = \partial_\xi^2 P\quad\text{with}\quad P=\Phi^\prime\at{U}\quad\text{for all}\quad 
\tau\geq0,\;\; \xi\neq\xi^*\at\tau
\end{align}
and the \emph{generalized Stefan condition}
\begin{align}
\label{eq:LD.stefan}
\jump{P}=0\quad \text{and}\quad 
\frac{\dint\xi^*}{\dint\tau}\jump{U}+\jump{\partial_{\xi}P}=0\qquad\text{for}\quad \xi=\xi^*\at\tau,
\end{align}
which ensures that \eqref{eq:LD.bulk} holds in a distributional sense even across of the interface. 
Here $\jump{U}$ denotes as usual the jump of $U$ across the interface, this means
\begin{align*}
\jump{U}\at\tau
=U_+\at\tau- U_-\at\tau,\qquad U_\pm\at\tau:=
\lim_{h\searrow0}U\pair{\tau}{\xi^*\at\tau\pm{h}}.
\end{align*}
The nontrivial part is to identify a further dynamical interface
condition since \eqref{eq:LD.bulk} and \eqref{eq:LD.stefan} do not the
determine the evolution of $\xi^*$ completely.  In view of the
numerical results, we propose the following \emph{hysteretic flow
  rule}, see Figures \ref{fig:hysteresis} and \ref{fig:interfaces}: At
almost each time $\tau>0$, the interface is either
\begin{enumerate}
\item \emph{standing} with $\jump{P}=\jump{\partial_\xi
    P}=\frac{\dint}{\dint\tau} \xi^*=0$ and
  $P\in[p_*,p^*]$, or
\item \emph{propagating into $U<u_*$} with $P=p^*$,
  $\frac{\dint}{\dint\tau} \xi^*>0$, and $\jump{U}=u_*-u^\#<0$, or
\item \emph{propagating into $U>u^*$} with $P=p_*$,
  $\frac{\dint}{\dint\tau} \xi^*<0$, and $\jump{U}=u_\#-u^*<0$.
\end{enumerate}
Notice that the combination of phase condition, bulk diffusion, Stefan
condition, and flow rule provides -- at least on a formal level -- a
well-posed free boundary value problem and that the limit model can
easily be generalized to an interface with $U<u_*$ for $x<\xi_*$ and
$U>u^*$ for $\xi>\xi_*$, and to the case of finitely many phase
interfaces.
\par
The interface laws can also be derived in a more sophisticated way.
For the viscous approximation, it has been shown in \cite{EvPo04} --
also assuming sufficiently strong convergence as $\eps\to0$ -- that
any limit solution must satisfy the distributional \emph{entropy
  condition}
\begin{align}
  \label{eq:LD.entropy}
  \partial_\tau
  \Psi(U)-\partial_\xi\bat{\Upsilon(P)\partial_\xi P}\leq 0
\end{align}
where $\Upsilon$ is an arbitrary but increasing function and $\Psi$
defined by $\Psi^\prime\at{U}=\Upsilon\bat{\Phi^\prime\at{U}}$. This
family of local entropy inequalities can also be established in the
lattice case. In fact, the scaled equation \eqref{eq:Scaled.Lattice}
implies
\begin{align*}
\partial_{\tau} \Psi\at{U}=\Upsilon\at{P}\laplace_\eps P=
\nabla_{+\eps}\Bat{\Upsilon\at{P}\nabla_{-\eps} P}-
\bat{\nabla_{+\eps}\Upsilon\at{P}}\bat{\nabla_{+\eps} P}\leq
\nabla_{+\eps}\Bat{\Upsilon\at{P}\nabla_{-\eps} P}
\end{align*}
due to $\Upsilon^{\prime}\geq0$, and passing to $\eps\to0$ in the weak
formulation of \eqref{eq:Scaled.Lattice} we get
\eqref{eq:LD.entropy}. The key observation is that
\eqref{eq:LD.entropy} combined with \eqref{eq:LD.stefan} yields
\begin{align*}
\frac{\dint\xi^*}{\dint\tau}\Bat{\jump{\Psi\at{U}}-\Upsilon\at{P}\jump{U}}\leq 0\qquad\text{and hence}\qquad
\frac{\dint\xi^*}{\dint\tau}\at{\int_{U_-}^{U_+}\Upsilon\bat{\Phi^\prime\at{s}}-\Upsilon\at{P}\dint{s}}\leq 0,
\end{align*}
where $U_->U_+$ and $P_-=P_+=P$ hold thanks to \eqref{eq:LD.phase.cond} and \eqref{eq:LD.stefan}.
The flow rule now follows from evaluating the latter inequality for appropriately chosen functions $\Upsilon$, 
see again \cite{EvPo04} for the details.
\bigpar
In the special case that $\Phi$ is given by the piecewise quadratic
potential \eqref{eq:Intro.ToyPotential}, the bulk diffusion can be
written as
\begin{align*}
\partial_\tau P = \partial_\xi^2 P\quad\text{for all}\quad 
\tau\geq0,\;\; \xi\neq\xi^*\at\tau,
\end{align*}
whereas the Stefan condition simplifies to
\begin{align*}
\jump{P}=0\quad \text{and}\quad 
2\frac{\dint}{\dint\tau}\xi_*=\jump{\partial_{\xi}P}\qquad\text{for}\quad \xi=\xi^*\at\tau.
\end{align*}
In the space of distributions, both equations can be combined to 
\begin{align}
\label{eq:LD.simple1}
\partial_\tau \at{P+\mu} = \partial_\xi^2 P \qquad\text{for all}\quad
\tau\geq0,\;\; \xi\in\Rset,
\end{align}
where the phase field
\begin{align}
\label{eq:LD.simple2}
\mu\pair{\tau}{\xi} := \sgn\bat{\xi^*\at{\tau}-\xi} 
\end{align}
is well-defined outside of the interface and takes values in
$\{-1,+1\}$.  We mention that the hysteretic flow rule can be encoded
as
\begin{align}
\label{eq:LD.simple3}
\mu=\calM\ato{P}
\end{align}
where $\calM$ is a particular \emph{relay operator}, see for instance \cite{Vis94,SpBr96},
and that well-posedness of the initial value problem for 
\eqref{eq:LD.simple1} and \eqref{eq:LD.simple3} has been proven in
\cite{Visintin06}. For our purposes, however, 
the differential relations from above are more convenient than the 
discontinuous hysteresis operator $\calM$. In particular, 
under the sharpened phase condition 
\begin{align*}
U\pair{\tau}{\xi}\geq u_\#= -2\qquad\text{for}\qquad \xi>\xi^*\at\tau,
\end{align*}
which implies that the interface is either at rest or
propagates to the right, the flow rule is equivalent to
\begin{align}
\label{eq:LD.simple4}
P\bpair{\tau}{\xi^*\at\tau}\in[-1,+1]\quad\text{with}\quad P
\bpair{\tau}{\xi^*\at\tau}=+1\quad\text{for}\quad\frac{\dint\xi^*}{\dint\tau}\at\tau>0.
\end{align}
In \S\ref{sect:ToyModel} we prove that the simplified free boundary
value problem \eqref{eq:LD.simple1} with \eqref{eq:LD.simple2} and
\eqref{eq:LD.simple4} indeed governs the parabolic scaling limit of
\eqref{eq:lattice} with \eqref{eq:Intro.ToyPotential} provided that we
impose macroscopic single-interface initial data as described in
Assumption \ref{Ass:InitialData}.


\section{Rigorous analysis for the piecewise quadratic potential}\label{sect:ToyModel}
In this section we characterize the lattice dynamics with
piecewise quadratic potential. In particular,  from now on we suppose that
\begin{align}
  \label{eq:pw-linear-potential}
  \Phi'(u)
  =
  u-\sgn{u},
\end{align}
where we define the sign function by
\begin{align}
\label{eq:pw-linear-signum}
\sgn{u}=\left\{ \begin{array}{lcl}-1&\text{for}&u<0,\\+1&\text{for}&u\geq0.\end{array}\right.
\end{align}
The condition $\sgn{0}=+1$ is essential for establishing global
existence and uniqueness of single-interface solutions, see the proof
of Theorem \ref{thm:existence-single-interface} and the remark
afterwards.
\par
Our goal in this section is to prove $\at{i}$ that a certain class of
well-prepared microscopic initial data give rise to a single phase
interface that moves in a certain direction, and $\at{ii}$ that the
macroscopic dynamics for $\eps\to0$ can in fact be described by the
simplified free boundary problem from \S\ref{sec:Limit.Models}.
%
%
%
%
\subsection{Existence of single-interface solutions}
\label{sec:exist-single-interf}
In order to make precise what we mean by single-interface solution, we
define for each $k\in\Zset$ the state space
\begin{align*}
X_k = \Big\{u\in\ell^\infty\at\Zset\;:\quad 0
    <
    \inf_{j < k} u_j
    \leq
    \sup_{j < k} u_j
    <
    \infty,\quad
    -2
    <
    \inf_{j \geq k} u_j
    \leq
    \sup_{j \geq k} u_j
    <
    0
    \Big\},
\end{align*}
see Figure \ref{Fig:interface_data} for an illustration. As shown
below, the dynamics for initial data chosen from $X_{k}$ is as
follows: The system stays inside $X_{k}$ until some maximal time
$t^*_{k}>0$ with $\lim_{t\to t^*_{k}}u_{k}\at{t}=0$ at which $u_k$
undergoes a phase transition by crossing the spinodal value $0$ from
below. For times $t>t^*_{k}$, the system evolves inside of $X_{k+1}$
until $u_{k+1}$ changes from negative to positive phase at time
$t^*_{k+1}>t^*_{k}$.  By iteration we therefore find a phase interface
that moves to the right, where we allow for $t^*_k=\infty$ to account
for standing interfaces.
\begin{figure}[ht!]%
\centering{%
\includegraphics[width=0.9\textwidth]{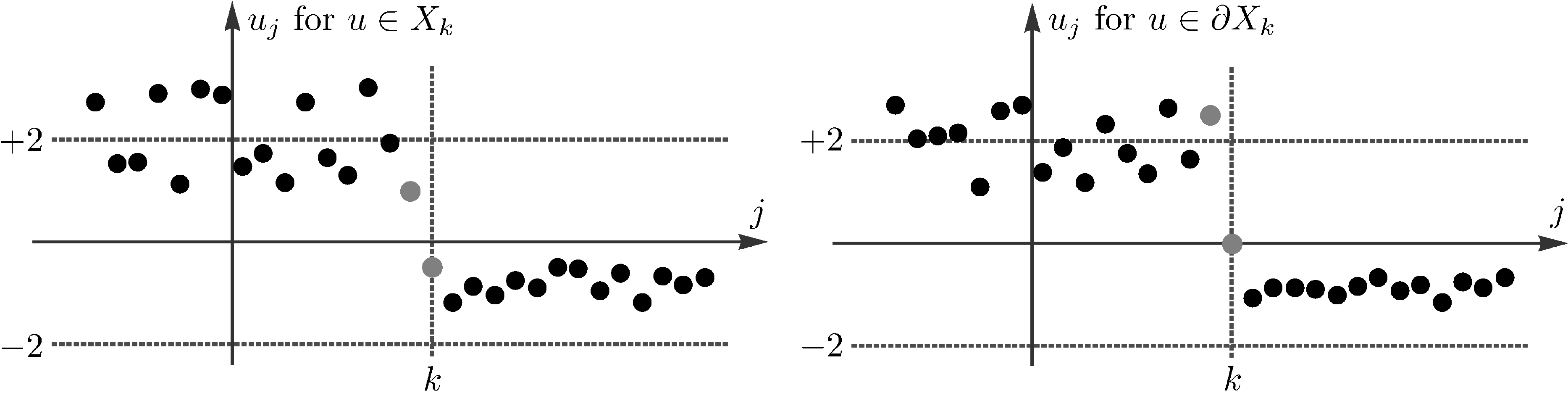}%
}%
\caption{%
  States from $X_k$ have a single phase interface at $k$ (\emph{left
    panel}, $u_{k-1}$ and $u_k$ are highlighted).  Under the dynamics,
  this interface is either stationary for all times or shifts
  eventually to the right due to a phase transition of $u_k$, because
  the system can reach the boundary $\partial X_k$ only via $u_k=0$.
  When this happens at time $t^*_k$ (\emph{right panel}), we also have
  $u_{k-1}\at{t^*_k}>2$ and $\dot{u}_k$ jumps from
  $\dot{u}_k\at{t^*_k-0}\geq0$ to
  $\dot{u}_k\at{t^*_k+0}\geq4$. Afterwards the state of the system
  belongs to $X_{k+1}$ until $u_{k+1}$ undergoes a phase transition at
  time $t^*_{k+1}$.  }%
\label{Fig:interface_data}%
\end{figure}%
\begin{definition}[single-interface solution]
  \label{def:single-interface-solution}
  A continuous function $u\colon[0,\infty) \to \ell^\infty\at\Zset$ is
  called \emph{single-interface solution} to \eqref{eq:lattice} with
  potential \eqref{eq:pw-linear-potential} if there exists
  $k_1\in\Zset$ along with a sequence $\at{t^*_k}_{k\geq
    k_1}\subset(0,\infty]$ such that the following conditions are
  satisfied with $t^*_{k_1-1}:=0$:
  \begin{enumerate}
  \item For all $k\geq k_1$, we have either $t^*_{k-1}=\infty$ or
    $t^*_{k-1}<t^*_{k}$,
  \item The function $u$ solves $\dot u_j = \laplace
    \Phi^\prime\at{u_j}$ for all times
    $t\in[0,\infty)\setminus\{t^*_k:k\geq k_1\}$,
  \item We have $u_k\at{t^*_k}=0$ and $\lim_{t\searrow
      t^*_k}\dot{u}_{k}\at{t}>0$ for all $k\geq k_1$ with
    $t^*_k<\infty$,
  \item The function $u$ takes values in $X_{k}$ on each time interval
    $(t^*_{k-1}, t^*_{k})$ with $t^*_{k-1}<\infty$.
  \end{enumerate}
  Moreover, a single-interface solution with $t^*_{k}=\infty$ for all
  $k\geq k_1$ is called a \emph{standing-interface solution}.
\end{definition}

For single-interface solutions in the sense of Definition
\ref{def:single-interface-solution}, any $u_k$ is continuous at the
phase transition time $t^*_k$, whereas $p_k$ jumps at $t^*_k$ from
$+1$ to $-1$. Notice that it is the other way around in the
macroscopic limit: when $\eps$ tends to $0$, $p$ becomes continuous
with respect to both $\tau$ and $\xi$, whereas $u$ jumps from $0$ to
$+2$ at the phase interface, see Figure \ref{Fig:FrontMotion} for an illustration.
\bigpar
As first main result we prove that single-interface solutions exist.
\begin{theorem}[existence and uniqueness of single-interface
  solutions]
  \label{thm:existence-single-interface}
  For given initial data in $X_{k_1}$, there exists a unique
  single-interface solution to \eqref{eq:lattice} with
  \eqref{eq:pw-linear-potential}. Moreover, this solution satisfies
  \begin{equation*}
    \sup_{t\in [0,\infty)} \sup_{j \in \bbZ} u_j(t)
    \leq
    D:=\max \big\{ 2, \,\sup_{j \in \bbZ} u_j\at{0} \big\},
  \end{equation*}
  and we have  
  \begin{align*}
    u_{k-1}\at{t^*_k} > 2,
    \qquad
    t^*_{k+1}-t^*_k \geq \ln\sqrt\frac{D+2}{D-2},
    \qquad
    \dot{u}_k\at{t^*_k}=\lim\limits_{t\searrow t^*_k} \dot{u}_k\at{t}=
    4+\lim\limits_{t\nearrow t^*_k}\dot{u}_k\at{t} \geq 4
  \end{align*}
  for all $k\geq k_1$ with $t^*_k<\infty$.
\end{theorem}

\begin{proof}
  We assume, without loss of generality, that $k_1=1$ and derive all
  claims by induction with respect to $k\geq k_1$.

  \par
  \underline{\emph{Dynamics inside of $X_{1}$:}}
  For states $u\in X_{1}$ we have $\sgn{u_j}=-\sgn \at{j-1}$ and hence
  \begin{align}
    \label{eqn:existence-single-interface.eqn1}
    \dot{u}_j
    =
    \laplace \Phi^\prime\at{u_j}
    =
    \laplace u_j + 2\delta _j^{0} - 2 \delta_{j}^{1},
  \end{align}
  where $\delta_j^k$ is the Kronecker delta.  In particular, the right
  hand side of \eqref{eqn:existence-single-interface.eqn1} is
  Lipschitz continuous with respect to the $\ell^\infty$-norm of $u$
  and this implies the local existence of a unique solution that is
  smooth in time and takes values in the open set $X_{1}$.  From the
  definition of $X_{1}$ we also infer that
  \begin{align*}
    \begin{array}{rccclcl}
      0 &\leq& \dot u_j\at{t} + 2u_j\at{t} &\leq& 2D 
      &\quad &\text{for } j < 1,
      \\
      0 &\leq& \dot u_j\at{t}+2u_j\at{t}&\leq& D+2
      &\quad &\text{for } j = 0,
      \\
      - 4 &\leq& \dot u_j\at{t}+2u_j\at{t}&&
      &\quad &\text{for } j = 1,
      \\
      -4 &\leq& \dot u_j \at{t}+2u_j\at{t}&\leq&0 
      &\quad &\text{for } j > 1,
    \end{array}
  \end{align*}
  and the comparison principle for scalar ODEs gives
  \begin{align*}
    \begin{array}{rccclcl}
      e^{-2 t}u_j\at{0} 
      &\leq&
      u_j(t)
      &\leq&
      e^{-2t} u_j\at{0}+\at{1-e^{-2t}} D
      &\quad &\text{for } j < 1,
      \\
      e^{-2 t}u_j\at{0} -2 \at{1-e^{-2 t}}
      &\leq&
      u_j(t)
      &&
      &\quad &\text{for } j \geq 1,
      \\
      &&
      u_j(t)
      &\leq&
      e^{-2 t}u_j\at{0}
      &\quad &\text{for } j > 1.
    \end{array}
  \end{align*}
  Using these estimates, we now conclude that the local solution with
  values in $X_{1}$ either exists for all times (in which case we set
  $t^*_{1}:=\infty$) or reaches the boundary of $X_{1}$ at some time
  $0< t^*_{1}<\infty$ via $u_{1}\at{t^*_{1}}=0$. Moreover, since both
  $u_1$ and $\dot{u}_1$ are continuous at any $t\in[0,t^*_1)$ we have
  $\lim_{t\nearrow t^*_1} \dot{u}_1\at{t}\geq0$.

  \par
  \underline{\emph{Phase transition at $t^*_{1}$:}}
  Now suppose that $t^*_{1}<\infty$. From
  \eqref{eqn:existence-single-interface.eqn1} and the above estimate
  for $u_j\at{t}$ we conclude that $u_j\at{t}$ converges for any
  $j\in\Zset$ as $t\nearrow t^*_{1}$ to some well-defined limit
  $u_j\at{t^*_{1}}$, where $u_1\at{t_1^*}=0$ as well as
  \begin{align*}
    0
    <
    \inf\limits_{j<1} u_j\at{t^*_{1}}
    \leq
    \sup\limits_{j<1} u_j\at{t^*_{1}}
    <
    D,
    \qquad
    -2
    <
    \inf\limits_{j>1} u_j\at{t^*_{1}}
    \leq
    \sup\limits_{j>1} u_j\at{t^*_{1}}
    <
    0.
  \end{align*}
  We therefore have
  \begin{align*}
    \sup\limits_{j\in\Zset}\abs{\laplace\Phi^\prime\bat{u_{j}\at{t^*_{1}}}}\leq
    {C}
  \end{align*}
  for some constant $C$. On the other hand, $u_{0}\at{t^*_1}>0$ and
  $-2<u_{2}\at{t^*_1}<0$ ensure that
  \begin{align*}
    \laplace\Phi^\prime\bat{u_1\at{t^*_1}}
    =\Bat{u_{0}\at{t^*_1}+u_{2}\at{t^*_1}}-
    \Bat{\sgn u_{0}\at{t^*_1}+\sgn u_{2} \at{t^*_1}- 2 \sgn0}>0
  \end{align*}
  thanks to $\sgn0=+1$. These results reveal that the vector field of
  the dynamical system $\dot{u}=\laplace\Phi^\prime\at{u}$ is well
  defined in the state $u\at{t^*_1}\in\partial X_1\cap \partial X_2$
  and points \emph{into} $X_2$, where $\partial X_k$ denotes the
  topological boundary of $X_k$ in $\ell^\infty\at\Zset$.  We
  therefore find a time $t_*>t^*_1$ such that the local solution to
  the microscopic dynamics from the first step has unique continuation
  to the time interval $[0,t_*)$ with $u\at{t}\in{X_2}$ for all
  $t\in(t^*_1,t_*)$. By construction, $u_j$ and $\dot{u}_j$ are
  continuous at $t^*_1$ for $j\neq1$, whereas
  \begin{align*}
    {u}_1\at{t^*_1}=\lim_{t\to t^*_1}{u}_1\at{t}=0,\qquad
    \dot{u}_1\at{t^*_1}=\lim_{t\searrow
      t^*_1}\dot{u}_1\at{t}=4+\lim_{t\nearrow
      t^*_1}\dot{u}_1\at{t}\geq4.
  \end{align*}
  thanks to $\sgn u_1\at{t^*_1}= \lim_{t\searrow t^*_1}\sgn
  u_1\at{t}=+1$ and $\lim_{t\nearrow t^*_1} u_{1}\at{t^*_1}=-1$.
  Finally, repeating all arguments from the first step we show that
  this solution in $X_2$ exists uniquely until a time $t^*_2$ with
  $t^*_1<t^*_2\leq\infty$ and $u_2\at{t^*_2}=0$ in case of
  $t^*_2<\infty$.

  \par
  \underline{\emph{Condition for $u_{0}\at{t^*_1}$ and lower bound for
      $t^*_{2}-t^*_{1}$:}}
  For $0<t<t^*_{1}$ we infer from
  \eqref{eqn:existence-single-interface.eqn1} that
  \begin{align*}
    \dot u_{1}\at{t} &= \bat{u_{0}\at{t} + u_{2}\at{t} - 2
      u_{1}(t)}-\bat{\at{+1}+\at{-1}-2\at{-1}} \leq
    u_{0}\at{t}+u_{2}\at{t}-2 u_{1}\at{t} -2.
  \end{align*}
  Passing to the limit $t\nearrow t^*_{1}$ we therefore get
  \begin{align*}
    0\leq \dot u_{1}\at{t^*_1} \leq u_{0}\at{t^*_1}+u_{2}\at{t^*_1} -2
    < u_{0}\at{t^*_{1}}-2
  \end{align*}
  and hence $u_{0}\at{t^*_{1}}>2$. In the same way we prove
  $u_{1}\at{t^*_{2}}>2$ and conclude that the interface can shift from
  $k=1$ to $k=2$ at time $t^*_2$ only if
  \begin{align*}
    2 < u_1\at{t^*_2}.
  \end{align*}
  Moreover, $u\at{t}\in X_2$ for all $t\in(t^*_1,t^*_2)$ implies
  \begin{align*}
    \dot u_{1}\at{t} &= \bat{u_{0}\at{t} + u_{2}\at{t} - 2
      u_{1}(t)}-\bat{\at{+1}+\at{-1}-2\at{+1}}\leq D + 2 - 2
    u_1\at{t}.
  \end{align*}
  By the comparison principle we therefore have
  \begin{align*}
    u_{1}\at{t} \leq \at{1-e^{-2t+2t^*_1}}\frac{D+2}{2},
  \end{align*} 
  and setting $t=t^*_2$ we obtain the estimate for $t^*_2-t^*_1$.

  \par
  \underline{\emph{Final step:}}
  All arguments derived above can easily be iterated, so the
  assertions of Theorem \ref{thm:existence-single-interface} are
  proved by induction. In particular, the lower bound for
  $t^*_{k+1}-t^*_k$ ensures that the solution exists for \emph{all}
  times $t\geq0$.
\end{proof}

\begin{samepage}
\begin{remark}
  Theorem \ref{thm:existence-single-interface} can be adapted to
  finite lattice systems with $j=1,\ldots,M$ provided that the
  corresponding ODE system is closed by imposing
\begin{enumerate}
\item either \emph{homogeneous Neumann boundary conditions} via $u_0\at{t}=u_1\at{t}$ and $u_{M+1}\at{t}=u_M\at{t}$, 
\item or \emph{inhomogeneous Dirichlet boundary conditions} via $u_0\at{t}\equiv c_1>0$ and $u_{M+1}\at{t}\equiv c_2<0$,
\end{enumerate}
where $j=0$ and $j=M+1$ are the ghost indices.
\end{remark}
\end{samepage}

The proof of Theorem \ref{thm:existence-single-interface} reveals that
the condition $\sgn 0=1$ (or equivalently, the right-sided continuity
of $\Phi^\prime$) is crucial for the microscopic dynamics to be well
defined after the phase transition time $t^*_k$ as it guarantees that
$u_k\at{t^*_k}=0$ implies the \emph{strict} inequality
$\dot{u}_k\at{t^*_k+0}>0$. Our convention \eqref{eq:pw-linear-signum}
is therefore implicitly but intimately related to phase interfaces
that propagate \emph{into} the phase $u<0$ (to observe phase
propagation into to the other phase $u>0$, one has to set
$\sgn{0}=-1$). This subtle direction selection reflects that the
spinodal region is actually an isolated point and can be regarded as
the degenerate analogue to the hysteresis flow rule \eqref{eq:Intro.Flow}.
\begin{corollary}[criterion for standing interfaces]
  For single-interface initial data
  \begin{equation*}
    u\at{0}\in X_{k_1}\quad\text{ with } \quad
    \sup_{j\in\Zset}u_j\at{0}\leq 2
  \end{equation*}
  we find $t^*_{k_1}=\infty$ and hence a standing-interface solution.
\end{corollary}
We finally characterize the evolution of $p_j\at{t}=\Phi'(u_j(t))$ for 
single-interface solutions. The key observations are as follows:
\begin{enumerate}
\item $p$ solves the linear discrete heat equation
  $\dot{p}=\laplace{p}$ on each time interval $(t^*_k,t^*_{k+1})$
  because $u\at{t}\in X_k$ implies $p_j\at{t}=u_j\at{t}+\sgn\at{j -
    k}$ and hence $\dot{p}_j\at{t}=\dot{u}_j\at{t}$.
\item 
  At each phase transition time $t^*_k<\infty$, the value of $p_k$
  jumps down according to
  \begin{align*}
    \lim\limits_{t\nearrow t^*_k}p_k\at{t}=+1,\qquad
    \lim\limits_{t\searrow t^*_k}p_k\at{t}=-1
  \end{align*}
  but we have $\lim_{t\nearrow t^*_k}p_j\at{t}=\lim_{t\searrow
    t^*_k}p_j\at{t}$ for all $j\neq{k}$.
\end{enumerate}
The dynamics of $p$ can therefore be interpreted as solving the linear
discrete heat equation for $t\geq0$ where at each phase transition
time $t^*_k$ we perturb the system by adding $-2\delta_j^k$ to the
current state. Combining this interpretation with the
superposition principle we conclude that $p$ consists of a
\emph{regular} part $q$ and a \emph{singular} part $r$, which account
for effects of the initial data and the perturbations induced by the
phase transitions, respectively.  In particular, denoting by $g$ the
discrete heat kernel, that is
\begin{equation}
  \label{eq:ap:heat-eq}
  \begin{aligned}
    \dot g_j(t) &= \laplace g_j(t), 
    \\
    g_j(0) &= \delta_j^0,
  \end{aligned}
\end{equation}
we arrive at the following result, see also Figure \ref{Fig:heat_kernel}.
\begin{corollary}[decomposition of $p$]
  \label{cor:structure-of-p}
  For each single-interface solution we have
  \begin{equation*}
    p_j(t)
    =
    q_j(t) + r_j\at{t},
  \end{equation*}
  with
  \begin{align*}
    r_j\at{t}
    :=
    - 2 \sum_{k \geq k_1} \chi_{\set{t\geq t^*_k}}(t)
    g_{j-k}(t-t^*_k),
  \end{align*}
  where $q$ is the solution of the discrete heat equation with initial
  data $q_j\at{0}=p_j\at{0}$ and $\chi_J$ denotes the indicator
  function of the set $J$.
\end{corollary}
\begin{figure}[ht!]%
\centering{%
\includegraphics[width=0.9\textwidth]{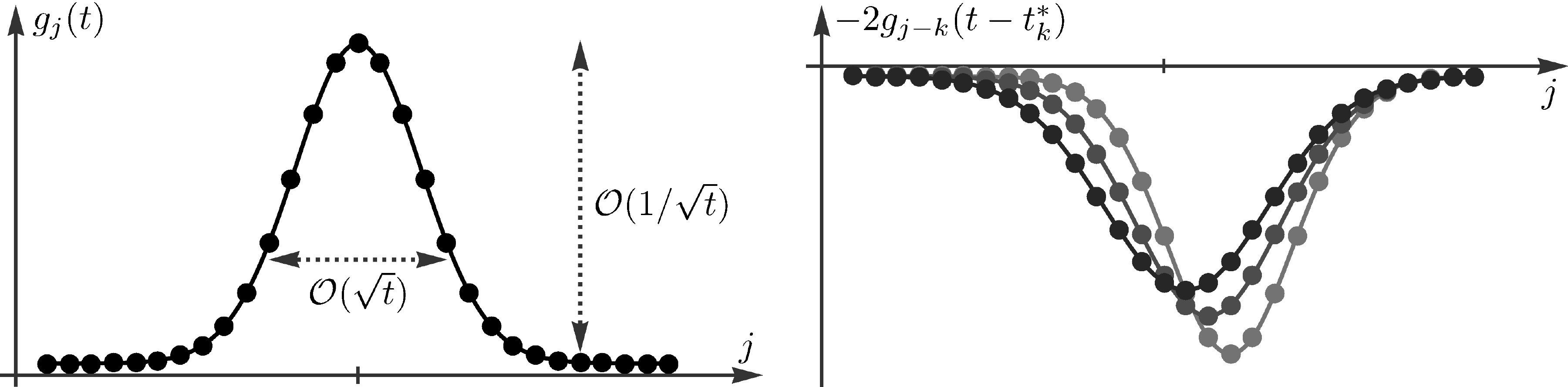}%
}%
\caption{%
 \emph{Left.} Cartoon of the discrete heat kernel $g$ (the thin lines
 represent interpolating splines and are drawn for better illustration).
 \emph{Right.} The lattice function $r$ is at any time $t$ given by a finite 
 sum of shifted and delayed version of the discrete heat kernel.
 }%
\label{Fig:heat_kernel}%
\end{figure}%

The decomposition formula from Corollary \ref{cor:structure-of-p} is
crucial for our analysis. In particular, the decay and continuity
properties of the discrete heat kernel, see appendix
\ref{sec:discrete-heat-kernel}, enable us to estimate the impact of
the sequence of singularities induced by the phase transitions. Notice
however, that the evolution of $p$ is still nonlinear as the phase
transition times $t^*_k$ depend on $p$ itself.
%
%
%
%
\subsection{Upper bounds for the macroscopic interface speed}
\label{sec:interface-speed}
In Theorem \ref{thm:existence-single-interface} we have shown that
single-interface solutions exist.  In order to pass to the macroscopic
limit, we must however guarantee that the corresponding macroscopic
interface speed is at most of order $\DO{1}$. More precisely, for a
given macroscopic distance $\delta>0$ we have to ensure that the
macroscopic time $\eps^2 t^*_{\lfloor\delta/\eps\rfloor}$ is bounded
from below by some constant independent of $\eps$, where
$\lfloor{x}\rfloor$ is shorthand for the integer part of $x$. The
derivation of this lower bound is actually at the heart of our
convergence proof and requires a careful quantitative analysis of the
lattice dynamics.  In this paper we restrict ourselves to the most
simple case and suppose that the initial data are compatible with the
macroscopic limit model.  We also assume without loss of generality
that the macroscopic interface is initially located at $\xi=0$.
\begin{assumption}[macroscopic single-interface initial data]
\label{Ass:InitialData}
The initial data $u\at{0}$ belong to $X_{1}$ and $p\at{0}=\Phi^\prime\at{u\at{0}}$ satisfies
\begin{align*}
\abs{\laplace p_0\at{0}}\leq\beta\eps, \qquad
\sup\limits_{j\in\Zset}\abs{\nabla_+ p_j\at{0}}\leq\alpha\eps,\qquad
\sup\limits_{j\neq 0}\abs{\laplace p_j\at{0}}\leq\alpha\eps^2.
\end{align*}
for two constants $\alpha$ and $\beta$ independent of $\eps$.
\end{assumption}
Assumption \ref{Ass:InitialData} is motivated by the limit dynamics,
see also Figure \ref{Fig:initial_data} for an illustration. In fact,
the prototypical example for macroscopic single-interface data is
\begin{align}
  \label{eq:example-initial-data}
  p_j\at{t} = P_\ini\at{\eps j}  + c_\eps,\qquad
  u_j\at{0} =  p_j\at{t} 
  + \sgn\at{-j}
\end{align}
where $c_\eps$ is a constant and $P_\ini$ some given macroscopic
function independent of $\eps$ such that:
\begin{enumerate}
\item $P_\ini$ is bounded and continuous for all $\xi\in\Rset$ with
  \begin{align*}
    P_\ini\at\xi > 1\quad\text{for}\quad \xi<0
    \qquad\text{and}\quad
    -1<P_\ini\at\xi<1 \quad\text{for}\quad \xi>0
  \end{align*}
  and
  \begin{align*}
    \liminf_{\xi\to-\infty} P_\ini\at{\xi} > 1
    \qquad\text{and}\qquad 
    -1 < \liminf_{\xi\to+\infty} P_\ini\at{\xi} \leq
    \limsup_{\xi\to+\infty} P_\ini\at{\xi} < +1.
  \end{align*}
\item $P_\ini$ is twice continuously differentiable for $\xi<0$ and
  $\xi>0$ with
  \begin{align*}
    \beta := \lim \limits_{\xi\searrow0}
    \abs{P_\ini^\prime\at{+\xi}- P_\ini^\prime\at{-\xi}}<\infty,
    \qquad
    \alpha := \sup\limits_{\xi\neq0}
    \abs{P_\ini^{\prime}\at\xi} + \abs{P_\ini^{\prime\prime}\at\xi}<\infty,
  \end{align*}
\item $c_\eps>0$ is sufficiently small with $c_\eps\to0$ as $\eps\to0$
  and chosen such that $t^*_1>0$, that means $u_j\at{0}> 0$ for all
  $j\leq 0$ and $u_j\at{0}<0$ for all $j\geq 1$.
\end{enumerate}

\begin{corollary}[bounds for the regular part $q$]
  \label{cor:decay-regular-velocities}
  We have
  \begin{align*}
    \abs{\nabla_+ q_j\at{t}}\leq \al\eps \qquad\text{and}\qquad
    \nabs{\dot{q}_j\at{t}}=\abs{\laplace q_j\at{t}}\leq
    \al\eps^2+\beta \eps g_0\at{t}
  \end{align*}
  for all $t\geq0$, $j\in\Zset$, and $\eps>0$.
\end{corollary}

\begin{proof}
  By construction and Assumption \ref{Ass:InitialData} we have
  \begin{align*}
    \tfrac{\dint}{\dint{t}}\dot{q}_j\at{t}
    = \laplace \dot{q}_j\at{t},
    \qquad
    \nabs{\dot{q}_j\at{0}} = \abs{\dot{p}_j\at{0}} = \abs{\dot{u}_j\at{0}}
    = \abs{\laplace{p}_j\at{0}} \leq \alpha\eps^2+\beta\eps\delta_j^{0}
  \end{align*} 
  as well as 
  \begin{align*}
    \tfrac{\dint}{\dint{t}}\nabla_+{q}_j\at{t}
    = \laplace \nabla_+{q}_j\at{t},
    \qquad  
    \abs{\nabla_+{q}_j\at{0}} = \abs{\nabla_+{p}_j\at{0}} \leq \alpha \eps.
\end{align*} 
The claim now follows using both the superposition and the
maximum principle for the discrete 
heat equation.
\end{proof}
We remark that the assertions of Corollary \ref{cor:decay-regular-velocities} are sufficient for
showing that the macroscopic interface speed is bounded. All results derived below therefore 
hold (with different constants) even in the case that
\begin{enumerate}
\item 
$\abs{\nabla_+{p}_j\at{0}}\leq \alpha \eps$ and $\abs{\laplace p_j\at{0}}\leq \beta\eps$ for all $j\in\Zset$, 
\item 
$\abs{\laplace p_j\at{0}}\leq \alpha\eps^2$ for almost all $j\in\Zset$,
\end{enumerate}
that means the derivative of the function $P_\ini$ from
\eqref{eq:example-initial-data} can even be discontinuous at finitely
many points.
\begin{figure}[ht!]%
\centering{%
\includegraphics[width=0.9\textwidth]{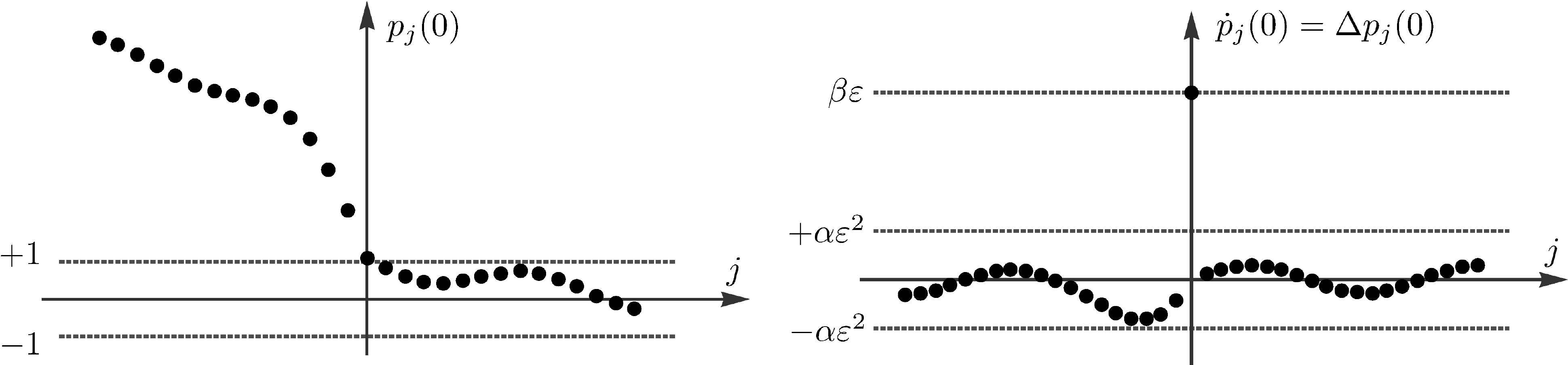}%
}%
\caption{%
  For initial data as in Assumption \ref{Ass:InitialData}, the
  discrete data $p_j\at{0}$ resemble a macroscopic function $P_\ini$
  that is continuous and piecewise twice continuously
  differentiable. In particular, the initial velocities
  $\dot{u}_j\at{0}=\dot{p}_j\at{0}$ satisfy
  $\dot{p}_0\at{0}=\beta\eps$ and $\abs{\dot{p}_j\at{0}}\leq\al\eps^2$
  for $j\neq0$.  }%
\label{Fig:initial_data}%
\end{figure}%
\bigpar
We are now able to derive upper bounds for
the macroscopic interface speed. To this end we prove that Assumption \ref{Ass:InitialData}
implies that the time $t^*_{k+1}-t^*_k$ between two adjacent phase transitions is mesoscopic as it can be bounded from below by $1/\eps$ (recall that Theorem \ref{thm:existence-single-interface} provides only $t^*_{k+1}-t^*_k\geq 2/\at{D-2}$). Our considerations are based on the estimate
\begin{align}
  \label{eq:neighbour-travel-time}
  2
  \leq p_k\at{t^*_{k+1}}-p_{k}\at{t^*_k}=
  \int_{t^*_k}^{t^*_{k+1}} \dot p_k(t) \dint{t}
  =
  \int_{t^*_k}^{t^*_{k+1}} \dot q_k(t) \dint{t}
  - 2 \sum_{n=1}^k \int_{t^*_k}^{t^*_{k+1}} \dot g_{k-n}(t-t^*_n)\dint{t},
\end{align}
which follows from combining the conditions $u_k\at{t^*_k}=0$ and $u_{k}\nat{t^*_{k+1}}\geq 2$  
with the representation formula from Corollary~\ref{cor:structure-of-p}.
\begin{lemma}[refined lower bound for the time between two phase
  transitions]
  \label{lem:estimate-m_k}
  For any $\tau_\fin>0$ there is a constant $d_*>0$, which depends on
  $\alpha$, $\beta$, and $\tau_\fin$, along with a constant
  $0<\eps_*<1$, which depends only on $\alpha$ and $\beta$, such that
\begin{equation*}
    \eps \at{t^*_{k+1}-t^*_k}
    \geq
    2d_*
\end{equation*}
holds for all $k\geq 1$ with $0 \leq t^*_k \leq \tau_\fin/\eps^2$ provided that 
$\eps\leq\eps_*$.
\end{lemma}
\begin{proof}
  In what follows we consider $k\geq1$ with $t^*_k<\tau_\fin/\eps^2$
  and denote by $c$ and $C$ generic constants independent of $\al$,
  $\beta$, and $\eps$.
\par
\underline{\emph{Weaker variant of \eqref{eq:neighbour-travel-time}}}:
In order to study the implications of
\eqref{eq:neighbour-travel-time}, we first simplify the right hand
side as follows. By Lemma \ref{lem:discr-heat-kernel} we have
\begin{align*}
- \sum_{n=1}^k \int_{t^*_k}^{t^*_{k+1}} \dot g_{k-n}(t-t^*_n)\dint{t}
&\leq
-\sum_{n=1}^k\int_{t^*_k}^{t^*_{k+1}}\dot g_{0}(t-t^*_n) \dint{t}
=\sum_{n=1}^k g_0\at{t^*_k-t^*_n}-\sum_{n=1}^k g_0\at{t^*_{k+1}-t^*_n}.
\end{align*}
whereas Corollary \ref{cor:decay-regular-velocities} provides
\begin{align*}
  \int_{t^*_k}^{t^*_{k+1}} \dot q_{k}(t)\dint{t}
  \leq
  \int_{t^*_k}^{t^*_{k+1}}\bat{\alpha\eps^2+\beta\eps
    g_0\at{t}}\dint{t}
\end{align*}
Rearranging terms and using $g_0(0) = 1$, inequality
\eqref{eq:neighbour-travel-time} becomes
\begin{equation*}
  \sum_{n=1}^k g_0(t^*_{k+1}-t^*_n)
  \leq
  \frac{1}{2}
  \int_{t^*_k}^{t^*_{k+1}} \at{\alpha\eps^2 + \beta\eps g_0(t)} \dint{t}
  +
  \sum_{n=1}^{k-1} g_0(t^*_k-t^*_n),
\end{equation*}
and writing $s_{k} := t^*_{k+1} - t^*_k$ we arrive at
\begin{equation}
  \label{eq:trans-est-m}
  \sum_{n=1}^{k} g_0(s_{k} + \cdots + s_n)
  \leq
  \frac{1}{2}
  \int_{t^*_k}^{t^*_{k}+s_k}
  \at{\alpha\eps^2 + \beta\eps g_0(t) } \dint{t}
  +
  \sum_{n=1}^{k-1} g_0(s_{k-1} + \cdots + s_n).
\end{equation}
In the remainder of this proof we transform this inequality into a
lower bound for $s_k$.
\par
\underline{\emph{Estimate for $k=1$}}:
Inequality \eqref{eq:trans-est-m} combined with Lemma \ref{lem:discr-heat-kernel-monotonicty} provides
\begin{equation*}
    g_0(s_1)
    \leq
    \tfrac{1}{2}\int_{t^*_1}^{t^*_1+s_1}
    \at{ \alpha\eps^2 + \beta\eps g_0(t) } \dint{t}
    \leq
    \tfrac{1}{2}\int_{0}^{s_1} \at{ \alpha\eps^2 + \beta\eps g_0(t) } \dint{t},
\end{equation*}
and we deduce that there exists $\eps_*>0$, which depends on $\alpha$
and $\beta$, such that $s_1\geq 1$ for all $\eps\leq \eps_*$.  We
employ Lemma \ref{lem:discr-heat-kernel} again to estimate
\begin{equation*}
    \frac{c}{\sqrt{s_1}}\leq {C}\bat{\alpha\eps^2 + \beta\eps\sqrt{s_1}},
 \end{equation*}
and this gives
\begin{equation*}
  \eps s_1 \geq d_1 > 0
\end{equation*}
for some constant $d_1$ which depends on $\alpha$ and $\beta$ but not
on $\eps$. For $\beta=0$ we even find $\eps^{4/3} s_1\geq d_1$.
\par
\underline{\emph{Estimate for $k>1$ in a special case}}:  
For the following considerations we further suppose that
\begin{align*}    
    s_k < \min\limits_{n=1,\ldots,k-1} s_n.
\end{align*}
Since $g_0$ is strictly decreasing, we therefore get
\begin{equation*}
    g_0(s_{k} + \cdots + s_{n+1})
    >
    g_0(s_{k-1} + \cdots + s_{n})
\end{equation*}
for all $n = 1,\ldots,k-1$, and hence
\begin{equation*}
  g_0(s_k + \cdots + s_1)
    <
    \frac{1}{2}
    \int_{t^*_k}^{t^*_{k}+s_k}
    \at{\alpha\eps^2 + \beta\eps g_0(t)} \dint{t},
 \end{equation*}
 thanks to \eqref{eq:trans-est-m}. This implies $f_k\at{s_k}<0$ for
\begin{equation*}
   f_{k}\at{s}:=
    g_0(s+t^*_{k}) - \frac{1}{2}
    \int_{t^*_k}^{t^*_k+s} \at{\alpha\eps^2 + \beta\eps g_0(t)} \dint{t},
\end{equation*}
where we used that $g_0(s_k + \cdots + s_1)=g_0(s_k + t^*_k-t^*_1)\geq
g_0(s_k + t^*_k)$. The properties of $g_0$, see Lemma
\ref{lem:discr-heat-kernel-monotonicty} once more, guarantee that the
function $f_k$ is convex, continuous, and strictly decreasing in
$s$. Therefore, and in view of $f_k\at{0}>0$ and $\lim_{s\to\infty}
f_k\at{s}=-\infty$, we conclude that $f_k$ has a unique zero
$\bar{s}_k$ with $\bar{s}_k<s_k$. Due to convexity of $f_k$ we also
find
\begin{align*}
  \hat{s}_k<\bar{s}_k<s_k,
\end{align*}
where $\hat{s}_k:=- f_k\at{0}/f_k^\prime\at{0}$ is the first
approximation to $\bar{s}_k$ when starting the Newton algorithm at
$s=0$. By direct computations we now verify
\begin{equation*}
    \hat{s}_k
    =
    \frac{g_0\at{t^*_k}}{- \dot g_0\at{t^*_k}
      + \frac{\alpha\eps^2}{2} + \frac{\beta\eps}{2} g_0(t^*_k)},
\end{equation*}
and using $t^*_k \geq s_1 \geq 1$ as well as Lemma
\ref{lem:discr-heat-kernel} we obtain
\begin{align*}
    \hat{s}_k
        \geq
    \frac{c \at{t^*_k}^{-1/2}}
    {C \at{t^*_k}^{-3/2} + \alpha\eps^2
      + {C} \beta\eps \at{t^*_k}^{-1/2}}
    \geq
    \frac{c}{\at{t^*_k}^{-1}
      + \alpha\eps^2 \at{t^*_k}^{1/2}
      + \beta\eps},
\end{align*}
and hence
\begin{align*}
  \eps s_k\geq \frac{c }{d_1^{-1}
      + \alpha \tau_\fin^{1/2}
      + \beta} =: d_2
\end{align*}
where we used that $t^*_k\geq s_1\geq d_1/\eps$ and 
$t^*_k\leq \tau_\fin/\eps^2$. 
\par
\underline{\emph{Estimate for $k>1$ in the general case}}:  
We have established the estimate $\eps s_1\geq d_1$ as well as the
implication
\begin{align*}
  \eps s_k \leq \min\{\eps s_1,\ldots,\eps s_{k-1}\}\quad\implies
  \quad \eps s_k\geq d_2,
\end{align*}
and the desired estimate for $s_k$ follows with
$d_*:=\tfrac12\min\{d_1, d_2\}$ by induction.
\end{proof}
From Lemma~\ref{lem:estimate-m_k} we immediately obtain 
$t^*_k\geq 2kd_*/\eps$ and we deduce for each
macroscopic time $\tau_\fin$ that at most $\tau_\fin/\at{2\eps d_*}$ phase transitions 
can happen for $\tau\leq\tau_\fin$, shifting the interface over a macroscopic distance smaller than
$\tau_\fin/\at{2 d_*}$.
\par
We conclude this section with some comments concerning the
microscopic fluctuations caused by the phase transition of $u_k$ at
$t^*_k$.  The properties of the discrete heat kernel imply that the
amplitude as well the inverse length of the effective spatial support
scale with
\begin{align*}
\frac{\eps}{\sqrt{\eps^2+\at{\tau-\eps^2t^*_k}}},
\end{align*}
which decays quite rapidly within a macroscopic time of order
$\eps^2$ but much slower afterwards.  In particular, when
$u_{k+1}$ undergoes the next phase transition at time $t^*_{k+1}$, the
fluctuations evoked by $u_k$ have reached an amplitude of order
$\eps^{1/2}$ and spread over a macroscopic length of order
$\eps^{-1/2}$. Similarly, the amplitude of the velocity
fluctuations at time $t^*_{k+1}$ is of order $\eps^{3/2}$.
\par
These scaling arguments, especially the fractional powers of $\eps$,
reveal that \emph{macroscopic} single-interface data are \emph{not}
invariant under the dynamics. In other words, the lattice data at
times $t \lessapprox t^*_k$ -- this means just before the phase
transitions -- do \emph{not} satisfy Assumption \ref{Ass:InitialData},
and we conclude that interface propagation in discrete
forward-backward diffusion equations is not only a two-scale but a
genuine multi-scale problem.

%
%
%
%
\subsection{Macroscopic continuity and compactness results}
\label{sec:macr-cont-comp}
In order to pass to the macroscopic limit $\eps\to0$, we regard the
discrete data $p_j\at{t}$, $q_j\at{t}$, and $r_j\at{t}$ as piecewise
constant functions with respect to the macroscopic coordinates
$(\tau,\xi)$.  More precisely, we set
\begin{align*}
P_\eps\bpair{\tau}{\xi+\zeta}:= p_{\xi/\eps}\at{\tau/\eps^2}\qquad\text{for all}\quad
\tau\geq0, \quad \xi\in\eps\Zset, \quad\zeta\in[-\eps/2,+\eps/2)
\end{align*}
and define $Q_\eps$ and $R_\eps$ by analogous formulas. We further
introduce the macroscopic interface position
\begin{align*}
\xi^*_\eps\at{\tau}:=\eps \sum_{k=1}^\infty {k} \chi_{[t^*_k,t^*_{k+1})}\at{\tau/\eps^2}\qquad
\text{for all}\quad\tau\geq0,
\end{align*}
as a piecewise constant function in time that jumps at the macroscopic
phase transitions times defined by
\begin{align*}
\tau^*_{k,\,\eps}:=\eps^2 t^*_{k}\qquad\text{for all}\quad k\in\Zset.
\end{align*}
In what follows we fix a macroscopic final time $0<\tau_\fin<\infty$
and wish to pass to the limit $\eps \to 0$ on the macroscopic
time-space domain
\begin{align*}
\Om:=I\times\Rset,\qquad  I:=[0,\tau_\fin].
\end{align*} 
To this end, we recall that Lemma \ref{lem:estimate-m_k} provides
constants $d_*>0$ and $0<\eps_*\leq1$ such that
\begin{align}
  \label{eq:props-keps}
  K_\eps \leq \frac{\tau_\fin}{2d_*\eps},
  \qquad
  \inf\limits_{k=1,\ldots,K_\eps}
  \tau^*_{k+1,\,\eps}-\tau^*_{k,\,\eps} \geq 2 d_*\eps
\end{align}
holds for all $0<\eps\leq\eps_*$, where 
\begin{equation*}
  K_\eps:=\max\Big\{k\in\Zset\;:\; \tau^*_{k,\,\eps} <\tau_\fin \Big\}
\end{equation*}
is the number phase transitions taking place in the microscopic time
interval $[0,\tau_\fin/\eps^2]$.  Notice that $d_*$, $\eps_*$, and all
constants derived below depend on $\tau_\fin$ and on the initial data
via $\alpha$ and $\beta$.
\bigpar
Our first result in this section concerns the compactness of the
discrete interface curves $\xi^*_\eps$.
\begin{lemma}[compactness of interface curves]
  \label{lem:compactness-xi}
  The family $\at{\xi_*}_{0<\eps\leq \eps_*}$ is compact in
  $\fspaceL^\infty\at{I}$ and each limit curve is Lipschitz
  continuous.
\end{lemma}

\begin{proof}
  We define a piecewise linear function $\bar{\xi}^*_\eps$ on
  $[0,\tau_\fin]$ by the condition
  \begin{align*}
    \bar{\xi}^*_\eps\at{\tau^*_{k,\,\eps}} =
    \xi^*_\eps\at{\tau^*_{k,\,\eps}}
  \end{align*}
  for all $ k=1,\ldots,K_\eps$ and $\bar{\xi}^*_\eps\at{\tau_\fin} =
  \xi^*_\eps\nat{\tau^*_{K_\eps,\,\eps}}$.  We readily check that
  \begin{align*}
    \babs{\bar{\xi}^*_\eps\at{\tau}-\xi^*_\eps\at{\tau}}\leq
    \eps,\qquad 0\leq \tdiff{\tau} \bar{\xi}^*_\eps\at{\tau}\leq
    \frac{\eps}{2d_*\eps}
  \end{align*}
  for almost all $\tau\in[0,\tau_\fin]$, and conclude that the family
  $\at{\bar{\xi}^*_\eps}_{0<\eps\leq\eps_*}$ is bounded in
  $\fspaceC^{0,1}(I)$. All claims hence follow from standard results
  in real analysis.
\end{proof}
Our main goal this section is to derive compactness result for
$R_\eps$ and $Q_\eps$ that imply $\at{i}$ the existence of pointwise
convergent subsequences, and $\at{ii}$ the continuity of any limit
function. In a preparatory step we next derive an auxiliary result for
piecewise constant functions $F_\eps$ on $\Om$ revealing that for
$\fspaceL^\infty$-compactness it is sufficient to establish uniform
H\"older estimates with respect to $\tau\in{I}$ and discrete
$\xi\in\eps\Zset$. Here piecewise constant means, as above, continuous
with respect to $\tau$ but spatially constant in each interval
$\xi\in[\eps j-\eps/2,\, \eps j +\eps/2)$, $j\in\Zset$.  Our auxiliary
result and its proof are straight forward and can very likely be found
somewhere in the literature on numerical analysis (though we are not
aware of any reference).
\begin{lemma}[compactness criterion for piecewise constant functions]
\label{lem:compactness-aux-res}
Let $\at{F_\eps}_{0<\eps\leq\eps_*}$ be a family of bounded and 
piecewise constant functions on $\Om$, and suppose that there
exist constants $\ga_1,\ga_2\in (0,1]$ and $C>0$ 
such that
\begin{align*}
\babs{F_\eps\pair{\tau_2}{\xi_2}-F_\eps\pair{\tau_1}{\xi_1}}\leq {C}\bat{\babs{\tau_2-\tau_1}^{\ga_1}+\babs{\xi_2-\xi_1}^{\ga_2}}
\end{align*}
holds for all $0<\eps\leq\eps_*$, every $\tau_1, \tau_2\in{I}$, and all $\xi_1,\xi_2\in\eps\Zset$. 
Then this family is compact in $\fspaceL^\infty\at\Om$ and any limit function is locally H\"older continuous
with exponent $\min\{\ga_1,\ga_2\}$.
\end{lemma}
\begin{proof}
For each $\eps$ we define a piecewise linear function $\bar{F}_\eps$ by
\begin{align*}
\bar{F}_\eps\pair{\tau}{\bar\xi}=F_\eps\pair{\tau}{\bar\xi}\qquad\text{for all}\quad
\tau\in{I},\quad\bar\xi\in\eps\Zset,
\end{align*}
and our assumptions yield the H\"older continuity
of $\bar{F}_\eps$ with respect to time, that means
\begin{align*}
\babs{\bar{F}_\eps\pair{\tau_2}{\xi}-\bar{F}_\eps\pair{\tau_1}{\xi}}\leq C\abs{\tau_2-\tau_1}^{\ga_1}
\end{align*}
for all $\xi\in\Rset$ and $\tau_1,\tau_2\in{I}$. Moreover, since 
$\bar{F}_\eps$ is piecewise linear with respect to $\xi$ and due to
\begin{align*}
\babs{F_\eps\pair{\tau}{\bar\xi}}\leq{C},\qquad
\abs{\frac{F_\eps\pair{\tau}{\bar\xi\pm\eps}-F_\eps\pair{\tau}{\bar\xi}}{\eps}}\leq {C} \eps^{\ga_2-1}
\end{align*}
we verify
-- discussing the cases $\sgn{\zeta_1}=\sgn{\zeta_2}$ and 
$\sgn{\zeta_1}\neq\sgn{\zeta_2}$ separately -- the estimate
\begin{align}
\label{lem:compactness-aux-res.eqn1}
\babs{\bar{F}_\eps\pair{\tau}{\bar\xi+\zeta_2}-\bar{F}_\eps\pair{\tau}{\bar\xi+\zeta_1}}\leq
 C\eps^{\ga_2-1}\abs{\zeta_2-\zeta_1}\leq C\abs{\zeta_2-\zeta_1}^{\ga_2}
\end{align}
for all $\tau\in {I}$, $\bar\xi\in\eps\Zset$, and all $\zeta_1,\zeta_2
\in [-\eps,+\eps]$.  In particular, setting $\zeta_1=0$ and taking the
supremum over $\bar\xi$ and $\zeta_2$ we obtain the convergence
estimate
\begin{align*}
\norm{F_\eps-\bar{F}_\eps}_{\fspaceL^\infty\at\Om}\leq C\eps^{\ga_2}.
\end{align*}
We next show that $\bar{F}_\eps$ is H\"older continuous with respect
to $\xi$.  To this end, let $\tau\in{I}$ and $\xi_1,\xi_2\in\Rset$ be
given and denote by $\bar\xi_i$ the natural projection of $\xi_i$ to
$\eps\Zset$, that means $\abs{\bar\xi_i-\xi_i}\leq\eps/2$.  For
$\abs{\xi_1-\xi_2}\leq \eps/2$ we have $\xi_1,\xi_2\in
[\bar\xi_1-\eps, \bar\xi_1+\eps]$ and are hence done by
\eqref{lem:compactness-aux-res.eqn1}.  In the case of
$\abs{\xi_2-\xi_1}> \eps/2$ we combine our assumptions of $F_\eps$
with \eqref{lem:compactness-aux-res.eqn1} and the triangle inequality
to obtain
\begin{align*}
\babs{\bar{F}_\eps\pair{\tau}{\xi_2}-\bar{F}_\eps\pair{\tau}{\xi_1}}
&\leq 
C\abs{\bar\xi_2-\bar\xi_1}^{\ga_2}+C\eps^{\ga_2-1}\bat{
\abs{\xi_1-\bar\xi_1}^{\ga_2}+\abs{\xi_2-\bar\xi_2}^{\ga_2}}
\\&
\leq C\abs{\bar\xi_2-\bar\xi_1}^{\ga_2}+C\eps^{\ga_2}
\end{align*}
and the desired estimate follows with
\begin{align*}
\abs{\bar\xi_1-\bar\xi_2}\leq \abs{\xi_2-\xi_1}+2\cdot\eps/2\leq 3\abs{\xi_2-\xi_1}.
\end{align*}
The claim of the Lemma is now a consequence of the convergence
estimate, the spatial and temporal H\"older estimates for
$\bar{F}_\eps$, and the Arzel\'a-Ascoli theorem applied to
$\bar{F}_\eps$.
\end{proof}
Since the functions $Q_\eps$ are obtained by solving the discrete heat
equation with macroscopic initial data, they converge as $\eps\to0$ to
a smooth solution of the macroscopic heat equation $\partial_\tau
Q=\partial_\xi^2 Q$. This is not surprising and can be proven in many
different ways. For our purposes, it is sufficient to observe that
strong compactness is provided by combining Lemma
\ref{lem:compactness-aux-res} with the following H\"older estimates.
\begin{lemma}[H\"older estimates for $Q_\eps$]
  \label{lem:holder-Q}
There exists a constant $C$ independent of $\eps$ such that
\begin{align*}
  \babs{Q_\eps\pair{\tau_2}{\xi_2}-Q_\eps\pair{\tau_1}{\xi_1}}
  \leq
  {C}\bat{\abs{\tau_2-\tau_1}^{1/2}+\abs{\xi_2-\xi_1}}
\end{align*}
holds with $\xi_1,\xi_2\in\eps\Zset$ and $0\leq\tau_1,\tau_2\leq\tau_\fin$
for all $0<\eps\leq\eps_*$.
\end{lemma}
\begin{proof}
From Corollary \ref{cor:decay-regular-velocities} and $\abs{{g}_0\at{t}}\leq C t^{-1/2}$ we derive
\begin{align*}
\abs{q_{j}\at{t_2}-q_j\at{t_1}}\leq \int_{t_1}^{t_2}\abs{\dot{q}_j\at{t}}\dint{t}\leq \alpha\eps^2 \at{t_2-t_1}+{C}\beta\eps\at{\sqrt{t_2}-\sqrt{t_1}},
\end{align*}
as well as
\begin{align*}
\babs{q_{j_2}\at{t_1}-q_{j_1}\at{t_1}}\leq \sum_{j=j_1}^{j_2-1} \babs{\nabla_+ q_j\at{t_1}}\leq \alpha\eps\at{j_2-j_1}
\end{align*}
for all $j_1,j_2\in\Zset$ with $j_1<j_2$ and all $0\leq t_1\leq
t_2<\infty$. Setting $j_i=\xi_i/\eps$ and $t_i=\tau_i/\eps^2$ we
therefore get
\begin{align*}
  \babs{Q_\eps\pair{\tau_2}{\xi_2}-Q_\eps\pair{\tau_1}{\xi_1}}
  \leq
  \alpha\at{\xi_2-\xi_1}+\alpha\at{\tau_2-\tau_1}
  +
  \beta\at{\sqrt{\tau_2}-\sqrt{\tau_1}},
\end{align*}
and the claim follows since $0\leq\tau_1\leq\tau_2\leq\tau_\fin$ implies
$\abs{\tau_2-\tau_1}+\at{\sqrt{\tau_2}-\sqrt{\tau_1}}\leq
C\sqrt{\tau_2-\tau_1}$.
\end{proof}
A crucial part of our analysis is to establish strong
$\fspaceL^\infty$-compactness of the functions $R_\eps$. The main
difficulty is that $R_\eps$ equals the sum of $K_\eps$ shifted and
delayed versions of the discrete heat kernel producing a temporal
discontinuity at any of the phase transition times
$\tau^*_{1,\,\eps},\ldots,\tau^*_{K_\eps,\,\eps}$. In order to control
the impact of all these jumps in time we split
\begin{align*}
  R_\eps\pair{\tau}{\xi}
  =
  R_{1,\,\eps}\pair{\tau}{\xi}+R_{2,\,\eps}\pair{\tau}{\xi}
\end{align*}
with 
\begin{align*}
  R_{1,\,\eps}\pair{\tau}{\xi}
  :=
  -2 \sum_{k =1}^{K_\eps} H_\eps\pair{\tau-\tau^*_{k,\,\eps}}{\xi-\eps{k}}
\end{align*}
and $R_{2,\,\eps}:= R_\eps- R_{1,\,\eps}$. Here the function
$H_{\eps}:\Rset^2\to\Rset$,
\begin{align*}
H_{\eps}\pair{\tau}{\xi}:=\left\{ \begin{array}{lcl}
0&\text{for}& \tau\leq 0,
\\%
\D \frac{\tau}{d_*\eps}G_\eps\pair{d_*\eps}{\xi}  &\text{for}& 0\leq{\tau}\leq  d_*\eps,
\\%
G_\eps\pair{\tau}{\xi}&\text{for}& \tau\geq {d_*}\eps
\end{array}
\right.
\end{align*}
can be regarded as a temporally regularized version of $G_\eps$, 
where the latter represents the discrete heat kernel in macroscopic 
coordinates according to $G_\eps\pair{\eps^2t}{\eps j}=g_j\at{t}$.
In particular, $H_\eps$ is continuous with respect to $\tau\in\Rset$ and 
differs from $G_\eps$ for $0\leq\tau\leq d_*\eps$ only.
\par
The function $R_{2,\,\eps}$ contains all the temporal jumps caused by
the phase transitions and can therefore not be compact
$\fspaceL^\infty\at\Om$.  The key observation, however, is that
$R_{2,\,\eps}$ is still uniformly bounded and converges as $\eps\to0$
to $0$ in $\fspaceL^s\at\Om$ for any $1\leq{s}<\infty$.  The
macroscopic limit of $\at{R_\eps}_{0<\eps\leq\eps_*}$ is therefore
completely determined by the family
$\at{R_{1,\,\eps}}_{0<\eps\leq\eps_*}$, where each function
$R_{1,\,\eps}$ is continuous with respect to $\tau$ and satisfies
\begin{align}
  \label{eq:interface.condition}
  Q_\eps\pair{\tau^*_{k,\,\eps}}{\eps k}
  +
  R_{1,\,\eps}\pair{\tau^*_{k,\,\eps}}{\eps k}
  =
  \lim\limits_{\tau\nearrow\tau^*_{k,\,\eps}} P_\eps\pair{\tau}{\eps{k}}=1
\end{align}
for all $k\geq 1$ with $\tau^*_{k,\,\eps}\leq\tau_\fin$ thanks to
$\lim_{t\nearrow t_k^*}p_k\at{t}=1$.
\begin{lemma}[bounds for $R_{2,\,\eps}$]
  \label{lem:bounds-R2}
There exists a constant $C$ independent of $\eps$ such that
\begin{align*}
\norm{R_{2,\,\eps}}_{\fspaceL^\infty\at\Om}\leq{C},\qquad \norm{R_{2,\,\eps}}_{\fspaceL^1\at\Om}\leq C\eps
\end{align*}
holds for all $0<\eps\leq\eps_*$.
\end{lemma}
\begin{proof}
By construction and Corollary \ref{cor:structure-of-p} we have 
\begin{align*}
  \supp \, R_{2,\,\eps}\subset \bat{ I_{1,\,\eps}\cup\dotsb\cup I_{K_\eps,\,\eps}}\times\Rset ,\qquad
I_{k,\,\eps}:=[\tau^*_{k,\,\eps},\, \tau^*_{k,\,\eps}+d_*\eps],
\end{align*}
and the intervals $I_{k,\,\eps}$ are pairwise disjoint thanks to
\eqref{eq:props-keps}.  In particular, using
\begin{align*}
R_{2,\,\eps}\pair{\tau^*_{k,\,\eps}+\si}{j\eps}= 2 H_\eps\pair{\si}{\eps j -\eps k} - 2 G_\eps\pair{\si}{\eps j- \eps k}
\qquad
\text{for all}\quad j\in\Zset,\quad \si\leq d_*\eps,
\end{align*}
we estimate
\begin{align*}
\abs{R_{2,\,\eps}\pair{\tau^*_{k,\,\eps}+\si}{j\eps}}\leq 2
\Bat{G_{\eps}\pair{0}{\eps j-\eps k} +  G_\eps\pair{d_*\eps}{\eps j- \eps k}} \leq2,
\end{align*}
see Lemma \ref{lem:discr-heat-kernel}, as well as
\begin{align*}
\int_\Rset \abs{R_{2,\,\eps}\pair{\tau^*_{k,\,\eps}+\si}{\xi}}\dint\xi\leq 2\eps
\sum_{j\in\Zset}  \Bat{g_{j-k}\at{0}+g_{j-k}\bat{d_*/\eps}}=4\eps
\end{align*}
thanks to $\sum_{j\in\Zset} g_j\at{t}=1$ for all $t$. The first
estimate implies $\norm{R_{2,\,\eps}}_{\fspaceL^\infty\at\Om} \leq 2$,
whereas the second gives rise to
\begin{align*}
\norm{R_{2,\,\eps}}_{\fspaceL^1\at\Om}\leq\sum_{k=1}^{K_\eps}
\int _0^{d_*\eps}
\int _\Rset \abs{R_{2,\,\eps}\pair{\tau^*_{k,\,\eps}+\si}{\xi}}\dint\xi\dint\si\leq
2 K_\eps d_* \eps ^2 \leq  \tau_\fin \eps,
\end{align*}
where we used \eqref{eq:props-keps} again.
\end{proof}
It remains to establish $\fspaceL^\infty$-compactness results for
$R_{1,\,\eps}$. To this end we next derive a further auxiliary result
concerning the H\"older continuity of $H_\eps$.
\begin{lemma}[H\"older estimates for $H_\eps$]
\label{lem:hoelder-est-h}
For each $0<\gamma<1$ there exists a constant $C$ independent of $\eps$ such that
\begin{align*}
\babs{H_{\eps}\pair{\tau_2}{\xi_2}-H_{\eps}\pair{\tau_1}{\xi_1}}\leq 
C \eps  \at{
\frac{\abs{\tau_2-\tau_1}^{\gamma}}{\max\{d_*\eps,\tau_1\}^{\gamma+1/2}}+
\frac{\abs{\xi_2-\xi_1}^{1/2}}{\max\{ d_* \eps,\tau_1\}^{3/4}}
}
\end{align*}
holds with $\xi_1,\xi_2\in\eps\Zset$ and $0\leq\tau_1\leq\tau_2\leq\tau_\fin$
for all $0<\eps\leq\eps_*$.
\end{lemma}
\begin{proof}
Suppose at first that $d_*\eps\leq\tau_1\leq \tau_2$. Thanks to the temporal
H\"older estimates for the discrete heat kernel, see Lemma \ref{Lem:App.Holder}, we find
\begin{align*}
\babs{H_{\eps}\pair{\tau_2}{\xi_2}-H_{\eps}\pair{\tau_1}{\xi_2}}\leq \frac{C\at{\D\frac{\tau_2}{\eps^2}-\frac{\tau_1}{\eps^2}}^\gamma}{\D\at{\frac{\tau_1}{\eps^2}}^{\gamma+1/2}}=
\frac{C\eps}{ \tau_1^{\gamma+1/2}}\abs{\tau_2-\tau_1}^\gamma
\end{align*}
for some constant $C$ independent of $\eps$ and $\xi_2$. Similarly, Lemma \ref{Lem:App.Holder} also ensures that
\begin{align*}
\babs{H_{\eps}\pair{\tau_1}{\xi_2}-H_{\eps}\pair{\tau_1}{\xi_1}}\leq \frac{C\abs{\D\frac{\xi_2}{\eps}-\frac{\xi_1}{\eps}}^{1/2}}{\D\at{\frac{\tau_1}{\eps^2}}^{3/4}}=\frac{C\eps}{\tau_1^{3/4}}\abs{\xi_2-\xi_1}^{1/2}.
\end{align*}
Now suppose that 
$0\leq\tau_1\leq\tau_2\leq d_*\eps$. We then estimate
\begin{align*}
\babs{H_{\eps}\pair{\tau_2}{\xi_2}-H_{\eps}\pair{\tau_1}{\xi_2}}
\frac{}{}&\leq \frac{g_{\xi_2/\eps}\bat{d_*/\eps}}{d_*\eps} \at{\tau_2-\tau_1}
\leq \frac{C}{d_*^{3/2}\eps^{1/2}} \at{\tau_2-\tau_1}
\\&\leq \frac{C}{d_*^{3/2}\eps^{1/2}} \at{d_*\eps}^{1-\gamma}\abs{\tau_2-\tau_1}^\gamma
\leq \frac{C\eps}{\at{d_*\eps}^{\gamma+1/2}}\abs{\tau_2-\tau_1}^\gamma
\end{align*}
as well as
\begin{align*}
\babs{H_{\eps}\pair{\tau_1}{\xi_2}-H_{\eps}\pair{\tau_1}{\xi_1}}=
\frac{\tau_1}{d_*\eps}\babs{g_{\xi_2/\eps}\at{d_*/\eps}
-g_{\xi_1/\eps}\at{d_*/\eps}}\leq  \frac{}{}\frac{C\abs{\D\frac{\xi_2}{\eps}-\frac{\xi_1}{\eps}}^{1/2}}{\D\at{\frac{d_*}{\eps}}^{3/4}}
=\frac{C\eps \abs{\xi_2-\xi_1}^{1/2}}{\at{d_*\eps}^{3/4}}.
\end{align*}
In summary, we have established the desired estimates in the special cases $0\leq\tau_1\leq\tau_2\leq d_*\eps$ or
$d_*\eps\leq\tau_1\leq\tau_2$. All other cases can be easily be traced back to these cases using the triangle inequality.
\end{proof}
We are now able to prove our main technical result in this section.
\begin{lemma}[H\"older estimates for $R_{1,\,\eps}$]
  \label{lem:hoelder-est-r1eps}
For each $0<\ga<1/2$ there exists a constant $C$ independent of $\eps$ such that
\begin{align*}
\babs{R_{1,\,\eps}\pair{\tau_2}{\xi_2}-R_{1,\,\eps}\pair{\tau_1}{\xi_1}} \leq {C}\bat{\abs{\tau_2-\tau_1}^{\ga}+\abs{\xi_2-\xi_1}^{1/2}}
\end{align*}
holds with $\xi_1,\xi_2\in\eps\Zset$ and $0\leq\tau_1\leq\tau_2\leq\tau_\fin$
for all $0<\eps\leq\eps_*$.
\end{lemma}
\begin{proof}
  It is sufficient to proof the assertions in time and space separately.
\par
\underline{\emph{H\"older continuity with respect to $\xi$}}:  
Let $0\leq \tau\leq\tau_\fin$ and $\xi_1,\xi_2\in\Rset$ be given. Then
there exists $m_\eps\in\Zset$ such that
\begin{align*}
\tau^*_{m_\eps,\,\eps} < \tau \leq \tau^*_{m_\eps+1,\,\eps}
\end{align*}
and \eqref{eq:props-keps} ensure that $\eps m_\eps \leq {\tau}/\at{2
  d_*}$ as well as
\begin{align*}
\tau-\tau^*_{k,\,\eps}\geq 2d_*\eps\at{m_\eps-k}\qquad \text{for all}\quad
 k=1,\ldots, m_\eps.
\end{align*}
In particular, we have
\begin{align*}
H_\eps\pair{\tau-\tau^*_{k,\,\eps}}\xi=0 \qquad \text{for all}\quad k> m_\eps,\quad \xi\in\Rset,
\end{align*}
so Lemma \ref{lem:hoelder-est-h} yields
\begin{align*}
\babs{R_{1,\,\eps}\pair{\tau}{\xi_2}-R_{1,\,\eps}\pair{\tau}{\xi_1}}&\leq
2\sum_{k=1}^{m_\eps}
\babs{H_{\eps}\pair{\tau-\tau^*_{k,\,\eps}}{\xi_2}-H_{\eps}\pair{\tau-\tau^*_{k,\,\eps}}{\xi_1}}
\leq C S_{\eps} \abs{\xi_2-\xi_1}^{1/2}
\end{align*}
with
\begin{align*}
S_\eps &:= 
\sum_{k=1}^{m_\eps}
\frac{\eps}{\at{\max\big\{d_*\eps,\, \tau-\tau^*_{k,\,\eps}\big\}}^{3/4}}
\leq
\frac{1}{d_*}+\sum_{k=1}^{m_\eps-1}\frac{\eps}{\at{2d_*\eps \at{m_\eps - k}}^{3/4}}
\leq \frac{1}{d_*}+\frac{1}{2d_*}\int_{0}^{\tau}\frac{\dint{\si}}{\si^{3/4}}\leq{C},
\end{align*}
where we used the Riemann sum approximation of the integral as well as the monotonicity of the integrand. 
\par
\underline{\emph{H\"older continuity with respect to $\tau$}}:  
Now let $\xi\in\Rset$ and $0\leq\tau_1<\tau_2\leq\tau_\fin$ be fixed, and choose
$m_{\eps},n_{\eps}\in\Zset$ such that
\begin{align*}
\tau^*_{m_{\eps},\,\eps} < \tau_1 \leq \tau^*_{m_{\eps}+1,\,\eps},\qquad
\tau^*_{n_{\eps},\,\eps} < \tau_2 \leq \tau^*_{n_{\eps}+1,\,\eps}.
\end{align*}
This gives
\begin{align*}
\eps m_{\eps}\leq \tau_1/\at{2d_*},\qquad
\eps\at{ n_{\eps} -m_{\eps}} \leq \at{\tau_2-\tau_1}/\at{2d_*}
\end{align*}
as well as
\begin{align*}
\babs{R_{1,\,\eps}\pair{\tau_2}{\xi}-R_{1,\,\eps}\pair{\tau_1}{\xi}}&\leq  2X_{\eps}+2Y_\eps,
\end{align*}
where
\begin{align*}
X_{\eps}:= 
\sum_{k=1}^{m_{\eps}}
\babs{H_{\eps}\pair{\tau_2-\tau^*_{k,\,\eps}}{\xi}-H_{\eps}\pair{\tau_1-\tau^*_{k,\,\eps}}{\xi}},\qquad
Y_{\eps}:= 
\!\!\sum_{k=m_{\eps}+1}^{n_{\eps}}\!\!\
\babs{H_{\eps}\pair{\tau_2-\tau^*_{k,\,\eps}}{\xi}}.
\end{align*}
Similar to the above, we deduce that
\begin{align*}
X_\eps \leq
\sum_{k=1}^{m_{\eps}} \frac{C \eps \abs{\tau_2-\tau_1}^\ga}{\max\{d_*\eps,\, \tau-\tau^*_{k,\,\eps}\}^{\ga+1/2}}\leq {C}\at{
\frac{1}{d_*}+\frac{1}{2d_*}\int_{0}^{\tau_1}\frac{\dint{\si}}{\si^{\ga+1/2}}}\abs{\tau_2-\tau_1}^\ga\leq
{C}\abs{\tau_2-\tau_1}^\ga,
\end{align*}
whereas $Y_\eps$ can be estimated by
\begin{align*}
Y_\eps&\leq \sum_{k=m_{\eps}+1}^{n_{\eps}}G_\eps\pair{\tau_2-\tau^*_{k,\,\eps}}{0}
\leq
\sum_{k=m_{\eps}+1}^{n_{\eps}} G_\eps\bpair{\tau_2 - 2d_*\eps\at{n_{\eps}-k}}{0}
\\& 
\leq \frac{1}{2d_*\eps }\int_{\tau_1}^{\tau_2} G_\eps\pair{\tau_2 - \sigma}{0}\dint\si=
\frac{1}{2d_*\eps }\int_{0}^{\tau_2-\tau_1} G_\eps\pair{\si}{0}\dint\si\\
&\leq
{C}\int_{0}^{\tau_2-\tau_1}\frac{\dint\si}{\si^{1/2}}=C\abs{\tau_2-\tau_1}^{1/2}\leq {C} \abs{\tau_2-\tau_1}^{\ga},
\end{align*}
see Lemma \ref{lem:discr-heat-kernel} and \eqref{eq:props-keps}.
\end{proof}

We conclude this section by showing that both the $\fspaceL^1$-norm of
$R_{1,\,\eps}\pair{\tau}{\cdot}$ and the $\fspaceL^2$-norm of
$\nabla_{\eps} R_{1,\,\eps}(\tau,\,\cdot)$ are bounded uniformly with
respect to $\tau$ and $\eps$, where
\begin{align*}
  \nabla_{\eps} F \at{\xi}:=\frac{F\at{\xi+\eps}-F\at{\xi}}{\eps}
\end{align*}
is the discrete spatial gradient of a function $F$ defined on $\Rset$.
\begin{lemma}[Lebesgue bounds for $R_{1,\,\eps}$ and its discrete gradient]
  \label{lem:lebesgue-r1}
There exists a constant $C$ such that
\begin{align*}
\norm{R_{1,\,\eps}}_{\fspaceL^\infty\at{I;\,\fspaceL^1\at\Rset}}+
\norm{\nabla_{\eps} R_{1,\,\eps}}_{\fspaceL^\infty\at{I;\,\fspaceL^2\at\Rset}}\leq{C}
\end{align*}
holds for all $0\leq\eps\leq\eps_*$.
\end{lemma}
\begin{proof}
For $\tau\geq\eps d_*$, the properties of the discrete heat kernel, see Lemma \ref{lem:discr-heat-kernel}, imply
\begin{align*}
\int_\Rset H_\eps\pair{\tau}{\xi}\dint\xi=\eps \sum_{j\in\Zset} g_j\at{\tau/\eps^2}=\eps,\qquad
\int_\Rset \Bat{\nabla_\eps H_\eps\pair{\tau}{\xi}}^2\dint\xi=\eps^{-1} \sum_{j\in\Zset} \Bat{\nabla_+ g_j\at{\tau/\eps^2}}^2\leq\frac{\eps^2 C}{\tau^{3/2}},
\end{align*}
and combining this with the definition of $H_\eps$ for all
$\tau\in\Rset$ we find
\begin{align*}
\norm{H_\eps\pair{\tau}{\cdot}}_{\fspaceL^1\at\Rset}\leq\eps,\qquad 
\norm{\nabla_\eps H_\eps\pair{\tau}{\cdot}}_{\fspaceL^2\at\Rset}\leq{C}\left\{\begin{array}{ll}
0&\quad\text{for $\tau<0$},\\
\eps^{1/4}&\quad\text{for $0<\tau<d_*\eps$},\\
\eps \tau^{-3/4}&\quad\text{for $\tau>d_*\eps$}.
\end{array}\right.
\end{align*}
From the first estimate we infer that
\begin{align*}
\norm{R_{1,\,\eps}\pair{\tau}{\cdot}}_{\fspaceL^1\at\Rset}\leq \eps K_\eps \leq \frac{\tau_\fin}{2d_*}
\end{align*}
holds for all $\tau\in{I}$.  We next fix $\tau\in{I}$ and choose an
integer $m_\eps$ such that $\tau^*_{m_\eps,\,\eps} < \tau \leq
\tau^*_{m_\eps+1,\,\eps}$.  As in the first part of the proof of Lemma
\ref{lem:hoelder-est-r1eps}, we estimate
\begin{align*}
\norm{\nabla_\eps R_{1,\,\eps}\pair{\tau}{\cdot}}_{\fspaceL^2\at\Rset}&\leq
\sum_{k=1}^{m_\eps}
\norm{\nabla_\eps H_{\eps}\pair{\tau-\tau^*_{k,\,\eps}}{\cdot}}_{\fspaceL^2\at\Rset}
\\&\leq
C\eps^{1/4}+\sum_{k=1}^{m_\eps-1}\frac{C\eps}{\bat{2kd_*\eps\at{m_\eps-k}}^{3/4}}=C\at{\eps^{1/4}+1},
\end{align*}
and the proof is complete.
\end{proof}


\subsection{Convergence results and verification of limit dynamics}
\label{sec:PassageToLimit}

In view of the compactness results in Lemmas \ref{lem:compactness-xi},
\ref{lem:holder-Q}, and \ref{lem:hoelder-est-r1eps}, we may select a
subsequence of $\eps \to 0$, which we do not relabel, such that
\begin{align}
  \label{eq:conv-subseq}
  \xi^*_\eps \to \xi^* \text{ in } \fspaceL^\infty(I),
  \qquad
  Q_\eps \to Q \text{ in } \fspaceL^\infty(\Omega),
  \qquad
  R_{1,\,\eps} \to R \text{ in } \fspaceL^\infty(\Omega).
\end{align}
As $R_{2,\,\eps} \to 0$ in $\fspaceL^s(\Omega)$ for any $1 \leq {s}<
\infty$ by Lemma~\ref{lem:bounds-R2}, we find
\begin{align}
  \label{eq:conv-subseq-2}
  P_\eps = Q_\eps + R_{1,\,\eps} + R_{2,\,\eps}
  \to
  Q + R = P
  \qquad
  \text{in } \fspaceL^s_\loc(\Omega),
\end{align}
that means, $P$ is the limit of $P_\eps$ in $\fspaceL^s_\loc(\Omega)$
and the limit of $Q_\eps + R_{1,\,\eps}$ in $\fspaceL^\infty(\Omega)$.
Moreover, convergence of $(Q_\eps)_\eps$ implies convergence of the
initial data
\begin{align}
  \label{eq:conv-idata}
  P_\eps(0,\cdot) = Q_\eps(0,\cdot) \to Q(0,\cdot) = P(0,\cdot)
  \qquad
  \text{in } \fspaceL^\infty(\bbR).
\end{align}

\begin{theorem}[limit dynamics along subsequences]
  \label{thm:existence}
  Any limit $(P,Q,R,\xi^*)$ satisfies:
  \begin{enumerate}
  \item $\Xi^* :=
    \set[(\tau,\xi^*(\tau))]{ \tau \in I}$ is a Lipschitz curve in $\Om$;
    the functions $Q$, $R$, and $P=Q+R$ are bounded and locally
    H\"older continuous in $\Omega$; furthermore, $R \in
    \fspaceL^\infty(I; \fspaceL^1(\bbR))$ and $\partial_\xi R \in
    \fspaceL^\infty(I; \fspaceL^2(\bbR))$.
  \item $Q$ is a solution of the heat equation in $\Omega$ with
    initial data $P(0,\cdot)$.
  \item $(P,\xi^*)$ is a distributional solution of
   \begin{equation}
      \label{eq:strong-lim-eq}      
        \partial_\tau P = \partial^2_\xi P
        \;\;\text{in}\;\;\Omega \setminus \Xi^*,\qquad\quad
        \jump{P}=0\;\; \text{and}\;\; 2 \tdiff{\tau} \xi^* = \jump{\partial_\xi P}
        \;\; \text{on}\;\; \Xi^*
    \end{equation}
    with initial data $(P(0,\cdot),\xi^*(0))$ attained in
    $\fspaceL^\infty(\bbR) \times \bbR$. 
    Moreover, we have
    \begin{alignat*}{2}
      P(\tau,\xi) &\geq -1
      &\qquad &\text{ for all } (\tau,\xi) \in \Omega,
      \\
      P(\tau,\xi) &\leq 1
      &\qquad &\text{ if } \xi \geq \xi^*(\tau),
    \end{alignat*}
    which implies $P \in [-1,1]$ on $\Xi^*$, and the movement of the
    interface is determined by
    \begin{alignat*}{2}
      \tdiff{\tau} \xi^*(\tau) &\geq 0
      &\qquad &\text{ for almost all } \tau \in I,
      \\
      \tdiff{\tau} \xi^*(\tau) &= 0
      &\qquad &\text{ if } P(\tau,\xi^*(\tau)) \not=1.
    \end{alignat*}
  \end{enumerate}
\end{theorem}

\begin{remark}
  \label{rem:xi-and-mu}
  Being a distributional solution of \eqref{eq:strong-lim-eq} means
  \begin{align}
    \label{eq:distr-lim-eq}
    - \int_0^{\tau_\fin} \int_\bbR
    \left( P + \mu \right) \partial_\tau \psi \dint\xi \dint\tau
    =
    \int_0^{\tau_\fin} \int_\bbR
    P \,\partial^2_\xi \psi \dint\xi \dint\tau,
    \qquad
    \psi \in \fspaceC^\infty_c( (0,\tau_\fin) \times \bbR ),
  \end{align}
  where $\mu(\tau,\xi) = \sgn\left(\xi^*(\tau)-\xi\right)$. In the
  following we will use $\xi^*$ and $\mu$ interchangeably to represent
  the solution, whichever is more convenient.
\end{remark}

\begin{proof}[Proof of Theorem~\ref{thm:existence}]
  The continuity properties of $\xi^*$ and $Q$ are immediate
  consequences of Lemma~\ref{lem:compactness-xi} and
  Lemma~\ref{lem:holder-Q}; Lemma~\ref{lem:compactness-xi} also yields
  $\tdiff{\tau} \xi^* \geq 0$. H\"older continuity of $P$ and the
  claims for $R$ follow from $P = Q+R$ and the bounds on $R_{1,\,\eps}$
  proved in Lemmas~\ref{lem:hoelder-est-r1eps} and
  \ref{lem:lebesgue-r1}.

  Setting $\mu_\eps(\tau,\xi) = \sgn \left(\xi_\eps^*(\tau)-\xi\right)
  = \sgn U_\eps(\tau,\xi)$, we write the equation for $U_\eps = P_\eps
  + \mu_\eps$ in distributional form as
  \begin{align*}
    - \int_0^{\tau_\fin} \int_\bbR
    \left( P_\eps + \mu_\eps \right) \partial_\tau \psi \dint\xi \dint\tau
    =
    \int_0^{\tau_\fin} \int_\bbR
    P_\eps \laplace_\eps \psi \dint\xi \dint\tau,
    \qquad
    \psi \in \fspaceC^\infty_c((0,\tau_\fin)\times\bbR)
  \end{align*}
  and deduce \eqref{eq:distr-lim-eq} in the limit $\eps \to 0$.
  Similarly, $Q$ solves the heat equation, and both $P$ and $Q$
  attain their initial data in $\fspaceL^\infty(\bbR)$ 
  due to \eqref{eq:conv-idata} and continuity of  $Q$.
  \par
  By construction, the discrete solutions satisfy $P_\eps \geq -1$ in
  $\Omega$ and $P_\eps \leq 1$ in $\set[(\tau,\xi) \in \Omega]{\xi
    \geq \xi^*_\eps(\tau)}$ for all $\eps>0$, and in the limit $\eps\to0$ we obtain the
  corresponding inequalities for $P$ and $\xi^*$.
  In particular, we have $P(\tau,\xi^*(\tau)) \in [-1,1]$ for all $\tau
  \in I$.

  It remains to check that $P(\tau,\xi^*(\tau))<1$ implies
  $\tdiff{\tau} \xi^*(\tau)=0$. To this end, let
  $(\bar\tau,\bar\xi)\in{\Xi^*}$ and $\delta>0$ be given such that
  $P(\bar\tau,\bar\xi) = 1-2\delta$, and suppose at first that
  $\bar\xi<\xi^*\at{\tau_\fin}$.  For any $\eps$ choose
  $\bar\xi_\eps\in\eps\Zset$ such that $\bar\xi_\eps\leq
  \bar\xi\leq\bar\xi_\eps+\eps$ and denote by $\bar\tau_\eps$ the
  phase transition time corresponding to $\bar\xi_\eps$. See the left
  panel of Figure \ref{Fig:sign_app} for an illustration and notice
  that $\bar\tau_\eps<\tau_\fin$ because otherwise the interface
  position would be maximal via $\bar\xi=\xi^*\at{\tau_\fin}$.
  H\"older continuity of $P$ now implies
  \begin{align*}
    P(\bar\tau_\eps,\bar\xi_\eps) - 1 + 2\delta
    \leq
    |P(\bar\tau_\eps,\bar\xi_\eps) - P(\bar\tau,\bar\xi)|
    \leq
    C\bat{ |\bar\tau_\eps-\bar\tau|^\gamma +\eps^{\ga}}
  \end{align*}
  for some exponent $0<\ga<1$, while uniform convergence of
  $Q_\eps+R_{1,\,\eps} \to P$ as $\eps\to0$ and
  $(Q_\eps+R_{1,\,\eps})(\bar\tau_\eps,\bar\xi_\eps) = 1$, see
  \eqref{eq:interface.condition}, yield
  \begin{align*}
    P(\bar\tau_\eps,\bar\xi_\eps)
    \geq
    1 - o(1)_{\eps\to0}.
  \end{align*}
  We thus find $\delta \leq C |\bar\tau_\eps-\bar\tau|^\gamma$ for all
  sufficiently small $\eps>0$ and may select a subsequence of
  $\eps\to0$ such that $\bar\tau_\eps \to \bar\tau_0$ and 
  $\bar\tau_0<\bar\tau$. The uniform convergence $\xi^*_\eps\to\xi^*$
  implies
  \begin{align*}
    \xi^*(\bar\tau_0)= \lim_{\eps\to0}
  \xi_\eps^*(\bar\tau_\eps) = \lim_{\eps\to0} \bar\xi_\eps = \xi^*(\bar\tau),
  \end{align*}
  and $\tdiff{\tau}\xi^*\geq 0$ ensures that $\xi^*$ is constant in
  $[\bar\tau_0,\bar\tau]$. By a similar argument using $\hat\xi_\eps
  =\bar\xi_\eps+\eps$ and the corresponding phase transition time
  $\hat\tau_\eps$ we finally conclude that $\bar\tau$ is a regular
  point of $\xi^*$ and $\tdiff{\tau} \xi^*(\bar\tau)=0$.  Moreover, in
  case of $\bar\tau<\tau_\fin$ and $\bar\xi=\xi^*\at{\tau_\fin}$ we
  find that $\xi^*$ is constant on $[\bar\tau,\tau_\fin]$,
  and for $\bar\tau=\tau_\fin$ we can repeat the above reasoning after
  enlarging the time interval slightly beyond $\tau_\fin$.
\end{proof}

\begin{figure}[ht!]%
  \centering
  \includegraphics[width=0.45\textwidth]{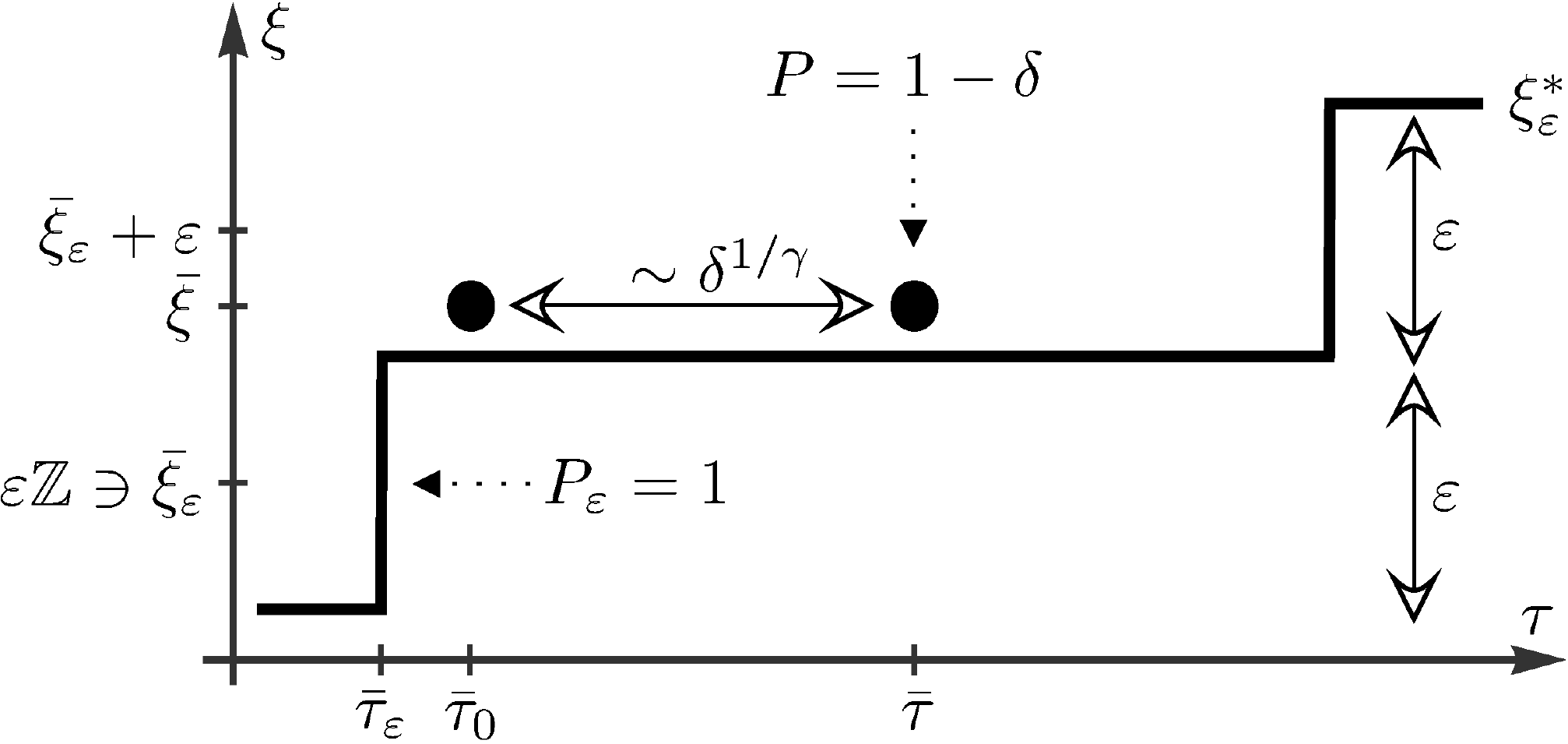}%
  \hspace{0.05\textwidth}%
  \includegraphics[width=0.45\textwidth]{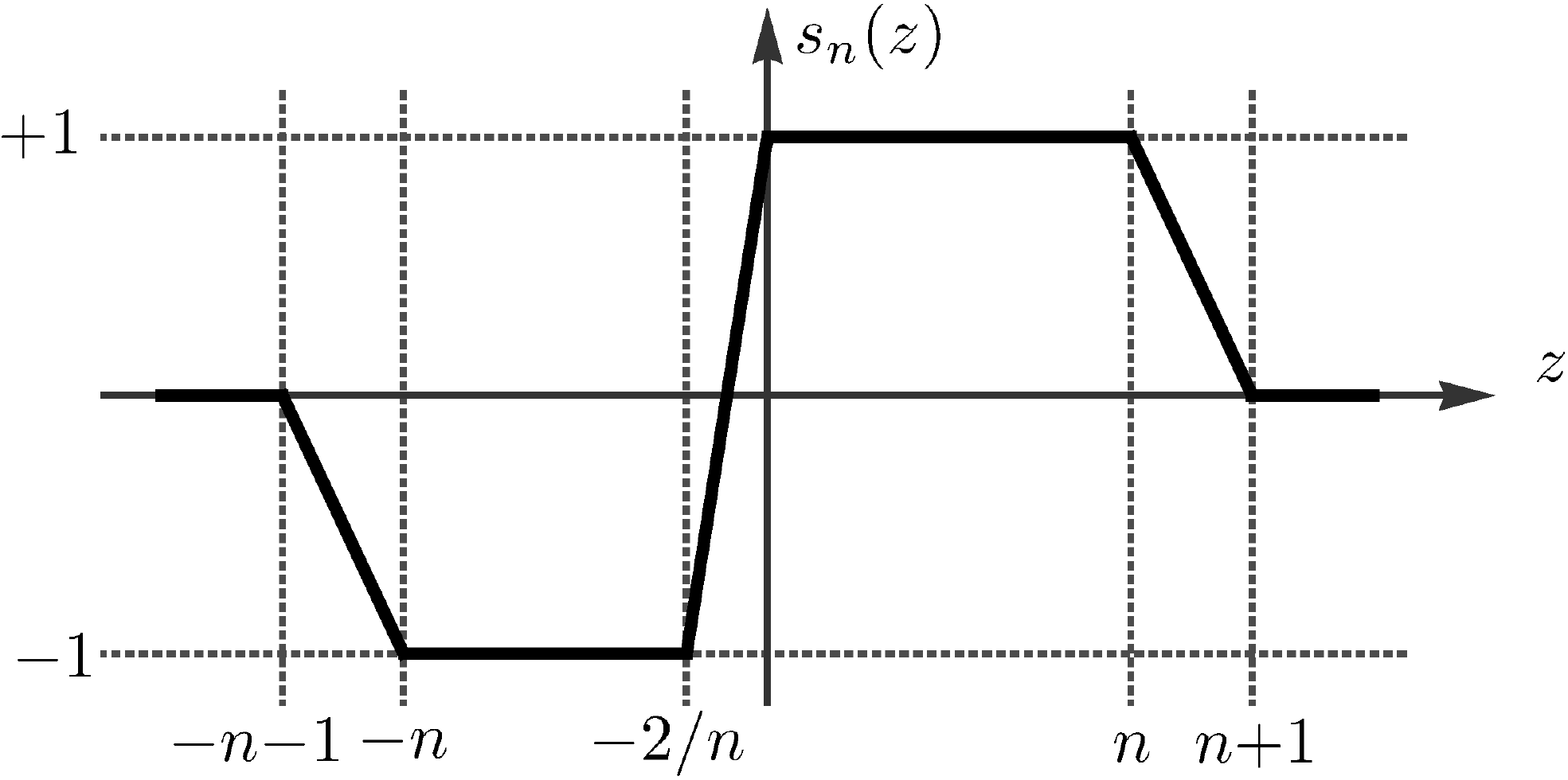}%
  \caption{%
    \emph{Left.} Illustration of the key argument in the proof of
    Theorem \ref{thm:existence}: By construction, we have
    $\bar\xi=\xi^*\nat{\bar\tau}=\xi^*\nat{\bar{\tau}_0}$ and
    $\tfrac{\dint{}}{\dint{\tau}}\xi^*\geq0$, so $\xi^*$ is constant
    on $[\bar{\tau}_0,\bar\tau]$.  \emph{Right.}
    Approximation $s_n$ of the sign function used in the proof of
    Theorem \ref{thm:uniqueness} for the case $\xi^*_1\geq\xi^*_2$;
    for $\xi^*_1\leq\xi^*_2$ one has to redefine $s_n$ such that
    $s_n\at{0}=-1$.}%
\label{Fig:sign_app}%
\end{figure}%

We complement Theorem \ref{thm:existence} with a uniqueness result by
adapting some techniques for hysteresis problems from
\cite{Hilpert89,Visintin06}.

\begin{theorem}[well-posedness of the limit problem]
  \label{thm:uniqueness}
  The solution to the limit problem in Theorem~\ref{thm:existence} is
  uniquely determined by the initial data $P(0,\cdot)$and $\xi^*(0)$.
\end{theorem}

\begin{proof}
  Given two solutions $(P_1,Q_1,R_1,\xi^*_1)$ and
  $(P_2,Q_2,R_2,\xi^*_2)$ with initial data
  $P_1(0,\cdot)=P_2(0,\cdot)$ and $\xi^*_1\at{0}=\xi^*_2\at{0}$, we
  set $\widebar{P} = P_1-P_2$ and $\widebar{\mu} = \mu_1-\mu_2$, where
  $\mu_i(\tau,\xi) = \sgn(\xi^*_i(\tau)-\xi)$ as in Remark
  \ref{rem:xi-and-mu}.
  In  order to show $\widebar{P} = \widebar{\mu} = 0$ we follow 
  the strategy of \cite[Theorem 5]{Hilpert89}, that means
  we first establish sufficient regularity in time and derive
  afterwards an $\fspaceL^1$-contraction inequality by testing
  the equation for $\widebar{P}$ with $\sgn \widebar{P}$.

  \underline{\emph{Regularity in time}}:
  Standard uniqueness results for the heat equation imply $Q_1 = Q_2$,
  and we find $\widebar{P} = R_1 -R_2 \in
  \fspaceL^\infty(I; \fspaceL^1(\bbR))$ and $\partial_\xi \widebar{P}
  \in \fspaceL^\infty(I; \fspaceL^2(\bbR))$ in addition to boundedness
  and continuity.
  Furthermore, $\widebar{\mu}(\tau,\xi)$ is bounded, and it is nonzero
  only if $\xi$ lies between $\xi^*_1(\tau)$ and $\xi^*_2(\tau$). We
  thus conclude
  \begin{align}
    \label{eq:integrability}
    \widebar{P} \in \fspaceL^\infty(I; \fspaceL^1(\bbR)) \cap
    \fspaceL^\infty(I; \fspaceH^1(\bbR))
    \qquad\text{and}\qquad
    \widebar{\mu} \in \fspaceL^\infty(I; \fspaceL^2(\bbR)).
  \end{align}
  In view of \eqref{eq:integrability}, we may integrate by parts after
  subtracting the equations for $(P_1,\mu_1)$ and $(P_2,\mu_2)$ from
  each other, which gives
  \begin{align*}
    \int_0^{\tau_\fin} \partial_\tau \phi(\tau)
    \int_\bbR 
    \big( \widebar{P}+\widebar{\mu} \big)(\tau,\xi)
    \,\eta(\xi) \dint\xi \dint\tau
    =
    \int_0^{\tau_\fin} \phi(\tau) \int_\bbR
    \partial_\xi \widebar{P}(\tau,\xi)
    \,\partial_\xi \eta(\xi) \dint\xi \dint\tau
  \end{align*}
  for all
  $\phi \in \fspaceC^\infty_c(I)$, $\eta \in \fspaceC^\infty_c(\bbR)$,
  and by density also for all
  $\phi \in \fspaceH^1_0(I)$, $\eta \in
  \fspaceH^1(\bbR)$. Consequently,
  \begin{align}
    \label{eq:diff-eq-deriv}
    \diff{\tau} \int_\bbR
    \big( \widebar{P}(\tau,\xi) + \widebar{\mu}(\tau,\xi) \big)
    \,\eta(\xi) \dint\xi
    =
    - \int_\bbR \partial_\xi \widebar{P}(\tau,\xi)
    \,\partial_\xi \eta(\xi) \dint\xi
  \end{align}
  for all $\tau \in I$.
  A direct computation shows that
  \begin{align}
    \label{eq:mu-deriv}
    \diff{\tau}
    \int_\bbR \widebar{\mu}(\tau,\xi) \, \eta(\xi) \dint\xi
    =
    2 \eta(\xi^*_1(\tau)) \, \tdiff{\tau} \xi^*_1(\tau)
    - 2 \eta(\xi^*_2(\tau)) \, \tdiff{\tau} \xi^*_2(\tau)
  \end{align}
  for all $\tau\in{I}$ where $\tdiff{\tau} \xi^*_1(\tau)$ and
  $\tdiff{\tau} \xi^*_2(\tau)$ are defined, and the right hand side of
  \eqref{eq:mu-deriv} can easily be bounded by $C
  \|\eta\|_{\fspaceH^1(\bbR)}$, where the constant $C$ depends on
  $\|\xi^*_j\|_{\fspaceW^{1,\infty}(I)}$, $j=1,2$.
  Thus, $\partial_\tau \widebar{P}(\tau,\cdot)$ exists in
  $\fspaceH^1(\bbR)$ and
  \begin{align*}
    \left| \left< \partial_\tau \widebar{P}(\tau,\cdot), \eta \right> \right|
    =
    \left| \tdiff{\tau} \int_\bbR \widebar{P}(\tau,\xi)
      \,\eta(\xi) \dint\xi \right|
    \leq
    C \left(
      \|\partial_\xi \widebar{P}\|_{\fspaceL^\infty(I; \fspaceL^2(\bbR))} + 1
    \right)
    \| \eta \|_{\fspaceH^1(\bbR)}.
  \end{align*}
  By standard embedding results, see for instance \cite[Thm.~3 in
  Sec.~5.9]{Evans98} and note that $I$ is bounded, $\partial_\tau
  \widebar{P} \in \fspaceL^\infty(I; \fspaceH^{-1}(\bbR))$ and
  $\widebar{P} \in \fspaceL^\infty(I;\fspaceH^1(\bbR))$ imply
  $\widebar{P} \in \fspaceC(I;L^2(\bbR))$, and with
  \eqref{eq:integrability} we conclude $\widebar{P} \in
  \fspaceC(I;L^1(\bbR))$.

  \underline{\emph{Contraction inequality}}:
  Given $\tau \in [\tau_1,\tau_2]$, where $0 \leq \tau_1<\tau_2 \leq
  \tau_\fin$ such that $\xi^*_1 \geq \xi^*_2$ in $[\tau_1,\tau_2]$, we
  approximate the sign function by
  \begin{align*}
    s_n(z) =
    \begin{cases}
      \max(-1,\min(1,1+nz))
      &\text{if } |z| \leq n, \\
      n + \sgn z - |z|
      &\text{if } n < |z| \leq n+1, \\
      0
      &\text{otherwise,}
    \end{cases}
  \end{align*}
  see Figure \ref{Fig:sign_app} for an illustration. In what follows
  we suppose $n > \| \widebar{P} \|_\infty$ and consider $\eta =
  s_n(\widebar{P}(\tau,\cdot)) \in \fspaceH^1(\bbR)$ in
  \eqref{eq:diff-eq-deriv}--\eqref{eq:mu-deriv}.
  Due to $\widebar{P}(\tau,\xi^*_1(\tau)) \geq 0$ if $\tdiff{\tau}
  \xi^*_1(\tau)>0$ and $s_n(0) = 1$, we then find
  \begin{equation*}
    \eta(\xi^*_1(\tau))  \tdiff{\tau} \xi^*_1(\tau)
    \geq
    \tdiff{\tau} \xi^*_1(\tau),
  \end{equation*}
  while $s_n \leq 1$ and $\tdiff{\tau} \xi^*_2(\tau) \geq 0$ imply
  \begin{equation*}
    - \eta(\xi^*_2(\tau))  \tdiff{\tau} \xi^*_2(\tau)
    \geq
    - \tdiff{\tau} \xi^*_2(\tau).
  \end{equation*}
  Hence, we obtain the Hilpert estimate
  \begin{align*}
    \eta(\xi^*_1(\tau)) \tdiff{\tau} \xi^*_1(\tau)
    -
    \eta(\xi^*_2(\tau)) \tdiff{\tau} \xi^*_2(\tau)
    \geq
    \tdiff{\tau} \xi^*_1(\tau) - \tdiff{\tau} \xi^*_2(\tau),
  \end{align*}
  and as moreover the right hand side of \eqref{eq:diff-eq-deriv} is
  nonpositive, we infer
  \begin{align}
    \label{eq:diff-contract}
    \left< \partial_\tau \widebar{P}(\tau,\cdot),
      s_n(\widebar{P}(\tau,\cdot)) \right>
    +
    2 \left( \tdiff{\tau} \xi^*_1(\tau) - \tdiff{\tau} \xi^*_2(\tau) \right)
    \leq
    0.
  \end{align}
  Using 
  \begin{align*}
    \left< \partial_\tau \widebar{P}(\tau,\cdot),
    s_n(\widebar{P}(\tau,\cdot)) \right>
    =
    \diff{\tau} \int_\bbR S_n(\widebar{P}(\tau,\xi)) \dint\xi,
  \end{align*}
  where $S_n'(z) = s_n(z)$ and $S_n(0) = 0$, we next integrate
  \eqref{eq:diff-contract} from $\tau_1$ to $\tau_2$ and arrive at
  \begin{align*}
    \int_\bbR S_n(\widebar{P}(\tau_2,\xi)) \dint\xi
    +
    2 \, |\xi^*_1(\tau_2) - \xi^*_2(\tau_2)|
    \leq
    \int_\bbR S_n(\widebar{P}(\tau_1,\xi)) \dint\xi
    +
    2 \, |\xi^*_1(\tau_1) - \xi^*_2(\tau_1)|.
  \end{align*}
  By construction, $S_n(\widebar{P}(\tau,\cdot))$ converges to
  $|\widebar{P}(\tau,\cdot)|$ in $\fspaceL^1(\bbR)$ as $n \to \infty$,
  and passing to the limit yields the desired inequality
  \begin{align}
    \label{eq:contraction}
    \| \widebar{P}(\tau_2,\cdot)\|_{\fspaceL^1(\bbR)}
    +
    2 \, |\xi^*_1(\tau_2) - \xi^*_2(\tau_2)|
    \leq
    \| \widebar{P}(\tau_1,\cdot)\|_{\fspaceL^1(\bbR)}
    +
    2 \, |\xi^*_1(\tau_1) - \xi^*_2(\tau_1)|
  \end{align}
  in the case of $\xi^*_1 \geq \xi^*_2$ in $[\tau_1,\tau_2]$. Moreover, 
  for $\xi^*_1 \leq \xi^*_2$ in $[\tau_1,\tau_2]$
  we derive \eqref{eq:contraction} by repeating the above arguments
  with $s_n(z) = \max(-1,\min(1,-1 + n z))$ for $|z| \leq n$, which
  satisfies $s_n(0)=-1$.
  Combining both cases and continuity of
  $\| \widebar{P}(\tau,\cdot)\|_{\fspaceL^1(\bbR)}$ we finally obtain
  \begin{align*}
  \norm{\widebar{P}}_{\fspaceL^\infty(I; \fspaceL^1(\bbR)) }+2\norm{\xi^*_1-\xi^*_2}_{\fspaceL^\infty(I)}\leq
  \norm{\widebar{P}(0,\cdot)}_{L^1(\bbR)} + 2\abs{\xi^*_1(0)-\xi^*_2(0)},
  \end{align*}
  so uniqueness follows from $\widebar{P}(0,\cdot)=0$ and $\xi_1^*(0)=\xi_2^*(0)$.
    \end{proof}

As a consequence of Theorems~\ref{thm:existence}
and~\ref{thm:uniqueness} we obtain the following approximation result.

\begin{corollary}[uniqueness and improved convergence]
  If $P_\eps(0,\cdot) \to P(0,\cdot)$ in $\fspaceL^\infty(\bbR)$ as
  $\eps \to 0$, the limit $(P,Q,R,\xi^*)$ in Theorem
  \ref{thm:existence} is unique and the convergence
  \eqref{eq:conv-subseq}--\eqref{eq:conv-subseq-2} holds along the
  whole family $\eps \to 0$.
\end{corollary}


\appendix

\section{The discrete heat kernel}
\label{sec:discrete-heat-kernel}
Denoting by $\widehat~$ the Fourier transform with respect to the discrete spatial variable $j$, that is
\begin{align*}
  \widehat g(t,k) = \sum_{j \in \bbZ} g_j(t)  e^{-\imath k j},
  \qquad k \in [-\pi,\pi),
\end{align*}
the initial value problem \eqref{eq:ap:heat-eq} for the discrete heat kernel transforms into
\begin{align}
  \label{eq:ap:FourierTransformedEqn}
  \partial_{t}\widehat{g}(t,k) &= \rho(k) \widehat{g}(t,k),
  \qquad \widehat{g}(0,k)=1,
\end{align}
where $\rho$ is the Fourier symbol of the negative discrete Laplacian, i.\,e.
\begin{align*}
\rho(k) = 2-e^{-\imath k }-e^{+\imath k}=2 (1-\cos k).
\end{align*}
Solving the parametrized ODE \eqref{eq:ap:FourierTransformedEqn} and applying 
the inverse Fourier transform we find
\begin{align}
  \label{eq:ap:discr-heat-kernel}
  g_j(t)
  =
  \frac{1}{2\pi} \int_{-\pi}^{+\pi} \widehat g(t,k) e^{\imath k j} \dint{k}
  =
  \frac{1}{2\pi} \int_{-\pi}^{+\pi} \exp\at{-\rho(k) t} \cos\at{jk} \dint{k}.
\end{align}
\begin{lemma}[monotonicity and convexity properties for $j=0$]
\label{lem:discr-heat-kernel-monotonicty}
\quad
\begin{enumerate}
\item $t\mapsto g_0\at{t}$ is strictly positive, strictly decreasing,
  and strictly convex.
\item $t\mapsto \int_0^t g_0\at{s}\dint{s}$ is strictly positive,
  strictly increasing, and strictly concave.
\item $t\mapsto\dot{g}_0\at{t}$ is strictly negative, strictly
  increasing, and strictly concave.
\end{enumerate}
\end{lemma}
\begin{proof} The representation formula \eqref{eq:ap:discr-heat-kernel} implies
\begin{align*}
g_0\at{t}>0,\qquad \dot{g}_0\at{t}<0,\qquad \ddot{g}_0\at{t}>0,\qquad
\dddot{g_0}\at{t}<0
\end{align*} 
for all $t\geq0$, so all assertions follow immediately.
\end{proof}
Employing standard methods from asymptotic analysis one finds
\begin{align*}
t^{1/2}g_j\at{t}
  \quad\xrightarrow{t\to\infty}\quad\frac{1}{2\sqrt{\pi}},\qquad
  t^{3/2}\dot{g}_j\at{t}
  \quad\xrightarrow{t\to\infty}\quad-\frac{1}{4\sqrt{\pi}}
\end{align*}
as well as asymptotic laws for the long-time behavior of any discrete moment.
For our considerations in \S\ref{sect:ToyModel}, however, the following rather rough estimates 
are sufficient.

\begin{lemma}[temporal decay properties]
  \label{lem:discr-heat-kernel}
  There exist positive constants $c$ and $C$ such that
  \begin{gather}
    \label{lem:discr-heat-kernel.Eqn1}
    0 \leq g_j(t) \leq g_0(t) \leq C\at{1+t}^{-1/2},
    \\
    \label{lem:discr-heat-kernel.Eqn2}
    |\laplace g_j(t)|
    =
    |\dot g_j(t)| \leq -\dot g_0(t)
    \leq
    C  \at{1+t}^{-3/2}
    \\
    \intertext{and}
    \label{lem:discr-heat-kernel.Eqn3}
    g_0(t) \geq c\at{1+t}^{-1/2}
  \end{gather}
  hold for all $j\in\Zset$ and all $t\geq0$. Moreover, we have
  \begin{align}
   \label{lem:discr-heat-kernel.Eqn4}
   \sum_{j\in\Zset} g_j\at{t}=1,\qquad
   \sum_{j\in\Zset} \babs{\nabla_+ g_j\at{t}}^2\leq C\at{1+t}^{-3/2}
  \end{align}
  for all $t\geq0$ and some constant $C$.
\end{lemma}
\begin{proof}
For $t\geq1$ we observe that 
\begin{equation*}
  \frac{1}{4} k^2 \leq \frac{4}{\pi^2} k^2 \leq \rho(k) \leq k^2
  \qquad
  \text{for all } \quad k \in [-\pi,\pi],
\end{equation*}
and this implies
\begin{equation*}
  \abs{g_j(t)}\leq g_0\at{t}
  \leq
  \frac{1}{2\pi}
  \int_{-\pi}^{+\pi} \exp\bat{ - k^2 t/4 } \dint{k}
  \leq
  \frac{1}{\pi \sqrt{t}}
  \int_{-\infty}^{+\infty} \exp\bat{ - k^2 } \dint{k}
  =
  \frac{\sqrt{\pi}}{\sqrt{t}}
\end{equation*}
 as well as
\begin{equation*}
  g_0(t)
  \geq
  \frac{1}{2\pi}
  \int_{-\pi}^{+\pi} \exp\bat{ - k^2 t } \dint{k}
  =
  \frac{1}{2\pi\sqrt{t}}
  \int_{-\pi\sqrt{t}}^{+\pi\sqrt{t}} \exp\bat{ - k^2 } \dint{k}
  \geq
  \frac{1}{2\pi\sqrt{t}}
  \int_{-\pi}^{+\pi} \exp\bat{ - k^2 } \dint{k}.
\end{equation*}
Combining these estimates with $0<g_0\at{1}\leq g_0\at{t}\leq g_0\at{0}=1$ for $0\leq{t}\leq1$ we readily obtain
\eqref{lem:discr-heat-kernel.Eqn1} and \eqref{lem:discr-heat-kernel.Eqn2}. Moreover, for $t\geq1$ we estimate
\begin{align*}
  |\laplace g_j(t)|
  =
  |\dot g_j(t)|
  &\leq
  \frac{1}{2\pi}
  \int_{-\pi}^{+\pi} \rho(k) \exp \bat{ - \rho(k) t} \dint{k}
  =-\dot{g}_0\at{t}\leq
  \frac{1}{2\pi}
  \int_{-\pi}^{+\pi} k^2 \exp \bat{ - k^2 t/4 } \dint{k}\leq
  \frac{2\sqrt{\pi}}{t^{3/2}}
\end{align*}
and this provides \eqref{lem:discr-heat-kernel.Eqn3} due to
$0<-\dot{g_0}\at{t}\leq -\dot{g}_0\at{0}=-\laplace g_0\at{0}=2$ for
all $t$. The discrete heat equation  \eqref{eq:ap:heat-eq} further ensures conservation of mass via $\sum_{j\in\Zset}
g_j\at{t}=\sum_{j\in\Zset}
g_j\at{0}=1$. In particular, using
discrete integration by parts as well as H\"older's inequality for series
we find
\begin{align*}
\sum_{j\in\Zset} \bat{\nabla_+{g_j}\at{t}}^2=-
\sum_{j\in\Zset} {g_j}\at{t}{\laplace g_j}\at{t}\leq \norm{\laplace g\at{t}}_\infty
\end{align*}
which implies the estimate in \eqref{lem:discr-heat-kernel.Eqn4} thanks to \eqref{lem:discr-heat-kernel.Eqn2}.
\end{proof}

A further key ingredient to our convergence proof in \S\ref{sect:ToyModel} are the following 
time-dependent H\"older estimates for $g_j\at{t}$.

\begin{lemma}[longtime behavior of spatial and temporal H\"older constants]
\label{Lem:App.Holder}
For each $\ga\in[0,1]$ there exists a constant $C_\ga$ such that
\begin{align*}
\sup\limits_{t_2\geq t_1}\frac{\babs{g_j\at{t_2}-g_j\at{t_1}}}{\abs{t_2-t_1}^\ga}\leq{C_\ga}\at{1+t_1}^{-\ga-1/2}
\end{align*}
holds for all $j\in\Zset$ and all $t_1>0$. Moreover, 
there exists a constant $C$ such that
  \begin{align*}
    \sup_{j_1,j_2\in\Zset}
    \frac{\babs{g_{j_2}\at{t}-g_{j_1}\at{t}}}{\abs{j_2-j_1}^{1/2}}
    \leq
    {C}\at{1+t}^{-3/4}
  \end{align*}
  holds for all $t>0$.
\end{lemma}
\begin{proof}
Let $j\in\Zset$ and $0<t_1<t_2$ be fixed. Thanks to \eqref{lem:discr-heat-kernel.Eqn2} we estimate
\begin{align*}
\babs{g_j\at{t_2}-g_j\at{t_1}}\leq\int_{t_1}^{t_2}\babs{\dot{g}_j\at{t}}\dint{t}
\leq C \int_{t_1}^{t_2}\at{1+t}^{-3/2}\dint{t}=C\babs{\at{1+t_1}^{-1/2}-\at{1+t_2}^{-1/2}},
\end{align*}
and writing $1+t_2=s\at{1+t_1}$ with $s\geq1$ we get
\begin{align*}
\frac{\abs{g_j\at{t_2}-g_j\at{t_1}}}{\abs{t_2-t_1}^\ga}\leq \frac{C f_\ga\at{s}}{\at{1+t_1}^{\ga+1/2}},\qquad
f_\ga\at{s}:=\frac{1}{s^{1/2}}\frac{s^{1/2}-1}{\at{s-1}^\ga}.
\end{align*}
We readily check that the function $f_\ga$ is bounded on $[1,\infty)$,
so the first claim follows by taking the supremum over $s\geq1$.
\par
Now let $t>0$ and $j_1,j_2\in\Zset$ with $j_2>j_1$ be arbitrary. By
H\"older's inequality for series we then find
\begin{align*}
\babs{g_{j_2}\at{t}-g_{j_1}\at{t}}^2&=\at{\sum_{j=j_1}^{j_2-1}\babs{\nabla_+ g_j\at{t}}}^2
\leq \at{\sum_{j=j_1}^{j_2-1}\abs{\nabla_+ g_j\at{t}}^2}\at{\sum_{j=j_1}^{j_2-1}1}
\\&\leq \at{\sum_{j\in\Zset}\abs{\nabla_+ g_j\at{t}}^2}\bat{j_2-j_1},
\end{align*} 
and the second assertion follows from \eqref{lem:discr-heat-kernel.Eqn4}.
\end{proof}


\section*{Acknowledgments}

The authors thank Wolfgang Dreyer for pointing them to the problem
discussed in the paper. Part of the work has been done at the
Harcourt Arms, Cranham Terrace, Jericho, whose hospitality is
gratefully acknowledged.


\bibliographystyle{alpha}
\bibliography{paper}

\end{document}
